\documentclass[preprint]{imsart}

\RequirePackage{amsthm,amsmath,amsfonts,amssymb}
\RequirePackage[authoryear]{natbib}
\RequirePackage[colorlinks,citecolor=blue,urlcolor=blue]{hyperref}
\RequirePackage{graphicx}
\RequirePackage{grffile}
\RequirePackage{bm}
\RequirePackage{color}
\RequirePackage{calrsfs}
\RequirePackage{enumerate}
\RequirePackage{booktabs,rotating}

\startlocaldefs

\numberwithin{equation}{section}
\theoremstyle{plain}
\newtheorem{prop}{Proposition}[section]

\newtheorem{thm}[prop]{Theorem}
\newtheorem{cor}[prop]{Corollary}
\newtheorem{lem}[prop]{Lemma}
\newtheorem{cond}[prop]{Condition}
\newtheorem{algo}[prop]{Algorithm}
\theoremstyle{remark}
\newtheorem{remark}[prop]{Remark}

\newcommand{\eps}{\varepsilon}
\newcommand{\N}{\mathbb{N}}
\newcommand{\R}{\mathbb{R}}
\newcommand{\Z}{\mathbb{Z}}

\newcommand{\Bb}{\mathbb{B}}
\newcommand{\Cb}{\mathbb{C}}

\newcommand{\Xb}{\mathbb{X}}
\newcommand{\Zb}{\mathbb{Z}}
\newcommand{\Db}{\mathbb{D}}

\newcommand{\Fc}{\mathcal{F}}
\newcommand{\Bc}{\mathcal{B}}
\newcommand{\Kc}{\mathcal{K}}

\newcommand{\Cc}{\mathcal{C}}
\newcommand{\Ac}{\mathcal{A}}
\newcommand{\Vc}{\bm{\mathcal{V}}}
\newcommand{\Wc}{\mathcal{W}}
\newcommand{\Xc}{\bm{\mathcal{X}}}
\newcommand{\dotu}[1]{\dot{\underline{#1}}}

\newcommand{\Dirac}{\mathrm{Dirac}}
\newcommand{\Bern}{\mathrm{Bern}}
\newcommand{\Bin}{\mathrm{Bin}}
\newcommand{\BetaB}{\mathrm{BetaB4}}
\newcommand{\dd}{\mathrm{d}}

\newcommand{\Ex}{\mathbb{E}}
\newcommand{\Var}{\mathrm{Var}}

\newcommand{\1}{\mathbf{1}}
\newcommand{\ip}[1]{\lfloor #1 \rfloor}
\newcommand{\up}[1]{\lceil #1 \rceil}
\renewcommand{\Pr}{\mathbb{P}}

\newcommand{\p}{\overset{\sss \Pr}{\to}}

\newcommand{\cwc}{\overset{\Pr}{\underset{\Wc}{\leadsto}}}

\newcommand{\pobs}[1]{\hat{\bm #1}}
\newcommand{\sss}[1]{\scriptscriptstyle{#1}}

\endlocaldefs

\allowdisplaybreaks

\begin{document}

\begin{frontmatter}

\title{Resampling techniques for a class of smooth, possibly data-adaptive empirical copulas}
\runtitle{Resampling techniques for smooth, possibly data-adaptive empirical copulas}

\begin{aug}
  \author[I]{\fnms{Ivan} \snm{Kojadinovic}\ead[label=eI]{ivan.kojadinovic@univ-pau.fr}}
  \and
  \author[I,M]{\fnms{Bingqing} \snm{Yi}\ead[label=eA]{bingqing.yi@univ-pau.fr}}
  \address[I]{CNRS / Universit\'e de Pau et des Pays de l'Adour / E2S UPPA \\ Laboratoire de math\'ematiques et applications IPRA, UMR 5142 \\ B.P. 1155, 64013 Pau Cedex, France}
  \address[M]{School of Mathematics \& Statistics \\ The University of Melbourne \\ Parkville, VIC 3010, Australia}
 
\end{aug}

\begin{abstract}
  We investigate the validity of two resampling techniques when carrying out inference on the underlying unknown copula using a recently proposed class of smooth, possibly data-adaptive nonparametric estimators that contains empirical Bernstein copulas (and thus the empirical beta copula). Following \cite{KirSegTsu21}, the first resampling technique is based on drawing samples from the smooth estimator and can only can be used in the case of independent observations. The second technique is a smooth extension of the so-called sequential dependent multiplier bootstrap and can thus be used in a time series setting and, possibly, for change-point analysis. The two studied resampling schemes are applied to confidence interval construction and the offline detection of changes in the cross-sectional dependence of multivariate time series, respectively. Monte Carlo experiments confirm the possible advantages of such smooth inference procedures over their non-smooth counterparts. A by-product of this work is the study of the weak consistency and finite-sample performance of two classes of smooth estimators of the first-order partial derivatives of a copula which can have applications in mean and quantile regression. 
\end{abstract}



\tableofcontents

\end{frontmatter}


\section{Introduction}

Let $\Xc_{1:n} = (\bm X_1,\dots,\bm X_n)$ be a stretch from a $d$-dimensional stationary time series $(\bm X_i)_{i \in \Z}$ of continuous random vectors. From a well-known theorem due to \citet{Skl59}, the multivariate distribution function (d.f.) $F$ of each $\bm X_i$ can be expressed as
\begin{equation}
\label{eq:sklar}
F(\bm x) = C\{F_1(x_1),\dots,F_d(x_d)\}, \qquad \bm x \in \R^d,
\end{equation}
in terms of a unique \emph{copula} $C$ and the univariate margins $F_1,\dots,F_d$ of $F$. Representation~\eqref{eq:sklar} is at root of many applications in probability, statistics and related fields \cite[see, e.g.,][and the references therein]{HofKojMaeYan18} because it suggests that $F$ can be modeled in two separate steps: the first (resp.\ second) step consists of estimating the univariate margins $F_1, \dots, F_d$ (resp.\ the copula $C$). This work is only concerned with the estimation of the copula.

Statistical inference on the unknown copula $C$ frequently involves the use of a nonparametric estimator of $C$. The best-known one is the \emph{empirical copula} \citep{Rus76,Deh79} which we shall define as the empirical d.f.\ of the multivariate ranks obtained from $\Xc_{1:n}$ scaled by $1/n$. Note that the latter function is piecewise constant and cannot therefore be a genuine copula. A promising smooth nonparametric estimator of $C$ that is a genuine copula when there are no ties in the components samples of $\Xc_{1:n}$ and that displays substantially better small-sample performance than the empirical copula is the \emph{empirical beta copula}. This estimator  was proposed by~\cite{SegSibTsu17} and is a particular case of the \emph{empirical Bernstein copula} studied by \cite{SanSat04} and \cite{JanSwaVer12} when all the underlying Bernstein polynomials have degree~$n$. Building upon the work of \cite{SegSibTsu17}, \cite{KojYi22} recently studied data-adaptive generalizations of the empirical beta copula that can perform even better in small samples.

Whatever nonparametric estimator of the unknown copula $C$ in~\eqref{eq:sklar} is used in inference procedures, it is almost always necessary to rely on resampling techniques to compute corresponding confidence intervals or p-values. To approximate the ``sampling distribution'' of the classical empirical copula, a frequently used approach in the literature is the so-called \emph{multiplier bootstrap} \citep[see, e.g.,][]{Sca05,RemSca09}. When the random vectors in $\Xc_{1:n}$ are independent and identically distributed (i.i.d.), \cite{BucDet10} found the latter resampling scheme to have better finite-sample properties than approaches consisting of adapting the empirical (multinomial) bootstrap. The multiplier bootstrap was extended to the time series and sequential settings in \cite{BucKoj16} and \cite{BucKojRohSeg14}.

One of the advantages of the empirical beta copula is that it is particularly easy to draw samples from it. The resulting \emph{smooth bootstrap} that can be used to approximate the ``sampling distribution'' of the empirical beta copula was recently studied both theoretically and empirically in \cite{KirSegTsu21}. The Monte Carlo experiments reported therein reveal that it is a competitive alternative to the multiplier bootstrap while being substantially simpler to implement. One practical inconvenience however is that the aforementioned smooth bootstrap cannot be directly extended to the time series setting.

The first aim of this work is to obtain, in the i.i.d.\ case, a smooth bootstrap \emph{à la} \cite{KirSegTsu21} for the smooth, possibly data-adaptive, nonparametric estimators of the copula investigated in \cite{KojYi22}. The second aim is to propose smooth versions of the dependent multiplier bootstrap that can be used to approximate the ``sampling distribution''  of the aforementioned estimators in a time series setting. Intuitively, one could expect that the resulting smooth inference procedures will perform better than corresponding non-smooth procedures in particular when the amount of data is low. Indeed, as already mentioned, it is when $n$ is small that smooth copula estimators can substantially outperform rough estimators such as the classical empirical copula; see for instance the finite-sample experiments reported in \cite{SegSibTsu17}, \cite{KirSegTsu21} or \cite{KojYi22}. Another situation where one could expect that the use of smooth estimators can be advantageous is when carrying out change-point detection. Indeed, statistics for change-point detection often involve the comparison of estimators computed from small subsets of observations. It is to be able to cover this application area that many of the theoretical investigations carried out in this work are of a sequential nature. 

A by-product of this work is the study of the weak consistency and finite-sample performance of two classes of smooth estimators of the first-order partial derivatives of the unknown copula $C$ in~\eqref{eq:sklar} as these are needed to carry out the dependent multiplier bootstrap. As explained for instance in \cite{JanSwaVer16}, such estimators have applications in mean and quantile regression as they lead to estimators of the conditional distribution function. From a practical perspective, our investigations lead to the proposal of a smooth data-adaptive estimator of the first-order partial derivatives of $C$ that substantially outperforms, among others, the Bernstein estimator considered in \cite{JanSwaVer16}.

This paper is organized as follows. In the second section, we recall the definition of the broad class of smooth, possibly data adaptive, empirical copulas studied in \cite{KojYi22} and the asymptotics of related sequential empirical processes. The third section is concerned with an extension of the smooth bootstrap of \cite{KirSegTsu21} that can be used to approximate the ``sampling distribution'' of the aforementioned smooth estimators in the i.i.d.\ case. After investigating its asymptotic validity, results of finite-sample experiments comparing smooth bootstraps based on the empirical beta copula and on its data-adaptive extension suggested in \cite{KojYi22} are reported. In Section~\ref{sec:mult}, to be able to cover the time series setting, we propose natural smooth extensions of the sequential dependent multiplier bootstrap. After providing asymptotic validity results, we compare the finite-sample performance of various versions of the multiplier bootstrap and consider an application to the offline detection of changes in the cross-sectional dependence of multivariate time series. The latter confirms the possible advantages of smooth inference procedures over their non-smooth counterparts. The fifth section is devoted to the study of two classes of smooth estimators of the first-order partial derivatives of $C$: their weak consistency is investigated and the finite-sample performance of selected estimators is studied. 

Unless stated otherwise, all convergences in the paper are as $n \to \infty$. Also, in the sequel, the arrow~`$\leadsto$' denotes weak convergence in the sense of Definition~1.3.3 in \cite{vanWel96} and, given a set $T$, $\ell^\infty(T)$ (resp.\ $\Cc(T)$) represents the space of all bounded (resp.\ continuous) real-valued functions on $T$ equipped with the uniform metric.

All the numerical experiments presented in the work were carried out using the \textsf{R} statistical environment \citep{Rsystem} as well as its packages \texttt{copula} \citep{copula} and \texttt{extraDistr} \citep{extraDistr}.

\section{Smooth, possibly data-adaptive, empirical copulas and their asymptotics}

In this section, we start by defining the broad class of smooth, possibly data adaptive, empirical copulas studied in \cite{KojYi22}. We then recall the asymptotics of related sequential empirical processes established in the same reference.

\subsection{Smooth, possibly data-adaptive, nonparametric copula estimators}
\label{sec:ec}

Because the results to be stated in the next section are of a sequential nature, all the quantities hereafter are defined for a substretch $\Xc_{k:l} = (\bm X_k,\dots,\bm X_l)$, $1 \leq k \leq l \leq n$, of the available data $\Xc_{1:n} = (\bm X_1,\dots,\bm X_n)$.

For any $j \in \{1,\dots,d\}$, let $F_{k:l,j}$ be the empirical d.f.\ computed from the $j$th component subsample $X_{kj},\dots,X_{lj}$ of $\Xc_{k:l}$. Then, $R_{ij}^{k:l} = (l-k+1) F_{k:l,j}(X_{ij}) = \sum_{t=k}^l \1( X_{tj} \le X_{ij})$ is the (maximal) rank of $X_{ij}$ among $X_{kj},\dots,X_{lj}$. Furthermore, let
$$
\bm R^{k:l}_i =  \left(R_{i1}^{k:l}, \dots, R_{id}^{k:l} \right) \qquad \text{and} \qquad \pobs{U}^{k:l}_i = \frac{\bm R^{k:l}_i}{l-k+1}, \qquad i \in \{k,\dots,l\},
$$
be the multivariate ranks (resp.\ multivariate scaled ranks) obtained from $\Xc_{k:l}$. 
Following \cite{Rus76}, the empirical copula $C_{k:l}$ of $\Xc_{k:l}$ is then defined, for any $\bm u = (u_1,\dots,u_d) \in [0,1]^d$, by
\begin{equation}
\label{eq:C:kl}
C_{k:l}(\bm u) = \frac{1}{l-k+1} \sum_{i=k}^l  \prod_{j=1}^d \1\left( \frac{R_{ij}^{k:l}}{l-k+1} \leq u_j \right)
= \frac{1}{l-k+1} \sum_{i=k}^l  \1(\pobs{U}^{k:l}_i \leq \bm u),
\end{equation}
where inequalities between vectors are to be understood componentwise.

As we continue, following \cite{KojYi22}, for any $m \in \N$, $\bm x \in (\R^d)^m$ and $\bm u \in [0,1]^d$, $\nu_{\bm u}^{\bm x}$ is the law of a $[0,1]^d$-valued mean $\bm u$ random vector $\bm W_{\bm u}^{\bm x}$. Its components are denoted by $W_{1,u_1}^{\bm x}, \dots, W_{d,u_d}^{\bm x}$ to indicate that the $j$th component depends on $u_j$ but not on $u_1,\dots,u_{j-1},u_{j+1},\dots,u_d$. Let $p \geq d$ be a fixed integer and let $\bm U$ be a $p$-dimensional random vector whose components are independent and standard uniform. The following assumption was considered in \cite{KojYi22} and is likely to be non-restrictive as discussed in Remark~3 therein.

\begin{cond}[Construction of smoothing random vectors]
  For any $m \in \N$, $\bm x \in (\R^d)^m$ and $\bm u \in [0,1]^d$, there exists a function $\Wc_{\bm u}^{\bm x}: [0,1]^p \to [0,1]^d$ such that $\bm W_{\bm u}^{\bm x} = \Wc_{\bm u}^{\bm x}(\bm U)$. 
\end{cond}

To be able to define, for any $n \in \N$, $\Xc_{1:n}$ and, for any $m \leq n$, the random vectors $\bm W_{\bm u}^{\bm x}$,  $\bm x \in (\R^d)^m$, $\bm u \in [0,1]^d$, on the same probability space $(\Omega, \Ac, \Pr)$, we assume a product structure, that is, $\Omega=\Omega_0 \times \Omega_1 \times \dots$ with probability measure $\Pr=\Pr_0 \otimes \Pr_1 \otimes \dots$, where $\Pr_i$ denotes the probability measure on~$\Omega_i$, such that, for any $\omega \in \Omega$, $\bm \Xc_{1:n}(\omega)$ only depends on the first coordinate of $\omega$, $\bm U(\omega)$ only depends on the second coordinate of~$\omega$ and potential ``bootstrap weights'' (to be introduced in Sections~\ref{sec:boot} and~\ref{sec:mult}) only depend on one of the remaining coordinates of $\omega$, implying in particular that $\Xc_{1:n}$, $\bm U$ and potential bootstrap weights are independent. A broad class of smooth versions of $C_{k:l}$ in~\eqref{eq:C:kl}, with possibly data-adaptive smoothing, is then given by
  \begin{equation}
  \label{eq:C:kl:nu}
  C_{k:l}^\nu(\bm u) = \int_{[0,1]^d} C_{k:l}(\bm w) \dd \nu_{\bm u}^{\sss \Xc_{k:l}}(\bm w), 
  \qquad \bm u \in [0,1]^d.
\end{equation}
Intuitively, for a given $\bm u \in [0,1]^d$, $C_{k:l}^\nu(\bm u)$ can be thought of as a ``weighted average'' of $C_{k:l}(\bm w)$ for $\bm w$ ``in a neighborhood of $\bm u$'' according to the smoothing distribution $\nu_{\bm u}^{\sss \Xc_{k:l}}$ (that may depend on the observations $\Xc_{k:l}$). Note that, if $k > l$, we adopt the convention that $C_{k:l} =  C_{k:l}^\nu = 0$ and that, for any $\bm u \in [0,1]^d$, $\nu_{\bm u}^{\sss \Xc_{k:l}}$ is the Dirac measure at $\bm u$.

\begin{remark}
  \label{rem:beta}
  Given $m \in \N$ and $\bm u \in [0,1]^d$, let $\mu_{m,\bm u}$ be the law of the $d$-dimensional random vector $(S_{m,1,u_1}/m,\dots, S_{m,d,u_d}/m)$ such that the random variables $S_{m,1,u_1}, \dots, S_{m,d,u_d}$ are independent and, for each $j \in \{1,\dots d\}$, $S_{m,j,u_j}$ is Binomial$(m,u_j)$. From Section~3 of \cite{SegSibTsu17}, the empirical Bernstein copula of $\Xc_{k:l}$ whose Bernstein polynomial degrees are all equal to $m$ is then given by
\begin{equation}
  \label{eq:C:Bern:kl:m}
  C_{k:l,m}^{\sss \Bern}(\bm u) =  \int_{[0,1]^d} C_{k:l}(\bm w) \dd \mu_{m,\bm u}(\bm w), \qquad \bm u \in [0,1]^d.
\end{equation}
The latter is clearly a special case of $C_{k:l}^\nu$ in~\eqref{eq:C:kl:nu}. If, additionally, $m = l-k+1$, that is, if the smoothing distributions satisfy $\nu_{\bm u}^{\sss \Xc_{k:l}} = \mu_{l-k+1,\bm u}$, $\bm u \in [0,1]^d$, $C_{k:l}^\nu$ in~\eqref{eq:C:kl:nu} or, equivalently, $C_{k:l,m}^{\sss \Bern}$ in~\eqref{eq:C:Bern:kl:m}, corresponds to the empirical beta copula of $\Xc_{k:l}$ studied in \cite{SegSibTsu17}.
\end{remark}

For any $m \in \N$, $\bm x \in (\R^d)^m$, $\bm r \in [0,m]^d$ and $\bm u \in [0,1]^d$, let
\begin{equation}
  \label{eq:K}
  \Kc_{\bm r}^{\bm x}(\bm u) =  \int_{[0,1]^d} \1(\bm r / m \leq \bm w) \dd\nu_{\bm u}^{\bm x}(\bm w) = \Ex \left \{ \1(\bm r / m \leq \bm W_{\bm u}^{\bm x}) \right\}.
\end{equation}
By linearity of the integral, 
we can then express $C_{k:l}^\nu$ in~\eqref{eq:C:kl:nu} as
\begin{equation}
  \label{eq:C:kl:K}
C_{k:l}^\nu(\bm u) = \frac{1}{l-k+1} \sum_{i=k}^l \Kc_{\sss{\bm R^{k:l}_i}}^{\sss \Xc_{k:l}}(\bm u), \qquad \bm u \in [0,1]^d.
\end{equation}

Since copulas have standard uniform margins, it is particularly meaningful to focus on estimators of the form~\eqref{eq:C:kl:K} that have standard uniform margins. As verified in Section 3.1 of \cite{KojYi22}, the following two assumptions imply the latter.

\begin{cond}[No ties]
\label{cond:no:ties}
With probability~1, there are no ties in each of the component samples $X_{1j}, \dots, X_{nj}$, $j \in \{1,\dots,d\}$, of $\Xc_{1:n}$.
\end{cond}

\begin{cond}[Condition for uniform margins]
  \label{cond:unif:marg}
For any $m \in \N$, $\bm x \in (\R^d)^m$, $\bm u \in [0,1]^d$ and $j \in \{1,\dots,d\}$, $W_{j,u_j}^{\bm x}$ takes its values in the set $\{0,1/m,\dots,(m-1)/m,1\}$.
\end{cond}

Under Condition~\ref{cond:unif:marg}, from Section~3.2 of \cite{KojYi22}, for any $m \in \N$, $\bm x \in (\R^d)^m$, $\bm r \in [1,m]^d$ and $\bm u \in [0,1]^d$, $\Kc_{\bm r}^{\bm x}(\bm u)$ in \eqref{eq:K} can be written as
\begin{equation}
  \label{eq:K:dis}
  \Kc_{\bm r}^{\bm x}(\bm u) = \bar \Cc_{\bm u}^{\bm x} \left[ \bar \Fc_{1,u_1}^{\bm x}\{ (r_1 - 1) / m\}, \dots, \bar \Fc_{d,u_d}^{\bm x} \{ (r_d - 1) / m \} \right],
\end{equation}
where $\bar \Cc_{\bm u}^{\bm x}$ (resp.\ $\bar \Fc_{1, u_1}^{\bm x}, \dots, \bar \Fc_{d, u_d}^{\bm x}$) is a survival copula (resp.\ are the marginal survival functions) of the random vector $\bm W_{\bm u}^{\bm x}$. Upon additionally assuming the following two conditions considered in Section 3.2 of \cite{KojYi22}, estimators of the form~\eqref{eq:C:kl:K} can be shown to be genuine copulas.

\begin{cond}[Condition on the smoothing survival margins]
  \label{cond:smooth:surv:marg}
  For any $m \in \N$, $\bm x \in (\R^d)^m$, $j \in \{1,\dots,d\}$ and $w \in [0,1)$, the function $t \mapsto \bar \Fc_{j,t}^{\bm x}(w)$ is right-continuous and increasing on $[0,1]$.
\end{cond}

\begin{cond}[Condition on the smoothing survival copulas]
  \label{cond:smooth:cop}
  For any $m \in \N$, $\bm x \in (\R^d)^m$ and $\bm u \in [0,1]^d$, the copulas $\bar \Cc_{\bm u}^{\bm x}$ in~\eqref{eq:K:dis} do not depend on $\bm u$, that is, $\bar \Cc_{\bm u}^{\bm x} = \bar \Cc^{\bm x}$.
\end{cond}

The following result was then proven in \cite{KojYi22}; see Proposition~11 and Corollary~12 therein.

\begin{prop}[$C_{k:l}^\nu$ is a genuine copula]
  \label{prop:genuine:copula}
  Assume that Conditions~\ref{cond:no:ties}, \ref{cond:unif:marg},~\ref{cond:smooth:surv:marg} and~\ref{cond:smooth:cop} hold. Then, the smooth empirical copula $C_{k:l}^\nu$ in~\eqref{eq:C:kl:nu} or in~\eqref{eq:C:kl:K} can be expressed, for any $\bm u \in [0,1]^d$, as
\begin{equation}
  \label{eq:C:kl:dis}
  C_{k:l}^\nu(\bm u) = \frac{1}{l-k+1} \sum_{i=k}^l \bar \Cc^{\sss \Xc_{k:l}} \left\{ \bar \Fc_{1,u_1}^{\sss \Xc_{k:l}} \left( \frac{R_{i1}^{k:l} - 1}{l-k+1} \right), \dots, \bar \Fc_{d,u_d}^{\sss \Xc_{k:l}} \left( \frac{R_{id}^{k:l} - 1}{l-k+1} \right) \right\}, 
\end{equation}
and is a genuine copula.
\end{prop}

From Remark~\ref{rem:beta} above, we can infer that the empirical beta copula of $\Xc_{k:l}$ studied in \cite{SegSibTsu17} is of the form~\eqref{eq:C:kl:dis} with $\bar \Cc^{\sss \Xc_{k:l}}$ the independence copula and, for any $j \in \{1,\dots,d\}$ and $u \in [0,1]$ $\bar \Fc_{j,u}^{\sss \Xc_{k:l}}$ the survival function of a scaled (by $1/(l-k+1)$) Binomial($l-k+1, u$) random variable. For that reason, the latter will be denoted as $C_{k:l}^{\sss \Bin}$ as we continue. As a possible improvement of the empirical beta copula $C_{k:l}^{\sss \Bin}$ of $\Xc_{k:l}$, \cite{KojYi22} suggested to consider a smooth data-adaptive empirical copula of the form~\eqref{eq:C:kl:dis} with $\bar \Cc^{\sss \Xc_{k:l}}$ the empirical beta copula $C_{k:l}^{\sss \Bin}$ and, for any $j \in \{1,\dots,d\}$ and $u \in [0,1]$, $\bar \Fc_{j,u}^{\sss \Xc_{k:l}}$ the survival function of a scaled (by $1/(l-k+1)$) Beta-Binomial$(m, \alpha, \beta)$  random variable, where $m = l- k + 1$, $\alpha = u(m - \rho)/(\rho - 1)$, $\beta = (1-u)(m - \rho)/(\rho - 1)$ and $\rho=4$. The resulting data-adaptive estimator, denoted by $C_{k:l}^{\sss \BetaB}$ as we continue, was found to outperform the empirical beta copula $C_{k:l}^{\sss \Bin}$ in terms of integrated mean squared error in all the bivariate and trivariate experiments considered in \cite{KojYi22}.

\subsection{Asymptotics of related sequential processes}
\label{sec:asym}

We can now define the sequential empirical processes corresponding to the empirical copula in~\eqref{eq:C:kl} and to its smooth generalizations in~\eqref{eq:C:kl:nu}. Let $\Lambda = \{(s,t) \in [0,1]^2 : s \leq t \}$ and let $\lambda_n(s,t) = (\ip{nt} - \ip{ns})/n$, $(s, t) \in \Lambda$. The corresponding two-sided sequential empirical copula processes are given, for any $(s,t) \in \Lambda$ and $\bm u \in [0,1]^d$,  by
\begin{align}
  \label{eq:Cb:n}
  \Cb_n(s,t,\bm u) &= \sqrt{n}\lambda_n(s,t)\{C_{\ip{ns} +1 :\ip{nt}}(\bm u)-C(\bm u)\}, \\
  \label{eq:Cb:n:nu}
\Cb_n^\nu(s,t,\bm u) &=\sqrt{n}\lambda_n(s,t)\{C_{\ip{ns} +1 :\ip{nt} }^\nu(\bm u)-C(\bm u)\},
\end{align}
where $C_{\ip{ns} +1 :\ip{nt}}$ and $C_{\ip{ns} +1 :\ip{nt}}^\nu$ are generically defined in~\eqref{eq:C:kl} and~\eqref{eq:C:kl:nu}, respectively. The asymptotics of $\Cb_n$ were established in \cite{BucKoj16}, while the asymptotics of $\Cb_n^\nu$ (which we recall in Theorem~\ref{thm:Cb:n:nu} hereafter) were investigated in \cite{KojYi22} by generalizing the arguments used in \cite{SegSibTsu17}.

The following conditions were considered in \cite{KojYi22}.

\begin{cond}[Smooth partial derivatives]
\label{cond:pd}
For any $j \in \{1,\dots,d\}$, the partial derivative $\dot C_j = \partial C/\partial u_j$ exists and is continuous on the set $V_j = \{ \bm u \in [0, 1]^d : u_j \in (0,1) \}$.
\end{cond}

\begin{cond}[Variance condition]
  \label{cond:var:W}
  There exists a constant $\kappa > 0$ such that, for any $n \in \N$, $\bm x \in (\R^d)^n$, $\bm u \in [0,1]^d$ and $j \in \{1,\dots,d\}$, $\Var( W_{j,u_j}^{\bm x}) \leq \kappa u_j(1-u_j) / n$.
\end{cond}

The first condition was initially considered in \cite{Seg12} and can be considered non-restricted as explained in the latter reference.  In the rest of the paper, for any $j \in \{1,\dots,d\}$, $\dot C_j$ is arbitrarily defined to be zero on the set $\{ \bm u \in [0, 1]^d : u_j \in \{0,1\} \}$, which implies that, under Condition~\ref{cond:pd}, $\dot C_j$ is defined on the whole of $[0, 1]^d$. The second condition imposes constraints on the spread of the smoothing distributions involved in the definition of the smooth, possibly data-adaptive, empirical copulas.

\begin{thm}[Asymptotics of $\Cb_n^\nu$]
\label{thm:Cb:n:nu}
Assume that Conditions~\ref{cond:pd} and~\ref{cond:var:W} hold, and that $\Cb_n \leadsto \Cb_C$ in $ \ell^\infty(\Lambda \times [0,1]^d)$, where the trajectories of the limiting process $\Cb_C$ are continuous almost surely. Then,
$$
\sup_{\substack{(s,t) \in \Lambda \\ \bm u \in [0,1]^d}} |\Cb_n^\nu(s,t,\bm u) - \Cb_n(s,t,\bm u) | = o_\Pr(1).
$$
Consequently,  $\Cb_n^\nu \leadsto \Cb_C$ in $\ell^\infty(\Lambda \times[0,1]^d)$.
\end{thm}

Hence, the smooth sequential empirical copula process $\Cb_n^\nu$ in~\eqref{eq:Cb:n:nu} and the classical sequential empirical copula process $\Cb_n$ in~\eqref{eq:Cb:n} are asymptotically equivalent when the latter converges weakly to a limiting process whose trajectories are continuous almost surely. As discussed in Section~3 of \cite{BucKoj16}, for such a convergence to hold, it suffices that the corresponding ``uniform multivariate sequential empirical process'' converges weakly to a limiting process whose trajectories are continuous almost surely. Specifically, let $\bm U_1,\dots,\bm U_n$ be the unobservable sample obtained from $\Xc_{1:n} = (\bm X_1, \dots, \bm X_n)$ by the probability integral transformations $U_{ij} = F_j(X_{ij})$, $i \in \{1,\dots,n\}$, $j \in \{1,\dots,d\}$, and let
\begin{equation}
\label{eq:Bb:n}
  \Bb_n(s, t, \bm u) = \frac{1}{\sqrt{n}} \sum_{i=\ip{ns}+1}^{\ip{nt}} \{\1(\bm U_i \leq \bm u) - C(\bm u) \}, \qquad (s, t,\bm u) \in \Lambda \times [0, 1]^d,
\end{equation}
with the convention that $\Bb_n(s, t, \cdot) = 0$ if $\ip{nt} - \ip{ns} = 0$. The aforementioned sufficient condition can then be stated as follows.

\begin{cond}[Weak convergence of $\Bb_n(0, \cdot, \cdot)$]
\label{cond:Bn}
The sequential empirical process $\Bb_n(0, \cdot, \cdot)$ converges weakly in $\ell^\infty([0,1]^{d+1})$ to a tight centered Gaussian process $\Zb_C$ concentrated on
\begin{multline*}
\{ f \in \Cc([0,1]^{d+1}) : f(s,\bm u) = 0 \text{ if one of the components of $(s,\bm u)$ is 0,} \\ \text{ and } f(s,1,\dots,1) = 0 \text{ for all } s \in (0,1] \}.
\end{multline*}
\end{cond}

Under Condition~\ref{cond:Bn} (which holds for instance when $(\bm X_i)_{i \in \Z}$ is \emph{strongly mixing}; see, e.g., \cite{Buc15} as well as forthcoming Section~\ref{sec:intuition}), it immediately follows from the continuous mapping theorem that $\Bb_n \leadsto \Bb_C$ in $\ell^\infty(\Lambda \times [0, 1]^d)$, where
\begin{equation}
\label{eq:Bb:C}
  \Bb_C(s, t,\bm u) = \Z_C(t,\bm u) - \Z_C(s,\bm u),   \qquad (s, t,\bm u) \in \Lambda \times [0,1]^d.
\end{equation}

For any $j \in \{1,\dots,d\}$ and any $\bm u \in [0, 1]^d$, let $\bm u^{(j)}$ be the vector of $[0, 1]^d$ defined by $u^{(j)}_i = u_j$ if $i = j$ and 1 otherwise. The following result is then an immediate consequence of Theorem~3.4 in \cite{BucKoj16} and Proposition~3.3 of \cite{BucKojRohSeg14}.

\begin{thm}[Asymptotics of $\Cb_n$]
\label{thm:wc:Cb:n}
Under Conditions~\ref{cond:pd} and~\ref{cond:Bn},
\begin{equation*}
  \sup_{(s, t,\bm u)\in \Lambda \times [0, 1]^d} \left| \Cb_n(s, t, \bm u) - \tilde \Cb_n(s, t, \bm u) \right| = o_\Pr(1),
\end{equation*}
where
\begin{equation}
\label{eq:Cb:n:tilde}
\tilde \Cb_n(s, t, \bm u) = \Bb_n(s, t, \bm u) - \sum_{j=1}^d \dot C_j(\bm u) \, \Bb_n(s, t, \bm u^{(j)}), \qquad (s, t, \bm u) \in \Lambda \times [0, 1]^d,
\end{equation}
and $\Bb_n$ is defined in~\eqref{eq:Bb:n}. Consequently, $\Cb_n \leadsto \Cb_C$ in $\ell^\infty(\Lambda \times [0, 1]^d)$, where
\begin{equation}
\label{eq:Cb:C}
\Cb_C(s, t, \bm u)  = \Bb_C(s, t, \bm u) - \sum_{j=1}^d \dot C_j(\bm u) \, \Bb_C(s, t, \bm u^{(j)}), \qquad (s, t, \bm u) \in \Lambda \times [0, 1]^d,
\end{equation}
and $\Bb_C$ is defined in~\eqref{eq:Bb:C}.
\end{thm}

We end this section with the statement of a corollary of Theorems~\ref{thm:Cb:n:nu} and~\ref{thm:wc:Cb:n}. Having~\eqref{eq:K} in mind, two natural smooth extensions of the unobservable empirical process $\Bb_n$ in~\eqref{eq:Bb:n} can be defined, for any $(s, t,\bm u) \in \Lambda \times [0, 1]^d$, by
\begin{align}
  \label{eq:tilde:Bb:n:nu}
  \tilde \Bb_n^\nu(s, t, \bm u) &= \frac{1}{\sqrt{n}} \sum_{i=\ip{ns}+1}^{\ip{nt}} \left\{ \int_{[0,1]^d} \1(\bm U_i \leq \bm{w}) \dd \nu_{\bm u}^{\sss \Xc_{\ip{ns}+1:\ip{nt}}}(\bm w) - C(\bm u) \right\}, \\
  \label{eq:bar:Bb:n:nu}
  \bar \Bb_n^\nu(s, t, \bm u) &= \frac{1}{\sqrt{n}} \sum_{i=\ip{ns}+1}^{\ip{nt}} \left\{ \int_{[0,1]^d} \1(\bm U_i \leq \bm{w}) \dd \nu_{\bm u}^{\sss \Xc_{1:n}}(\bm w) - C(\bm u) \right\}.
\end{align}
Combining Theorem~\ref{thm:wc:Cb:n} with key intermediate results used in \cite{KojYi22} for proving Theorem~\ref{thm:Cb:n:nu} stated above, we obtain the following asymptotic representations for the smooth sequential empirical process $\Cb_n^\nu$ in~\eqref{eq:Cb:n:nu}. The proof of this result is given in Appendix~\ref{proof:cor:wc:Cb:n:nu}.

\begin{cor}[Asymptotic representations of $\Cb_n^\nu$]
 \label{cor:wc:Cb:n:nu}
Under Conditions~\ref{cond:pd},~\ref{cond:var:W} and~\ref{cond:Bn},
\begin{align*}
   \sup_{(s, t,\bm u)\in \Lambda \times [0, 1]^d} \left| \Cb_n^\nu(s, t, \bm u) - \tilde \Cb_n^\nu(s, t, \bm u) \right| &= o_\Pr(1), \\
   \sup_{(s, t,\bm u)\in \Lambda \times [0, 1]^d} \left| \Cb_n^\nu(s, t, \bm u) - \bar \Cb_n^\nu(s, t, \bm u) \right| &= o_\Pr(1),
\end{align*}
 where, for any $(s, t,\bm u) \in \Lambda \times [0, 1]^d$,
\begin{align*}
   \tilde \Cb_n^\nu(s, t, \bm u) &= \tilde \Bb_n^\nu(s, t, \bm u) - \sum_{j=1}^d \dot C_j(\bm u) \, \tilde \Bb_n^\nu(s, t, \bm u^{(j)}), \\
 \bar \Cb_n^\nu(s, t, \bm u) &= \bar \Bb_n^\nu(s, t, \bm u) - \sum_{j=1}^d \dot C_j(\bm u) \, \bar \Bb_n^\nu(s, t, \bm u^{(j)}).
\end{align*}
\end{cor}

\begin{remark}
  \label{rem:conv:rate}
  The previous results do not unfortunately allow us to decide which of the above two asymptotic representations for $\Cb_n^\nu$ may be better. The knowledge of the underlying convergence rates would be needed for that. As we shall see in Section~\ref{sec:mult}, these representations will be at the root of smooth proposals for bootstrapping $\Cb_n^\nu$ in a time series context.
\end{remark}

\section{Bootstrap by drawing samples from the estimators in the i.i.d.\ case}
\label{sec:boot}

The aim of this section is to study both theoretically and empirically a smooth bootstrap \emph{à la} \citet{KirSegTsu21} based on drawing samples from the smooth estimators defined in the previous section. As hinted at in the introduction, such an approach can only be used in the i.i.d.\ case. Throughout this section, we thus assume that the random vectors in $\Xc_{1:n}$ are i.i.d. Notice that the latter implies Condition~\ref{cond:no:ties}. Given that change-point analysis is essentially of interest in the time series setting, we do not consider a sequential setting below but instead focus only on the situation where $k=1$ and~$l=n$.

This section is organized as follows. After describing the sampling algorithm on which the smooth bootstrap is based, we state conditions under which it is asymptotically valid and report results of finite-sample experiments comparing smooth bootstraps based on the empirical beta copula $C_{1:n}^{\sss \Bin}$ and on its data-adaptive extension $C_{1:n}^{\sss \BetaB}$ proposed in \cite{KojYi22} and recalled at the end of Section~\ref{sec:ec}. 

\subsection{Drawing samples from the smooth empirical copulas}

As explained in Section~\ref{sec:ec}, the empirical beta copula $C_{1:n}^{\sss \Bin}$ is a particular case of the smooth estimators $C_{1:n}^\nu$ defined in~\eqref{eq:C:kl:nu}. From \cite{SegSibTsu17} (see also Lemma~1 in \citealt{KojYi22}), one has that
\begin{equation}
\label{eq:C:1n:beta}
C_{1:n}^{\sss \Bin}(\bm{u}) = \frac{1}{n} \sum_{i=1}^n \prod_{j=1}^d F_{n,R_{ij}^{1:n}}(u_j), \qquad \bm u = (u_1,\dots,u_d) \in [0,1]^d,
\end{equation}
where, for any $n \in \N$ and $r \in \{1,\dots,n\}$, $F_{n,r}$ denotes the d.f.\ of the beta distribution with shape parameters $\alpha = r$ and $\beta = n+1-r$. It follows from~\eqref{eq:C:1n:beta} that $C_{1:n}^{\sss \Bin}$ is a mixture of $n$ $d$-dimensional distributions having beta margins and whose copula is the independence copula. To generate one random variate from $C_{1:n}^{\sss \Bin}$, it thus suffices to randomly select one of the $n$ components of the mixture by drawing a uniform on $\{1,\dots,n\}$ and then generate one random variate from the selected $d$-dimensional distribution. This is detailed in Algorithm~3.2 of \cite{KirSegTsu21}.

In a related way, having~\eqref{eq:C:kl:K} in mind, it thus suffices to assume the following to be able sample from $C_{1:n}^\nu$.

\begin{cond}($C_{1:n}^\nu$ is a mixture)
  \label{cond:mixture}
For any $n \in \N$, $\bm x \in (\R^d)^n$ and $\bm r \in \{1,\dots,n\}^d$, $\Kc_{\bm r}^{\bm x}$ in~\eqref{eq:K} is a d.f. on $[0,1]^d$.
\end{cond}

The sampling algorithm is then conceptually the same as Algorithm~3.2 of \cite{KirSegTsu21}.

\begin{algo}(Sampling from $C_{1:n}^\nu$ under Condition~\ref{cond:mixture})
  \label{algo:sampling}
  \begin{enumerate}
  \item Generate $I$ from the discrete uniform distribution on $\{1,\dots,n\}$.
  \item Generate a random variate $\bm V^{\sss \#}$ from a $d$-dimensional distribution whose d.f.\ is $\Kc_{\bm R_I^{1:n}}^{\sss \Xc_{1:n}}$.
  \end{enumerate}
\end{algo}

The above algorithm can be used in practice as soon as one knows how to sample from the d.f.s $\Kc_{\bm R_i^{1:n}}^{\sss \Xc_{1:n}}$, $i \in \{1,\dots,n\}$.

Interestingly enough, three of the conditions stated in Section~\ref{sec:ec} imply Condition~\ref{cond:mixture} as shown in the next result proven in Appendix~\ref{proofs:boot}. 

\begin{prop}
  \label{prop:mixture}
  Conditions~\ref{cond:unif:marg}, \ref{cond:smooth:surv:marg} and~\ref{cond:smooth:cop} imply Condition~\ref{cond:mixture}. Specifically, under Conditions~\ref{cond:unif:marg}, \ref{cond:smooth:surv:marg} and \ref{cond:smooth:cop}, for any $n \in \N$, $\bm x \in (\R^d)^n$ and $\bm r \in \{1,\dots,n\}^d$, $\Kc_{\bm r}^{\bm x}$ in~\eqref{eq:K} is a d.f. on $[0,1]^d$ whose $d$ univariate margins, denoted by $\Kc_{r_1,1}^{\bm x},\dots,\Kc_{r_d,d}^{\bm x}$, respectively, satisfy $\Kc_{r_j,j}^{\bm x}(u) = \bar \Fc_{j,u}^{\bm x}\{ (r_j - 1) / n\}$, $u \in [0,1]$, $j \in \{1,\dots,d\}$, and whose copula is $\bar \Cc^{\bm x}$.
\end{prop}

\begin{remark}
The previous result leads to an alternative (and simpler) proof of Proposition~11 of \cite{KojYi22}. Indeed, under the assumptions of Proposition~\ref{prop:mixture}, $C_{1:n}^\nu$ in \eqref{eq:C:kl:K} is a convex combination of multivariate d.f.s on $[0,1]^d$ and therefore a multivariate d.f.\ on $[0,1]^d$. Since Condition~\ref{cond:no:ties} holds in the current i.i.d.\ setting, from Section~3.1 in \cite{KojYi22}, Condition~\ref{cond:unif:marg} also implies that $C_{1:n}^\nu$ has standard uniform margins. Hence, under the assumptions of Proposition~\ref{prop:mixture}, $C_{1:n}^\nu$ is a genuine copula.
\end{remark}

For any univariate d.f.\ $H$, let $H^{-1}$ denote its associated quantile function (generalized inverse) defined by $H^{-1}(y) = \inf\{x \in \R : H(x) \geq y \}$, $y \in [0,1]$, with the convention that $\inf \emptyset = \infty$. The second step of Algorithm~\ref{algo:sampling} can then be made explicit under Conditions~\ref{cond:unif:marg}, \ref{cond:smooth:surv:marg} and \ref{cond:smooth:cop} :
\begin{enumerate}[(i)]
\item Generate a random variate $\bm U^{\sss \#}$ from the copula $\bar \Cc^{\sss \Xc_{1:n}}$ independently of $I$.
\item A random variate from the distribution whose d.f.\ is $\Kc_{\bm R_I^{1:n}}^{\sss \Xc_{1:n}}$ is then
  \begin{equation}
    \label{eq:V:hash}
  \bm V^{\sss \#} = \Big(\Kc_{R_{I1}^{1:n},1}^{\sss \Xc_{1:n},-1}(U_1^{\sss \#}),\dots, \Kc_{R_{Id}^{1:n},d}^{\sss \Xc_{1:n},-1}(U_d^{\sss \#})\Big).
  \end{equation}
\end{enumerate}

We end this section by discussing how Algorithm~\ref{algo:sampling} can be practically implemented for the smooth data-adaptive estimator $C_{1:n}^{\sss \BetaB}$ introduced in \cite{KojYi22} as a possible improvement of the empirical beta copula $C_{1:n}^{\sss \Bin}$. Recall from Section~\ref{sec:ec} that $C_{1:n}^{\sss \BetaB}$ is of the form~\eqref{eq:C:kl:dis} with $\bar \Cc^{\sss \Xc_{1:n}}$ the empirical beta copula $C_{1:n}^{\sss \Bin}$ and, for any $j \in \{1,\dots,d\}$ and $u \in [0,1]$, $\bar \Fc_{j,u}^{\sss \Xc_{1:n}}$ the survival function of a scaled (by $1/n$) Beta-Binomial$(n,\alpha, \beta)$ random variable, where  $\alpha = u(n - \rho)/(\rho - 1)$, $\beta = (1-u)(n - \rho)/(\rho - 1)$ and $\rho=4$. The latter implies that, for any $i \in \{1,\dots,n\}$, $j \in \{1,\dots,d\}$ and $u \in [0,1]$,
\begin{equation}
  \label{eq:Kc:r:j}
  \Kc_{R_{ij}^{1:n},j}^{\sss \Xc_{1:n}}(u) = \bar \Fc_{j,u}^{\sss \Xc_{1:n}}\{ (R_{ij}^{1:n} - 1) / n\} = \Pr(n W_{j,u}^{\sss \Xc_{1:n}} > R_{ij}^{1:n} - 1) = \bar \Bc_{n,u,\rho}(R_{ij}^{1:n} - 1), 
\end{equation}
where $\bar \Bc_{n,u,\rho}$ is the survival function of the Beta-Binomial$(n,\alpha, \beta)$. As can be checked from Lemma~27 in \cite{KojYi22} and Lemma~\ref{lem:barBc:strict} in Appendix~\ref{proofs:boot}, the univariate d.f.\ $\Kc_{R_{ij}^{1:n},j}^{\sss \Xc_{1:n}}$ in~\eqref{eq:Kc:r:j} is continuous and strictly increasing, respectively. Hence, to compute its associated quantile function needed in~\eqref{eq:V:hash}, one can proceed numerically. In that respect, an implementation of Algorithm~\ref{algo:sampling} for the \textsf{R} statistical environment for the estimators $C_{1:n}^{\sss \Bin}$ and $C_{1:n}^{\sss \BetaB}$ is available on the web page of the first author.

\subsection{Asymptotic validity results}

Building upon the work of \cite{KirSegTsu21}, we will now provide asymptotic validity results for a smooth bootstrap based on drawing samples from $C_{1:n}^\nu$ in \eqref{eq:C:kl:K} under Conditions~\ref{cond:unif:marg}, \ref{cond:smooth:surv:marg} and~\ref{cond:smooth:cop}. Recall that, according to Proposition~\ref{prop:mixture}, the latter conditions imply Condition~\ref{cond:mixture}. Let $\Vc_{1:n}^{\sss \#} = (\bm V_1^{\sss \#}, \dots, \bm V_n^{\sss \#})$ be a random sample from $C_{1:n}^\nu$ obtained by applying Algorithm~\ref{algo:sampling} $n$ times independently. Note that this implies that the component samples of $\Vc_{1:n}^{\sss \#}$ do not contain ties with probability~1. For any $j \in \{1,\dots,d\}$, let $G_{1:n,j}^{\sss \#}$ be the empirical d.f.\ computed from the $j$th component sample $V_{1j}^{\sss \#},\dots,V_{nj}^{\sss \#}$ of $\Vc_{1:n}^{\sss \#}$. Then, ${R_{ij}^{1:n,\#}} = n G_{1:n,j}^{\sss \#}(V_{ij}^{\sss \#})$ is the rank of $V_{ij}^{\sss \#}$ among $V_{1j}^{\sss \#},\dots,V_{nj}^{\sss \#}$. The (classical) empirical copula of $\Vc_{1:n}^{\sss \#}$ is thus given by
\begin{equation}
  \label{eq:C:1n:hash}
C_{1:n}^{\sss \#}(\bm u) = \frac{1}{n} \sum_{i=1}^n  \prod_{j=1}^d \1\left( \frac{R_{ij}^{1:n,\#}}{n} \leq u_j \right), \qquad \bm u \in [0,1]^d,
\end{equation}
and the smooth analog of $C_{1:n}^\nu$ for $\Vc_{1:n}^{\sss \#}$ is 
\begin{equation}
  \label{eq:C:1n:hash:nu}
C_{1:n}^{\sss \#,\nu}(\bm u) = \int_{[0,1]^d} C_{1:n}^{\sss \#}(\bm w) \dd \nu_{\bm u}^{\sss \Vc_{1:n}^{\sss \#}}(\bm w), \qquad \bm u \in [0,1]^d.
\end{equation}

To state our asymptotic validity results, we consider independent copies $\Vc_{1:n}^{\sss \#,[1]}, \Vc_{1:n}^{\sss \#,[2]}, \dots$ of $\Vc_{1:n}^{\sss \#}$. Let $C_{1:n}^{\sss \#,[i]}$ (resp.\ $C_{1:n}^{\sss \#, \nu, [i]}$) be the version of $C_{1:n}^{\sss \#}$ in~\eqref{eq:C:1n:hash} (resp.\ $C_{1:n}^{\sss \#, \nu}$ in~\eqref{eq:C:1n:hash:nu}) obtained from $\Vc_{1:n}^{\sss \#,[i]}$, $i \in \N$. 

The following result can be regarded as an extension of Proposition~3.3 of \cite{KirSegTsu21} and is proven in Appendix~\ref{proof:thm:asym:val:Cb:n:hash}. 

\begin{thm}
\label{thm:asym:val:Cb:n:hash}
Assume that the random vectors in $\Xc_{1:n}$ are i.i.d., and that Conditions~\ref{cond:unif:marg}, \ref{cond:smooth:surv:marg}, \ref{cond:smooth:cop}, \ref{cond:pd} and~\ref{cond:var:W} hold. Then,
\begin{align*}
  \big(\Cb_n(0,1,\cdot), \sqrt{n}(C_{1:n}^{\sss \#,[1]} - C_{1:n}), &\sqrt{n}(C_{1:n}^{\sss \#,[2]} - C_{1:n}) \big) \\&\leadsto \big(\Cb_C(0,1,\cdot),\Cb_C^{\sss [1]}(0,1,\cdot),\Cb_C^{\sss [2]}(0,1,\cdot)\big), \\
  \big(\Cb_n^\nu(0,1,\cdot), \sqrt{n}(C_{1:n}^{\sss \#,\nu,[1]} - C_{1:n}^\nu), &\sqrt{n}(C_{1:n}^{\sss \#,\nu,[2]} - C_{1:n}^\nu) \big) \\
  &\leadsto \big(\Cb_C(0,1,\cdot),\Cb_C^{\sss [1]}(0,1,\cdot),\Cb_C^{\sss [2]}(0,1,\cdot)\big) 
\end{align*}
in $\{ \ell^\infty([0,1]^d) \}^3$, where $\Cb_n$ and $\Cb_n^\nu$ are defined in~\eqref{eq:Cb:n} and~\eqref{eq:Cb:n:nu}, respectively, and $\Cb_C^{\sss [1]}$ and $\Cb_C^{\sss [2]}$ are independent copies of $\Cb_C$ defined in~\eqref{eq:Cb:C}.
\end{thm}

\begin{remark}
The first joint weak convergence in Theorem~\ref{thm:asym:val:Cb:n:hash} establishes the asymptotic validity of a smooth bootstrap for the (non-sequential) classical empirical process while the second one provides a similar results for the smooth empirical copula process $\Cb_n^\nu(0,1,\cdot)$. According to Lemma~3.1 in \cite{BucKoj19}, these two joint weak convergences are equivalent to similar joint weak convergences with $B \geq 2$ bootstrap replicates. In a further step, the latter can be transferred to the ``statistic level'' using the continuous mapping theorem or the functional delta method, which could then be combined with the results in Section~4 of \cite{BucKoj19} to establish the validity of bootstrap-based confidence intervals or tests. Note also that, from Lemma~3.1 in \cite{BucKoj19}, the unconditional asymptotic validity results appearing in Theorem~\ref{thm:asym:val:Cb:n:hash} are equivalent to possibly more classical conditional results which rely, however, on a more subtle mode of convergence. For instance, the first claim can be equivalently informally stated as ``$\sqrt{n}(C_{1:n}^{\sss \#,[1]} - C_{1:n})$ converges weakly to $\Cb_C(0,1,\cdot)$ in $\ell^\infty([0,1]^d)$ conditionally on the data in probability''; see, e.g., \citet[Section~2.2.3]{Kos08} or Appendix~\ref{proof:thm:asym:val:Cb:n:hash} for a precise definition of that mode of convergence.
\end{remark}

\subsection{Finite-sample comparison of two smooth bootstraps} 
\label{sec:ci}

\begin{table}[t!]
\centering
\caption{Coverage probabilities (cov.) and average lengths (ave.) of 95\%-confidence intervals for Kendall's tau estimated from 1000 random samples of size $n \in \{20,40,80,160\}$ from the bivariate Clayton or Gumbel--Hougaard copula with a Kendall's tau of $\tau \in \{0,0.5,0.75,0.9\}$. Each confidence interval was computed using 1000 smooth bootstrap samples drawn from either $C_{1:n}^{\sss \Bin}$ or $C_{1:n}^{\sss \BetaB}$ using Algorithm~\ref{algo:sampling}.} 
\label{tab:kendall}
\begingroup\small
\begin{tabular}{rrrrrrrrrr}
  \hline
  \multicolumn{2}{c}{} & \multicolumn{4}{c}{Clayton} & \multicolumn{4}{c}{Gumbel--Hougaard} \\ \cmidrule(lr){3-6} \cmidrule(lr){7-10} \multicolumn{2}{c}{} & \multicolumn{2}{c}{Bin} & \multicolumn{2}{c}{BetaB4} & \multicolumn{2}{c}{Bin} & \multicolumn{2}{c}{BetaB4} \\ \cmidrule(lr){3-4} \cmidrule(lr){5-6} \cmidrule(lr){7-8} \cmidrule(lr){9-10}  $\tau$ & $n$ & cov. & ave. & cov. & ave. & cov. & ave. & cov. & ave.  \\ \hline
0.00 & 20 & 0.973 & 0.626 & 0.967 & 0.615 & 0.969 & 0.624 & 0.962 & 0.614 \\ 
   & 40 & 0.962 & 0.428 & 0.954 & 0.424 & 0.950 & 0.430 & 0.944 & 0.425 \\ 
   & 80 & 0.949 & 0.298 & 0.943 & 0.296 & 0.970 & 0.297 & 0.961 & 0.296 \\ 
   & 160 & 0.944 & 0.208 & 0.943 & 0.207 & 0.949 & 0.208 & 0.949 & 0.208 \\ 
  0.50 & 20 & 0.971 & 0.513 & 0.972 & 0.493 & 0.978 & 0.521 & 0.982 & 0.498 \\ 
   & 40 & 0.958 & 0.347 & 0.954 & 0.334 & 0.959 & 0.345 & 0.957 & 0.332 \\ 
   & 80 & 0.946 & 0.239 & 0.938 & 0.233 & 0.950 & 0.237 & 0.947 & 0.231 \\ 
   & 160 & 0.954 & 0.168 & 0.957 & 0.165 & 0.954 & 0.164 & 0.958 & 0.162 \\ 
  0.75 & 20 & 0.717 & 0.392 & 0.899 & 0.367 & 0.777 & 0.391 & 0.927 & 0.358 \\ 
   & 40 & 0.728 & 0.234 & 0.908 & 0.221 & 0.793 & 0.231 & 0.954 & 0.211 \\ 
   & 80 & 0.798 & 0.151 & 0.930 & 0.146 & 0.844 & 0.146 & 0.953 & 0.137 \\ 
   & 160 & 0.866 & 0.101 & 0.943 & 0.100 & 0.883 & 0.098 & 0.944 & 0.094 \\ 
  0.90 & 20 & 0.000 & 0.315 & 0.212 & 0.272 & 0.000 & 0.317 & 0.270 & 0.264 \\ 
   & 40 & 0.000 & 0.160 & 0.475 & 0.131 & 0.000 & 0.162 & 0.593 & 0.127 \\ 
   & 80 & 0.000 & 0.086 & 0.692 & 0.074 & 0.000 & 0.087 & 0.804 & 0.069 \\ 
   & 160 & 0.000 & 0.050 & 0.837 & 0.047 & 0.000 & 0.050 & 0.902 & 0.043 \\ 
   \hline
\end{tabular}
\endgroup
\end{table}

As already mentioned in the introduction, in their Monte Carlo experiments, \cite{KirSegTsu21} found the smooth bootstrap based on the empirical beta copula $C_{1:n}^{\sss \Bin}$ to be a competitive alternative to many other resampling schemes (including the multiplier bootstrap to be studied in the forthcoming section). Since the data-adaptive empirical copula $C_{1:n}^{\sss \BetaB}$ was found to outperform the empirical beta copula $C_{1:n}^{\sss \Bin}$ in the experiments reported in \cite{KojYi22}, it seems natural to empirically investigate how the smooth bootstrap based on $C_{1:n}^{\sss \BetaB}$ compares to the smooth bootstrap based on $C_{1:n}^{\sss \Bin}$. To do so, we reproduced some of the experiments reported in Sections 4.2 and 4.3 of \cite{KirSegTsu21}.

We first estimated coverage probabilities and average lengths of confidence intervals of level 95\% for Kendall’s tau from 1000 random samples of size $n \in \{20, 40, 80, 160\}$ from the bivariate Clayton or Gumbel–Hougaard copula with a Kendall’s tau of $\tau \in \{0, 0.5, 0.75, 0.9\}$. Each confidence interval was computed using 1000 smooth bootstrap samples drawn from either $C_{1:n}^{\sss \Bin}$ or $C_{1:n}^{\sss \BetaB}$. The results are reported in Table~\ref{tab:kendall}. As one can see, under independence or moderate dependence ($\tau \in \{0,0.5\}$), the estimated coverage probabilities are overall on target and very similar for the two resampling schemes. The intervals obtained using the smooth bootstrap based on $C_{1:n}^{\sss \BetaB}$ seem nonetheless to be slightly shorter on average. Under strong dependence ($\tau = 0.75$) however, the estimated coverage probabilities of the confidence intervals computed using the smooth bootstrap based on $C_{1:n}^{\sss \Bin}$ are substantially below the 0.95 target value. The results for $\tau = 0.9$ actually show that the smooth bootstrap based on $C_{1:n}^{\sss \Bin}$ is unable to generate samples with such a very strong dependence. While its results are not perfect, the smooth bootstrap based on $C_{1:n}^{\sss \BetaB}$ copes much better with strong dependence. This is likely to be due to the modification of the ``shape'' of the underlying smoothing distributions using the empirical beta copula in the expression of $C_{1:n}^{\sss \BetaB}$ as can be deduced from~\eqref{eq:C:kl:dis}.

\begin{table}[t!]
\centering
\caption{Coverage probabilities (cov.) and average lengths (ave.) of 95\%-confidence intervals for the parameter of a bivariate Frank copula estimated by maximum pseudo-likelihood from 1000 random samples of size $n \in \{20,40,80\}$ from the bivariate Frank copula with a Kendall's tau of $\tau \in \{-0.9,-0.75,-0.5,0,0.5,0.75,0.9\}$. Each confidence interval was computed using 1000 smooth bootstrap samples drawn from either $C_{1:n}^{\sss \Bin}$ or $C_{1:n}^{\sss \BetaB}$ using Algorithm~\ref{algo:sampling}.} 
\label{tab:mpl}
\begingroup\small
\begin{tabular}{rrrrrr}
  \hline
  \multicolumn{2}{c}{} & \multicolumn{2}{c}{Bin} & \multicolumn{2}{c}{BetaB4} \\ \cmidrule(lr){3-4} \cmidrule(lr){5-6} $\tau$ & $n$ & cov. & ave. & cov. & ave.  \\ \hline
-0.90 & 20 & 0.000 & 0.194 & 0.000 & 0.133 \\ 
   & 40 & 0.000 & 0.086 & 0.056 & 0.051 \\ 
   & 80 & 0.000 & 0.039 & 0.313 & 0.023 \\ 
  -0.75 & 20 & 0.731 & 0.286 & 0.940 & 0.237 \\ 
   & 40 & 0.641 & 0.153 & 0.939 & 0.126 \\ 
   & 80 & 0.662 & 0.088 & 0.947 & 0.077 \\ 
  -0.50 & 20 & 0.988 & 0.548 & 0.981 & 0.511 \\ 
   & 40 & 0.975 & 0.342 & 0.957 & 0.327 \\ 
   & 80 & 0.957 & 0.230 & 0.940 & 0.224 \\ 
  0.00 & 20 & 0.952 & 1.009 & 0.946 & 1.002 \\ 
   & 40 & 0.937 & 0.681 & 0.929 & 0.681 \\ 
   & 80 & 0.941 & 0.467 & 0.938 & 0.469 \\ 
  0.50 & 20 & 0.986 & 0.542 & 0.970 & 0.508 \\ 
   & 40 & 0.972 & 0.344 & 0.949 & 0.328 \\ 
   & 80 & 0.959 & 0.224 & 0.948 & 0.219 \\ 
  0.75 & 20 & 0.722 & 0.285 & 0.938 & 0.235 \\ 
   & 40 & 0.634 & 0.154 & 0.927 & 0.128 \\ 
   & 80 & 0.671 & 0.088 & 0.942 & 0.077 \\ 
  0.90 & 20 & 0.000 & 0.193 & 0.000 & 0.132 \\ 
   & 40 & 0.000 & 0.086 & 0.046 & 0.051 \\ 
   & 80 & 0.000 & 0.039 & 0.319 & 0.023 \\ 
   \hline
\end{tabular}
\endgroup
\end{table}

In a second experiment, following \citet[Section~4.3]{KirSegTsu21} we estimated coverage probabilities and average lengths of 95\%-confidence intervals for the parameter of a bivariate Frank copula estimated by maximum pseudo-likelihood \citep[see][]{GenGhoRiv95} from 1000 random samples of size $n \in \{20, 40, 80, 160\}$ from the bivariate Frank copula with a Kendall’s tau of $\tau \in \{-0.9,-0.75,-0.5,0,0.5,0.75,0.9\}$. Again, each confidence interval was computed using 1000 smooth bootstrap samples drawn from either $C_{1:n}^{\sss \Bin}$ or $C_{1:n}^{\sss \BetaB}$. The results are reported in Table~\ref{tab:mpl} and the main conclusion is qualitatively the same as for the previous experiment: the smooth bootstrap based on $C_{1:n}^{\sss \BetaB}$ copes much better with strong dependence than the smooth bootstrap based on $C_{1:n}^{\sss \Bin}$.

\section{Smooth sequential dependent multiplier bootstraps in the time series case}
\label{sec:mult}

The smooth bootstrap investigated in the previous section can only be used in the case of i.i.d.\ observations. Fortunately, the multiplier bootstrap, one of the most popular approaches for bootstrapping functionals of the classical empirical copula, can be employed in the time series setting. In this section, after providing some intuitions and defining \emph{multiplier sequences}, we recall the non-smooth sequential dependent multiplier bootstrap studied in \cite{BucKoj16}. We next propose smooth extensions of the latter, provide asymptotic validity results and compare the finite-sample performance of three (smooth) multiplier bootstraps for approximating three (smooth) empirical copula processes. Finally, as an application, we consider a smooth version (based on the empirical beta copula and corresponding smooth multiplier bootstrap replicates) of the test for change-point detection developed in \cite{BucKojRohSeg14} and we compare its finite-sample performance to that of its non-smooth counterpart.

\subsection{Main intuition and existing work}
\label{sec:intuition}

As mentioned in Section~\ref{sec:asym}, Condition~\ref{cond:Bn} holds under \emph{strong mixing}. Given a stationary time series $(\bm Y_i)_{i \in \Z}$, denote by $\Fc_j^k$ the $\sigma$-field generated by $(\bm Y_i)_{j \leq i \leq k}$, $j, k \in \Z \cup \{-\infty,+\infty \}$, and recall that the strong mixing coefficients corresponding to the stationary sequence $(\bm Y_i)_{i \in \Z}$ are then defined by
\begin{equation*}
\alpha_r^{\bm Y} = \sup_{A \in \Fc_{-\infty}^0,B\in \Fc_{r}^{+\infty}} \big| \Pr(A \cap B) - \Pr(A) \Pr(B) \big|, \qquad r \in \N, \, r > 0,
\end{equation*}
and that the sequence $(\bm Y_i)_{i \in \Z}$ is said to be \emph{strongly mixing} if $\alpha_r^{\bm Y} \to 0$ as $r \to \infty$.

From \cite{Buc15}, Condition~\ref{cond:Bn} holds 
if the strong mixing coefficients of the time series $(\bm X_i)_{i \in \Z}$ satisfy $\alpha_r^{\bm X} = O(r^{-a})$ with $a > 1$ as $r \to \infty$. In that case, Theorem~\ref{thm:wc:Cb:n} suggests that, in order to bootstrap the classical sequential empirical copula process $\Cb_n$ in~\eqref{eq:Cb:n} in an asymptotically valid way, it suffices to bootstrap the process $\tilde \Cb_n$ in~\eqref{eq:Cb:n:tilde}. The latter could be done by bootstrapping $\Bb_n$ in~\eqref{eq:Bb:n} and estimating the first-order partial derivatives $\dot{C}_j$, $j \in \{1,\dots,d\}$, of $C$. Such an approach was initially proposed in the independent non-sequential setting by \cite{Sca05} and \cite{RemSca09} who used a \emph{multiplier bootstrap} in the spirit of \citet[Chapter 2.9]{vanWel96} to resample $\Bb_n$, and finite-differencing to estimate the partial derivatives $\dot{C}_j$, $j \in \{1,\dots,d\}$. This resampling scheme was extended to the time series sequential setting in \cite{BucKoj16} and \cite{BucKojRohSeg14}.

\subsection{I.i.d.\ and dependent multiplier sequences}
\label{sec:mult:seq}

In the case of independent observations, multiplier bootstraps are based on \emph{i.i.d.\ multiplier sequences}. We say that a sequence of random variables $(\xi_{i,n})_{i \in \Z}$ is an i.i.d.\ multiplier sequence if:
\begin{enumerate}[({M}0)]
\item 
$(\xi_{i,n})_{i \in \Z}$ is i.i.d., independent of $\Xc_{1:n}$, with distribution not changing with~$n$, having mean 0, variance 1, and being such that $\int_0^\infty \{ \Pr(|\xi_{0,n}| > x) \}^{1/2} \dd x < \infty$.
\end{enumerate}

The time series extension of the multiplier bootstrap relies on the notion of \emph{dependent multiplier sequence}. The key idea due to \cite{Buh93} is to replace i.i.d.\ multipliers by suitably serially dependent multipliers that will capture the serial dependence in the data. We say that a sequence of random variables $(\xi_{i,n})_{i \in \Z}$ is a dependent multiplier sequence if:
\begin{enumerate}[({M}1)]
\item 
The sequence of random variables $(\xi_{i,n})_{i \in \Z}$ is stationary with $\Ex(\xi_{0,n}) = 0$, $\Ex(\xi_{0,n}^2) = 1$ and $\sup_{n \geq 1} \Ex(|\xi_{0,n}|^\gamma) < \infty$ for all $\gamma \geq 1$, and is independent of the available sample $\Xc_{1:n}$.
\item 
There exists a sequence $\ell_n \to \infty$ of strictly positive constants such that $\ell_n = o(n)$ and the sequence $(\xi_{i,n})_{i \in \Z}$ is $\ell_n$-dependent, i.e., $\xi_{i,n}$ is independent of $\xi_{i+h,n}$ for all $h > \ell_n$ and $i \in \N$.
\item 
There exists a function $\varphi:\R \to [0,1]$, symmetric around 0, continuous at $0$, satisfying $\varphi(0)=1$ and $\varphi(x)=0$ for all $|x| > 1$ such that $\Ex(\xi_{0,n} \xi_{h,n}) = \varphi(h/\ell_n)$ for all $h \in \Z$.
\end{enumerate}

As shall become clearer for instance from~\eqref{eq:Bb:n:hat} or~\eqref{eq:Bb:n:check} below, the bandwidth parameter $\ell_n$ defined in (M2) plays a role similar to that of the block length in the block bootstrap. In practice, for the non-smooth sequential dependent multiplier bootstrap to be presented in the forthcoming section, its value can be chosen in a data-driven way using the approach described in detail in \citet[Section~5]{BucKoj16}; see also Section~\ref{sec:cp}. The latter reference also describes in detail ways to generate dependent multiplier sequences.

\subsection{Non-smooth sequential dependent multiplier replicates}
\label{sec:non-smooth:mult}

Let $(\xi_{i,n}^{\sss [1]})_{i \in \Z}$, $(\xi_{i,n}^{\sss [2]})_{i \in \Z}$,\dots, be independent copies of the same multiplier sequence. Two different multiplier bootstrap replicates of the process $\Bb_n$ in~\eqref{eq:Bb:n} were proposed in \citet{BucKoj16} and \citet{BucKojRohSeg14}, respectively. For any $b \in \N$, $(s,t) \in \Lambda$ and $\bm u \in [0,1]^d$, they are defined by
\begin{equation}
\label{eq:Bb:n:hat}
\hat{\Bb}_n^{\sss [b]} (s,t,\bm u) = \frac{1}{\sqrt{n}}\sum_{i=\ip{ns}+1}^{\ip{nt}} \xi_{i,n}^{\sss [b]}\left\{  \1(\hat{\bm U}_i^{1:n} \leq \bm u)- C_{1:n}(\bm u) \right\}
\end{equation}
and
\begin{equation}
\label{eq:Bb:n:check}
\check{\Bb}_n^{\sss [b]} (s,t,\bm u) = \frac{1}{\sqrt{n}}\sum_{i=\ip{ns}+1}^{\ip{nt}} \xi_{i,n}^{\sss [b]}\left\{  \1(\hat{\bm U}_i^{\ip{ns}+1:\ip{nt}} \leq \bm u)- C_{\ip{ns}+1:\ip{nt}}(\bm u) \right\},
\end{equation}
respectively, where $C_{1:n}$ and $C_{\ip{ns}+1:\ip{nt}}$ are generically defined in~\eqref{eq:C:kl}  and with the convention that $\hat{\Bb}_n^{\sss [b]}(s, t, \cdot) = \check{\Bb}_n^{\sss [b]}(s, t, \cdot) = 0$ if $\ip{nt} - \ip{ns} = 0$.

In order to define multiplier bootstrap replicates of $\tilde{\Cb}_n$ in~\eqref{eq:Cb:n:tilde}, it is further necessary to estimate the unknown first-order partial derivatives $\dot C_j$, $j \in \{1,\dots,d\}$, of $C$.  In the rest of this section, $\dot C_{j,k:l}$ will denote an estimator of $\dot C_j$ based on a stretch $\Xc_{k:l} = (\bm X_k,\dots,\bm X_l)$ of observations, $1 \leq k \leq l \leq n$, with the convention that $\dot C_{j,k:l} = 0$ if $k > l$. Then, following \cite{BucKoj16} and \cite{BucKojRohSeg14}, we consider two types of multiplier bootstrap replicates of $\Cb_n$ in~\eqref{eq:Cb:n}. For any $b \in \N$, $(s,t) \in \Lambda$ and $\bm u \in [0,1]^d$, these are defined by
\begin{equation}
\label{eq:Cb:n:hat}
  \hat{\Cb}_n^{\sss [b]}(s, t, \bm u) = \hat\Bb_n^{\sss [b]}(s, t, \bm u) -   \sum_{j=1}^d \dot C_{j,1:n}(\bm u) \,  \hat\Bb_n^{\sss [b]}(s, t, \bm u^{(j)})
  \end{equation}
and
\begin{equation}
\label{eq:Cb:n:check}
  \check{\Cb}_n^{\sss [b]}(s, t, \bm u) = \check\Bb_n^{\sss [b]}(s, t, \bm u) -  \sum_{j=1}^d \dot C_{j,\ip{ns}+1:\ip{nt}}(\bm u) \,  \check\Bb_n^{\sss [b]}(s, t, \bm u^{(j)}),
\end{equation}
respectively, where $\hat\Bb_n^{\sss [b]}$ (resp.\ $\check\Bb_n^{\sss [b]}$) is defined in~\eqref{eq:Bb:n:hat} (resp.\ \eqref{eq:Bb:n:check}). Clearly, both types of replicates coincide in a non-sequential setting as $\hat{\Cb}_n^{\sss [b]}(0, 1, \cdot) = \check{\Cb}_n^{\sss [b]}(0, 1, \cdot)$. As far as the estimators of the partial derivatives are concerned, it is expected that the more accurate they are, the better the approximation of the ``sampling distribution'' of $\Cb_n$ by the multiplier replicates will be. The latter aspect will be discussed in detail in Section~\ref{sec:pd}, where two broad classes of smooth estimators will be introduced and studied both theoretically and empirically.

\subsection{Smooth sequential dependent multiplier replicates}

We now consider a similar construction but based on smooth analogs of $\hat\Bb_n^{\sss [b]}$ in~\eqref{eq:Bb:n:hat} and  $\check\Bb_n^{\sss [b]}$ in \eqref{eq:Bb:n:check}. Specifically, Corollary~\ref{cor:wc:Cb:n:nu} suggests that, to bootstrap $\Cb_n^\nu$ in~\eqref{eq:Cb:n:nu}, a first step is to bootstrap $\tilde \Bb_n^\nu$ in~\eqref{eq:tilde:Bb:n:nu} or $\bar \Bb_n^\nu$ in~\eqref{eq:bar:Bb:n:nu}. By analogy with~\eqref{eq:C:kl:nu} and~\eqref{eq:C:kl:K}, natural smooth analogs of $\hat\Bb_n^{\sss [b]}$  and $\check\Bb_n^{\sss [b]}$ could be defined, for any $b \in \N$, $(s,t) \in \Lambda$ and $\bm u \in [0,1]^d$, by
\begin{equation}
  \label{eq:Bb:n:nu:hat}
  \begin{split}
  \hat{\Bb}_n^{\sss [b],\nu} (s,t,\bm u) &= \int_{[0,1]^d} \hat\Bb_n^{\sss [b]}(s,t,\bm w) \dd \nu_{\bm u}^{\sss \Xc_{1:n}}(\bm w) \\
&= \frac{1}{\sqrt{n}} \sum_{i=\ip{ns}+1}^{\ip{nt}} \xi_{i,n}^{\sss [b]}\left\{  \Kc_{\sss{\bm R^{1:n}_i}}^{\sss \Xc_{1:n}}(\bm u)- C_{1:n}^\nu(\bm u) \right\}
\end{split}
\end{equation}
and
\begin{equation}
\label{eq:Bb:n:nu:check}
  \begin{split}
  \check{\Bb}_n^{\sss [b],\nu} (s,t,\bm u) &= \int_{[0,1]^d} \check \Bb_n^{\sss [b]}(s,t,\bm w) \dd \nu_{\bm u}^{\sss \Xc_{\ip{ns}+1:\ip{nt}}}(\bm w) \\ &= \frac{1}{\sqrt{n}} \sum_{i=\ip{ns}+1}^{\ip{nt}} \xi_{i,n}^{\sss [b]}\left\{ \Kc_{\sss{\bm R^{\ip{ns}+1:\ip{nt}}_i}}^{\sss \Xc_{\ip{ns}+1:\ip{nt}}}(\bm u)  - C_{\ip{ns}+1:\ip{nt}}^\nu(\bm u) \right\},
\end{split}
\end{equation}
respectively, where $\Kc_{\sss{\bm R^{1:n}_i}}^{\sss \Xc_{1:n}}$ and $\Kc_{\sss{\bm R^{\ip{ns}+1:\ip{nt}}_i}}^{\sss \Xc_{\ip{ns}+1:\ip{nt}}}$ are defined in~\eqref{eq:K}. Combining these ingredients with estimators of the unknown partial derivatives of $C$, as smooth analogs of $\hat\Cb_n^{\sss [b]}$ in~\eqref{eq:Cb:n:hat} and $\check\Cb_n^{\sss [b]}$ in~\eqref{eq:Cb:n:check}, we obtain
\begin{equation}
\label{eq:Cb:n:hat:nu}
\hat{\Cb}_n^{\sss [b], \nu}(s, t, \bm u) = \hat\Bb_n^{\sss [b], \nu}(s, t, \bm u) -   \sum_{j=1}^d \dot C_{j,1:n}(\bm u) \,  \hat\Bb_n^{\sss [b], \nu}(s, t, \bm u^{(j)}),
\end{equation}
and
\begin{equation}
\label{eq:Cb:n:check:nu}
  \check{\Cb}_n^{\sss [b], \nu}(s, t, \bm u) = \check\Bb_n^{\sss [b], \nu}(s, t, \bm u) -  \sum_{j=1}^d \dot C_{j,\ip{ns}+1:\ip{nt}}(\bm u) \,  \check\Bb_n^{\sss [b], \nu}(s, t, \bm u^{(j)}),
\end{equation}
respectively, for $b \in \N$, $(s,t) \in \Lambda$ and $\bm u \in [0,1]^d$. 

To establish the asymptotic validity of these smooth multiplier bootstrap replicates, it will suffice that the partial derivative estimators satisfy the following rather natural mild condition.

\begin{cond}[Bounded and weakly consistent partial derivative estimators]
  \label{cond:pd:est}
There exists a constant $\zeta > 0$ such that, for any $j \in \{1,\dots,d\}$ and $n \in \N$,
$$
\sup_{(s,t,\bm u) \in \Lambda \times [0,1]^d}  \left| \dot C_{j,\ip{ns}+1:\ip{nt}}(\bm u) \right| \leq \zeta.
$$
Furthermore, for any $\delta \in (0,1)$, $\eps \in (0,1/2)$ and $j \in \{1,\dots,d\}$,
\begin{equation*}
\sup_{\substack{(s,t) \in \Lambda \\ t-s \geq \delta}} \sup_{\substack{\bm u \in [0,1]^d\\ u_j \in [\eps,1-\eps]}}  \left| \dot C_{j,\ip{ns}+1:\ip{nt}}(\bm u) - \dot C_j(\bm u) \right| = o_\Pr(1).
\end{equation*}
\end{cond}

In addition, following~\cite{BucKojRohSeg14}, we impose the following condition on the observations and the underlying multiplier sequences.

\begin{cond}[Strong mixing and multiplier conditions]
\label{cond:DGP}
One of the following two conditions holds:
\begin{enumerate}[\bf (i)]
\item The random vectors in $\Xc_{1:n}$ are i.i.d.\ and $(\xi_{i,n}^{\sss [1]})_{i \in \Z},(\xi_{i,n}^{\sss [2]})_{i \in \Z}$,\dots are independent copies of a multiplier sequence satisfying~(M0).
\item The stretch $\Xc_{1:n}$ is drawn from a stationary sequence $(\bm X_i)_{i \in \Z}$ whose strong mixing coefficients satisfy $\alpha_r^{\bm X} = O(r^{-a})$ for some $a > 3 + 3d/2$ as $r \to \infty$. Furthermore, $(\xi_{i,n}^{\sss [1]})_{i \in \Z},(\xi_{i,n}^{\sss [2]})_{i \in \Z}$, \dots are independent copies of a dependent multiplier sequence satisfying~(M1)--(M3) with $\ell_n = O(n^{1/2 - \gamma})$ for some $0 < \gamma < 1/2$.
\end{enumerate}
\end{cond}

The following result is proven in Appendix~\ref{proof:thm:ae:Cb:n:hat:check}.

\begin{thm}[Asymptotic validity of the smooth dependent multiplier bootstraps]
\label{thm:ae:Cb:n:hat:check}
Under Conditions~\ref{cond:pd}, \ref{cond:var:W}, \ref{cond:pd:est} and~\ref{cond:DGP}, for any $b \in \N$, there holds
\begin{align}
  \label{eq:ae:Cb:n:hat}
  \sup_{(s,t,\bm u) \in \Lambda \times [0,1]^d} \left| \hat{\Cb}_n^{\sss [b],\nu} (s,t,\bm u)-\hat{\Cb}_n^{\sss [b]} (s,t,\bm u)\right| = o_\Pr(1), \\
  \label{eq:ae:Cb:n:check}
  \sup_{(s,t,\bm u) \in \Lambda \times [0,1]^d} \left| \check{\Cb}_n^{\sss [b],\nu} (s,t,\bm u)-\check{\Cb}_n^{\sss [b]} (s,t,\bm u)\right| = o_\Pr(1).
\end{align}
Furthermore,
\begin{align*}
  (\Cb_n^\nu,\hat{\Cb}_n^{\sss [1],\nu},\hat{\Cb}_n^{\sss [2],\nu}) &\leadsto (\Cb_C,\Cb_C^{\sss [1]},\Cb_C^{\sss [2]}), \\
  (\Cb_n^\nu,\check{\Cb}_n^{\sss [1],\nu},\check{\Cb}_n^{\sss [2],\nu}) &\leadsto (\Cb_C,\Cb_C^{\sss [1]},\Cb_C^{\sss [2]})
\end{align*}
in $\{ \ell^{\infty} (\Lambda \times [0,1]^d) \}^3$, where $\Cb_C^{\sss [1]}$ and $\Cb_C^{\sss [2]}$ are independent copies of $\Cb_C$ defined in~\eqref{eq:Cb:C}.
\end{thm}

\subsection{Finite-sample comparison of three multiplier bootstraps}

From Theorem~\ref{thm:Cb:n:nu}, we know that, under Conditions~\ref{cond:pd} and~\ref{cond:var:W}, the classical sequential empirical copula process $\Cb_n$ in~\eqref{eq:Cb:n} and the smooth sequential empirical copula process $\Cb_n^\nu$ in~\eqref{eq:Cb:n:nu} are asymptotically equivalent. In a related way, Theorem~\ref{thm:ae:Cb:n:hat:check} provides conditions under which corresponding multiplier and smooth multiplier replicates are asymptotically equivalent. Although one expects that $\Cb_n^\nu$ is probably best resampled using multiplier replicates constructed with the same smoothing distributions, that is, with $\hat{\Cb}_n^{\sss [b], \nu}$ in~\eqref{eq:Cb:n:hat:nu} or $\check{\Cb}_n^{\sss [b], \nu}$ in~\eqref{eq:Cb:n:check:nu}, we have no asymptotic results to support this (see also Remark~\ref{rem:conv:rate}). Indeed, given that all versions of multiplier replicates are asymptotically equivalent, it may well be that, for instance, in some cases, classical (non-smooth) multiplier replicates are equivalent or even preferable to smooth multiplier replicates when it comes to resampling $\Cb_n^\nu$. It is the aim of this section to study this empirically. For simplicity, we restrict our investigations to a non-sequential setting and independent observations.

Specifically, we designed experiments to study which multiplier replicates are best suited to estimate certain functionals of the three (non-sequential) empirical copula processes defined, for any $\bm u \in [0,1]^d$, by
\begin{align}
  \label{eq:Cb:n:Dirac}
  \Cb_n^{\sss \Dirac}(\bm u) &= \sqrt{n}\{ C_{1:n}(\bm u) - C(\bm u)\} = \Cb_n(0,1,\bm u), \\
  \label{eq:Cb:n:Bin}
  \Cb_n^{\sss \Bin}(\bm u) &=  \sqrt{n}\{ C_{1:n}^{\sss \Bin}(\bm u) - C(\bm u)\}, \\
   \label{eq:Cb:n:BetaB4}
  \Cb_n^{\sss \BetaB}(\bm u) &=  \sqrt{n}\{ C_{1:n}^{\sss \BetaB}(\bm u) - C(\bm u)\}, 
\end{align}
where 
\begin{itemize}
\item $\Cb_n$ is the classical (non-smooth) sequential empirical copula process defined in~\eqref{eq:Cb:n},
\item $C_{1:n}^{\sss \Bin}$ is the empirical beta copula in~\eqref{eq:C:1n:beta} (which is obtained by considering smoothing distributions with scaled binomial margins and independence copula as explained in Section~\ref{sec:ec}),
\item $C_{1:n}^{\sss \BetaB}$ is the version of $C_{1:n}^\nu$ introduced in Section~\ref{sec:ec} obtained by considering smoothing distributions with scaled beta-binomial margins and survival copula the empirical beta copula $C_{1:n}^{\sss \Bin}$, and found to have the best finite-sample performance in the numerical experiments of \cite{KojYi22}.
\end{itemize}

As already mentioned, since we are in a non-sequential setting, the two generic multiplier replicates defined in~\eqref{eq:Cb:n:hat:nu} and~\eqref{eq:Cb:n:check:nu} coincide. To approximate the ``sampling distributions'' of the three empirical copula processes defined above, we considered as candidate bootstraps the multiplier replicates defined using the same smoothing distributions. They will be denoted by $\hat \Cb_n^{\sss [b], \Dirac}$, $\hat \Cb_n^{\sss [b],\Bin}$ and $\hat \Cb_n^{\sss [b],\BetaB}$, $b \in \N$, respectively, as we continue. To only investigate the effect of the choice of the smoothing distributions involved in the definition of $\hat{\Bb}_n^{\sss [b],\nu}$ in~\eqref{eq:Bb:n:nu:hat}, all three multiplier replicates were computed using the true partial derivative $\dot C_j$, $j \in \{1,\dots,d\}$. Furthermore, since we restricted our experiments to independent observations, all the multiplier replicates were based on i.i.d.\ multiplier sequences defined in (M0) in Section~\ref{sec:mult:seq}. Following \cite{BucDet10}, these sequences were simply taken to be random samples drawn from the uniform distribution on $\{-1, 1\}$.

For the design of our experiments, we followed again \cite{BucDet10}. First, for $d=2$, we assessed how well the covariances of the empirical processes $\Cb_n^{\sss \Dirac}$ in~\eqref{eq:Cb:n:Dirac}, $\Cb_n^{\sss \Bin}$ in~\eqref{eq:Cb:n:Bin} and $\Cb_n^{\sss \BetaB}$ in~\eqref{eq:Cb:n:BetaB4} at the points $P = \{(i/3,j/3):i,j=1,2\}$ can be approximated using the three possible multiplier bootstrap replicates. For each target empirical copula process, we began by precisely estimating its covariance at the points in $P$ from $100\,000$ independent samples of size $n \in \{10, 20, 40, 80\}$ drawn from a bivariate copula $C$ with a Kendall's tau of $\tau \in \{0, 0.25, 0.5, 0.75\}$. For $C$, we considered either the Clayton or the Gumbel--Hougaard copula. Next, for each considered combination of $C$, $n$, $\tau$, target process and multiplier process, we generated 1000 samples from $C$, and, for each sample, we computed $B=1000$ multiplier bootstrap replicates. These $B=1000$ replicates were used to obtain one estimate of the covariance of the target process at the points in~$P$.

\begin{figure}[t!]
\begin{center}
  \includegraphics*[width=1\linewidth]{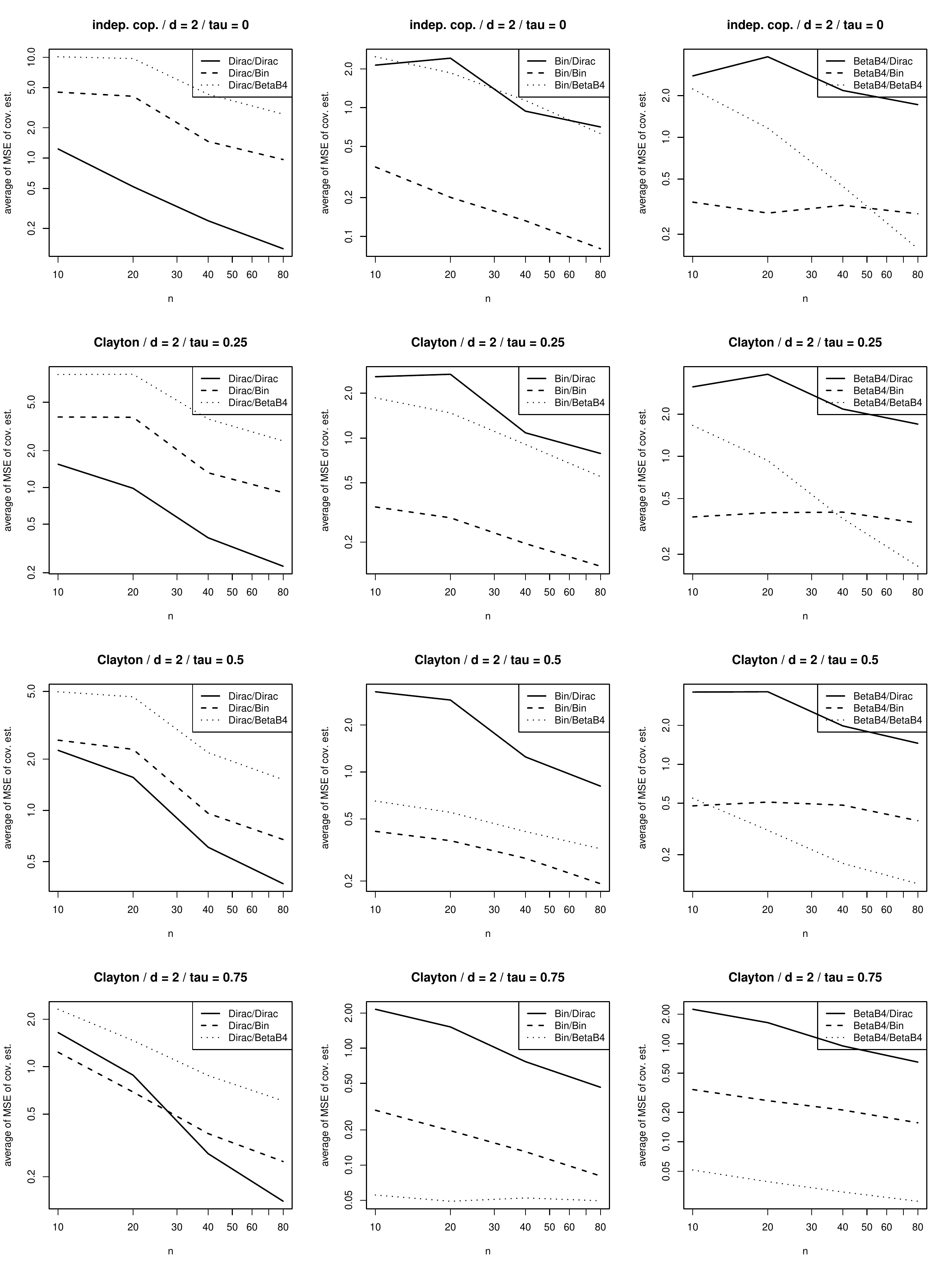}
  \caption{\label{fig:cov} For observations generated from the bivariate Clayton copula with a Kendall's tau of $\tau \in \{0,0.25,0.5,0.75\}$ and for each combination of target and multiplier process, average of the empirical MSEs ($\times 10^4$) of the bootstrap estimators of the covariance of the target process at the points in $P$ against the sample size $n$. The legend ``Dirac/Bin'' for instance refers to the situation when the target process is $\Cb_n^{\sss \Dirac}$ and the multiplier process is $\hat \Cb_n^{\sss [b],\Bin}$.}
\end{center}
\end{figure}

The results when $C$ is the Clayton copula with a Kendall's tau of $\tau \in \{0, 0.25, 0.5, 0.75\}$ are reported in Figure~\ref{fig:cov}. The first (resp.\ second, third) column of graphs reports the average of the empirical mean square errors (MSEs) $\times 10^4$ of the three candidate multiplier estimators of the covariance of $\Cb_n^{\sss \Dirac}$ (resp.\ $\Cb_n^{\sss \Bin}$, $\Cb_n^{\sss \BetaB}$) at the points in $P$ against the sample size~$n$. Each row of graphs corresponds to a different value of $\tau$. In the top-left panel for instance, the solid (resp.\ dashed, dotted) curve gives the average MSE when the covariance of $\Cb_n^{\sss \Dirac}$ is estimated using $\hat \Cb_n^{\sss [b], \Dirac}$ (resp.\ $\hat \Cb_n^{\sss [b],\Bin}$, $\hat \Cb_n^{\sss [b],\BetaB}$).

As one can see, reassuringly, all the curves are globally decreasing, confirming that, for each target process, the bootstrap approximations improve as $n$ increases. A more careful inspection reveals that, in almost all settings, it is the multiplier bootstrap constructed with the same smoothing distributions as the target process that leads to the best estimation. It is actually only when $\Cb_n^{\sss \BetaB}$ is the target process that covariance estimations based on $\hat \Cb_n^{\sss [b],\Bin}$ are sometimes better than estimations based on $\hat \Cb_n^{\sss [b],\BetaB}$. This happens mostly for small $n$ and $\tau$. Results for the Gumbel--Hougaard copula (not reported) are not qualitatively different.

In a second experiment, we assessed how well high quantiles of
\begin{equation}
  \label{eq:KS:CvM}
  KS(f_n) = \sup_{\bm u \in [0,1]^d} |f_n(\bm u)| \qquad \text{and} \qquad
  CvM(f_n) = \int_{[0,1]^d} \{ f_n(\bm u) \}^2 \dd \bm u
\end{equation}
for $d \in \{2,3\}$ and $f_n \in \{\Cb_n^{\sss \Dirac}, \Cb_n^{\sss \Bin}, \Cb_n^{\sss \BetaB} \}$ can be estimated by the three candidate multiplier bootstraps. From a practical perspective, the integral in~\eqref{eq:KS:CvM} was approximated by a mean using a uniform grid on $(0,1)^d$ of size $10^2$ when $d=2$ and $5^3$ when $d=3$. For $d \in \{2,3\}$, $C$ the Clayton or the Gumbel-Hougaard copula whose bivariate margins have a Kendall's tau of $\tau \in \{0, 0.25, 0.5, 0.75\}$ and $n \in \{10, 20, 40, 80\}$, the 90\% and 95\%-quantiles of $CvM(f_n)$ were first precisely estimated from $100\,000$ independent samples of size $n$ drawn from~$C$. Next, for each combination of $d$, $C$, $n$, $\tau$, target process and multiplier process, we generated 1000 samples from $C$ and, for each sample, we computed $B=1000$ multiplier bootstrap replicates. These $B=1000$ replicates were used to obtain one estimate of each of the target quantiles. Following \cite{KojSte19}, all such estimations were carried out using \emph{centered} replicates of $f_n$. When $f_n = \Cb_n^{\sss \Dirac}$ for instance, this amounts to using, for any $\bm u \in [0,1]^d$ and $b \in \{1,\dots,B\}$,
\begin{equation*}
\hat \Cb_n^{\sss [b],\Dirac}(\bm u) - \frac{1}{B} \sum_{b=1}^B \hat \Cb_n^{\sss [b],\Dirac}(\bm u),
\end{equation*}
instead of $\hat \Cb_n^{\sss [b],\Dirac}(\bm u) = \hat \Cb_n^{\sss [b]}(0,1,\bm u)$ in~\eqref{eq:Cb:n:hat}. The centered versions of the other replicates are defined analogously. The rationale behind centering is that the replicates, whatever their type, can be regarded as computable approximations of the limiting \emph{centered} Gaussian process $\Cb_C(0,1,\cdot)$ in~\eqref{eq:Cb:C}; see, for instance, Theorem~\ref{thm:ae:Cb:n:hat:check}. Note that the use of centered replicates was found to always lead to better finite-sample performance in the related Monte Carlo experiments carried out in \cite{KojSte19}. Its use is however irrelevant in the previous covariance estimation experiment given the formula of the empirical covariance.

\begin{figure}[t!]
\begin{center}
  \includegraphics*[width=1\linewidth]{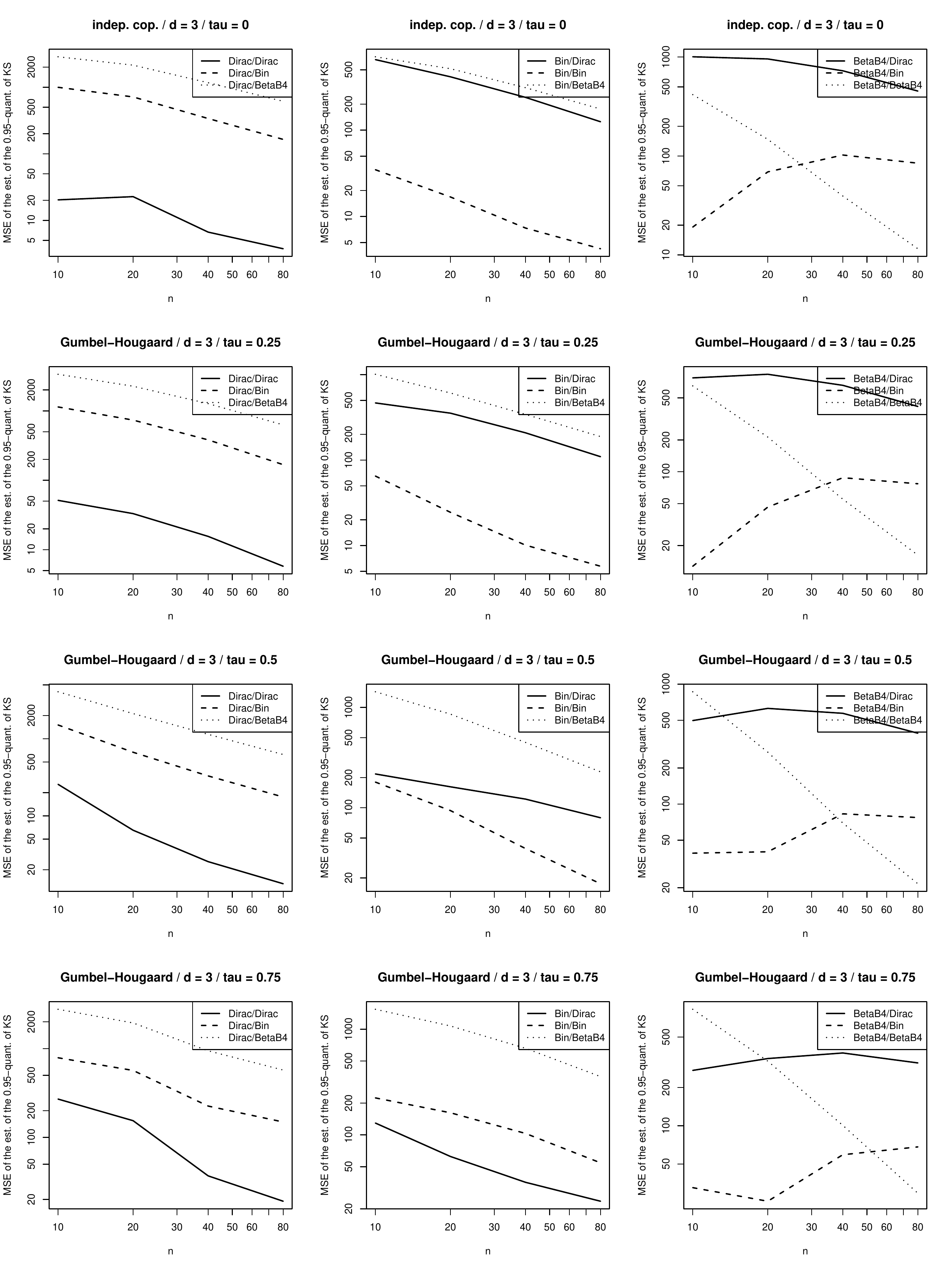}
  \caption{\label{fig:quant} For observations generated from the trivariate Gumbel--Hougaard copula whose bivariate margins have a Kendall's tau of $\tau \in \{0,0.25,0.5,0.75\}$, empirical MSE ($\times 10^4$) of the three candidate multiplier estimators of high quantiles of $KS(f_n)$ in~\eqref{eq:KS:CvM} for $f_n \in \{\Cb_n^{\sss \Dirac}, \Cb_n^{\sss \Bin}, \Cb_n^{\sss \BetaB} \}$ against the sample size $n$. The legend ``Dirac/Bin'' for instance refers to the situation when the target process is $\Cb_n^{\sss \Dirac}$ and the multiplier process is $\hat \Cb_n^{\sss [b],\Bin}$.}
\end{center}
\end{figure}

The results for the 95\%-quantiles of the Kolmogorov--Smirnov functionals when $C$ is the trivariate Gumbel--Hougaard are reported in Figure~\ref{fig:quant}. The conclusions are overall similar to those obtained after the first experiment:
\begin{itemize}
\item  The 95\%-quantile of the Kolmogorov--Smirnov functional of $\Cb_n^{\sss \Dirac}$ is always best estimated using the corresponding empirical quantile of the same functional of $\hat \Cb_n^{\sss [b],\Dirac}$.
\item When the target process is $\Cb_n^{\sss \Bin}$, the best results are obtained when the multiplier process is $\hat \Cb_n^{\sss [b],\Bin}$, except in the case of strongly dependent observations in which case, for the sample sizes under consideration, $\hat \Cb_n^{\sss [b],\Dirac}$ gives better estimations.
\item When the target process is $\Cb_n^{\sss \BetaB}$, it is only when $n$ reaches 40 or 80 that the best estimations are obtained using $\hat \Cb_n^{\sss [b],\BetaB}$. For smaller $n$, the use of $\hat \Cb_n^{\sss [b],\Bin}$ gives better results. 
\end{itemize}
Results for the Clayton copula, 90\%-quantiles, dimension $d=2$ or Cramér--von Mises functionals (not reported) are not qualitatively different.

The previous experiments confirm that it seems meaningful to resample $\Cb_n^\nu$ in~\eqref{eq:Cb:n:nu} using multiplier replicates constructed with the same smoothing distributions, that is, with $\hat{\Cb}_n^{\sss [b], \nu}$ in~\eqref{eq:Cb:n:hat:nu} or $\check{\Cb}_n^{\sss [b], \nu}$ in~\eqref{eq:Cb:n:check:nu}, although this choice may not be optimal in certain cases when $n$ is small.

\subsection{Application to change-point detection}
\label{sec:cp}

A natural application area for the smooth sequential empirical copula process $\Cb_n^\nu$ in \eqref{eq:Cb:n:nu} is that of change-point detection. To illustrate the possible advantages coming from the use of smooth empirical copulas in inference procedures, we first briefly explain in this section how the previous derivations can be used to obtain a smooth version of the test proposed in \cite{BucKojRohSeg14} for detecting changes in the cross-sectional dependence of multivariate time series. We then reproduce some of the experiments of \cite{BucKojRohSeg14} to compare the (non-smooth) test proposed therein with its smooth version based on the empirical beta copula and on corresponding smooth bootstrap replicates. Note that we did not consider the use of the alternative data-adaptive smoothing distributions considered in \cite{KojYi22} and leading to the estimator $C_{k:l}^{\sss \BetaB}$ because they incur a substantially higher computational cost.

The null hypothesis of such tests is that $\Xc_{1:n}$ is a stretch from a stationary time series (of continuous random vectors) and their aim is to be particularly sensitive to the alternative hypothesis
\begin{equation}   
   \begin{split}
   \label{eq:H1}
   H_1 :\;  & \exists \text{ distinct } C_1, \; C_2 \text{ and } k^\star \in \{1,\dots,n-1\} \text{ such that } \\ &\bm X_1, \dots, \bm X_{k^\star} \text{ have copula } C_1 \text{ and } \bm X_{k^\star+1}, \dots, \bm X_n \text{ have copula } C_2.
  \end{split}
\end{equation}

The ingredients of the smooth version of the test can be obtained \emph{mutatis mutandis} from \cite{BucKojRohSeg14}. Specifically, we consider as test statistic the maximally selected Cram\'er--von Mises functional defined by
$$
S_n^\nu = \sup_{s \in [0,1]} \int_{[0,1]^d} \left\{\Db_n^\nu(s, \bm u) \right\}^2 \dd C_{1:n}(\bm u),
$$
where
$$
\Db_n^\nu (s,\bm u) = \sqrt{n} \lambda_n(0,s) \lambda_n(s,1)\{C_{1:\ip{ns}}^\nu(\bm u) - C_{\ip{ns}+1:n}^\nu(\bm u)\}, \qquad (s,\bm u) \in [0,1]^{d+1}.
$$
As one can see, the latter involves comparisons of (smooth) empirical copulas computed from subsamples of the data. Noticing that, under the null,
$$
\Db_n^\nu (s,\bm u) = \lambda_n(s,1)\, \Cb_n^\nu(0,s,\bm u) - \lambda_n(0,s) \Cb_n^\nu(s,1,\bm u), \qquad (s,\bm u) \in [0,1]^{d+1},
$$
possible multiplier bootstrap replicates for $S_n^\nu$ can be defined either by
$$
\hat S_n^{\sss [b],\nu} = \sup_{s \in [0,1]} \int_{[0,1]^d} \{\hat \Db_n^{\sss [b],\nu}(s, \bm u) \}^2 \dd C_{1:n}(\bm u), \qquad b \in \N,
$$
or by
\begin{equation}
  \label{eq:check:mult}
  \check S_n^{\sss [b],\nu} = \sup_{s \in [0,1]} \int_{[0,1]^d} \{\check \Db_n^{\sss [b],\nu}(s, \bm u) \}^2 \dd C_{1:n}(\bm u), \qquad b \in \N,
\end{equation}
where, for any $(s,\bm u) \in [0,1]^{d+1}$,
\begin{align*}
  \hat \Db_n^{\sss [b],\nu} (s,\bm u) &= \lambda_n(s,1)\, \hat \Cb_n^{\sss [b],\nu}(0,s,\bm u) - \lambda_n(0,s)\, \hat \Cb_n^{\sss [b],\nu}(s,1,\bm u),\\
  \check \Db_n^{\sss [b],\nu} (s,\bm u) &= \lambda_n(s,1)\, \check \Cb_n^{\sss [b],\nu}(0,s,\bm u) - \lambda_n(0,s)\, \check \Cb_n^{\sss [b],\nu}(s,1,\bm u),
\end{align*}
with $\hat \Cb_n^{\sss [b],\nu}$ and $\check \Cb_n^{\sss [b],\nu}$ defined in~\eqref{eq:Cb:n:hat:nu} and~\eqref{eq:Cb:n:check:nu}, respectively. Note that, in the expressions of the multiplier replicates of $\Cb_n^\nu$, as estimators of the first-order partial derivatives of the copula, we use the ``truncated'' finite-difference based estimators defined in~\eqref{eq:pd:est:nu:Delta:trunc} of the forthcoming section with bandwidths $h = h' = \min\{(l-k+1)^{-1/2}, 1/2\}$. As we will see from Proposition~\ref{prop:wc:nu:Delta}, the latter can satisfy Condition~\ref{cond:pd:est}. Finally, as in \cite{BucKojRohSeg14}, approximate p-values for $S_n^\nu$ can be computed via either
$$
\frac{1}{B} \sum_{b=1}^B \1 \Big(\hat S_n^{\sss [b],\nu} \geq S_n^\nu \Big) \quad \text{or} \quad \frac{1}{B} \sum_{b=1}^B \1 \Big(\check S_n^{\sss [b],\nu} \geq S_n^\nu \Big),
$$
for some large integer $B$. Theoretical results confirming that the above way of proceeding is asymptotically valid under the null can be obtained by starting from Theorem~\ref{thm:ae:Cb:n:hat:check}, proceeding as in \cite{BucKojRohSeg14} and finally using results stated in Section~4 of \cite{BucKoj19}.

If, for any $m \in \N$, the underlying smoothing distributions $\nu_{\bm u}^{\bm x}$, $\bm x \in (\R^d)^m$,  $\bm u \in [0,1]^d$, are Dirac measures at $\bm u$, the previous ingredients are non-smooth and the resulting test coincides exactly with the test studied in \cite{BucKojRohSeg14}. The test statistic will naturally be denoted by $S_n^{\sss \Dirac}$ in that case. As alternative smoothing distributions, we considered those leading to the empirical beta copula and specified in Remark~\ref{rem:beta} as well as at the end of Section~\ref{sec:ec}. The resulting statistic will then naturally be denoted by $S_n^{\sss \Bin}$. 

To compare the test based on $S_n^{\sss \Bin}$ to the test based on $S_n^{\sss \Dirac}$, we considered experiments similar to those reported in Section~5 of \cite{BucKojRohSeg14}. Both tests were carried out at the 5\% significance level using replicates of the form~\eqref{eq:check:mult} as these seemed to lead to better results. The dependent multiplier sequences necessary to carry out the tests were generated as explained in the last paragraph of Appendix~C of \cite{BucKojRohSeg14}. The value of the bandwidth parameter $\ell_n$ appearing in (M2) and (M3) in Section~\ref{sec:mult:seq} was chosen using the procedure described in \citet[Section~5]{BucKoj16} (although this way of proceeding may not be ``optimal'' for the smooth multiplier bootstrap replicates).

\begin{table}[t!]
\centering
\caption{Percentages of rejection of the null hypothesis of stationarity computed from 1000 samples of size $n \in \{25,50,100,200\}$ generated as explained in Section~\ref{sec:cp}, where $C$ is the bivariate Frank copula with a Kendall's tau of $\tau \in \{0,0.33,0.66\}$.} 
\label{tab:H0}
\begingroup\small
\begin{tabular}{lrrrrrrrrr}
  \hline
  \multicolumn{2}{c}{} & \multicolumn{2}{c}{$n=25$} & \multicolumn{2}{c}{$n=50$} & \multicolumn{2}{c}{$n=100$} & \multicolumn{2}{c}{$n=200$} \\ \cmidrule(lr){3-4} \cmidrule(lr){5-6} \cmidrule(lr){7-8} \cmidrule(lr){9-10} $\beta$ & $\tau$ & $S_n^{\sss \Dirac}$ & $S_n^{\sss \Bin}$ & $S_n^{\sss \Dirac}$ & $S_n^{\sss \Bin}$ & $S_n^{\sss \Dirac}$ & $S_n^{\sss \Bin}$ & $S_n^{\sss \Dirac}$ & $S_n^{\sss \Bin}$ \\ \hline
0 & 0.00 & 17.5 & 13.3 & 7.7 & 8.0 & 5.5 & 5.8 & 3.8 & 4.4 \\ 
   & 0.33 & 18.7 & 13.4 & 7.6 & 7.3 & 4.9 & 6.3 & 4.2 & 4.0 \\ 
   & 0.66 & 21.1 & 11.7 & 5.6 & 4.9 & 3.0 & 3.1 & 3.2 & 3.8 \\ 
  0.3 & 0.00 & 18.8 & 16.1 & 6.2 & 7.4 & 4.3 & 4.7 & 6.4 & 6.0 \\ 
   & 0.33 & 21.4 & 16.4 & 7.8 & 8.7 & 5.2 & 5.9 & 5.4 & 5.4 \\ 
   & 0.66 & 25.3 & 16.8 & 5.4 & 5.9 & 2.1 & 3.0 & 1.2 & 1.4 \\ 
  0.5 & 0.00 & 26.1 & 22.8 & 11.4 & 11.7 & 6.1 & 6.6 & 6.2 & 7.2 \\ 
   & 0.33 & 22.9 & 23.0 & 10.3 & 11.2 & 5.5 & 7.2 & 2.4 & 3.6 \\ 
   & 0.66 & 27.5 & 20.1 & 10.5 & 11.0 & 2.2 & 3.6 & 1.6 & 1.6 \\ 
   \hline
\end{tabular}
\endgroup
\end{table}

As a first experiment, we estimated the percentages of rejection of the null hypothesis of stationarity for data generated under the null. As data generating model, we used a bivariate AR(1) model. Specifically, let $\bm U_i$, $i \in \{ -100, \dots, n \}$, be a bivariate i.i.d.\ sample from a copula~$C$. Then, set $\bm \epsilon_i = (\Phi^{-1}(U_{i1}),\Phi^{-1}(U_{i2}))$, where $\Phi$ is the d.f.\ of the standard normal distribution, and $\bm X_{-100} = \bm \epsilon_{-100}$. Finally, for any $j \in \{ 1, 2\}$ and $i \in \{ -99, \dots, n \}$, compute recursively
$$
X_{ij} = \beta X_{i-1,j} + \epsilon_{ij},
$$
where the first 100 observations are used as a burn-out sample. 

We considered $n \in \{25, 50, 100, 200\}$, $C$ to be bivariate Frank copula with a Kendall's tau of $\tau \in \{0, 0.33, 0.66\}$ and $\beta \in \{0,0.3,0.5\}$. The corresponding rejection percentages are reported in Table~\ref{tab:H0}. As one can see, both tests appear to hold their level reasonably well when $n \in \{100, 200\}$. The tests should however clearly not be used when $n = 25$ but might be employed when $n=50$ in the case of weakly serially dependent data.

\begin{table}[t!]
\centering
\caption{Percentages of rejection of the null hypothesis of stationarity computed from 1000 samples of size $n \in \{50,100,200\}$ generated under $H_1$ as explained in Section~\ref{sec:cp}, where $k^\star = \ip{nt}$, $C_1$ and $C_2$ are both bivariate Frank copulas such that $C_1$ has a Kendall's tau of 0.2 and $C_2$ a Kendall's tau of $\tau \in \{0.4,0.6\}$.} 
\label{tab:H1}
\begingroup\small
\begin{tabular}{rrrrrrr}
  \hline
  \multicolumn{3}{c}{} & \multicolumn{2}{c}{$\beta=0$} & \multicolumn{2}{c}{$\beta=0.3$}  \\ \cmidrule(lr){4-5} \cmidrule(lr){6-7} $\tau$ & $n$ & $t$ & $S_n^{\sss \Dirac}$ & $S_n^{\sss \Bin}$ & $S_n^{\sss \Dirac}$ & $S_n^{\sss \Bin}$  \\ \hline
0.4 & 50 & 0.10 & 8.8 & 8.7 & 8.0 & 8.1 \\ 
   &  & 0.25 & 13.5 & 16.1 & 14.0 & 15.5 \\ 
   &  & 0.50 & 14.7 & 15.3 & 17.5 & 18.4 \\ 
   & 100 & 0.10 & 4.0 & 4.9 & 5.5 & 7.6 \\ 
   &  & 0.25 & 16.9 & 19.3 & 14.8 & 17.9 \\ 
   &  & 0.50 & 26.6 & 28.8 & 22.5 & 25.3 \\ 
   & 200 & 0.10 & 6.6 & 7.4 & 5.6 & 6.6 \\ 
   &  & 0.25 & 29.4 & 31.8 & 22.0 & 24.2 \\ 
   &  & 0.50 & 51.4 & 53.8 & 42.0 & 43.8 \\ 
  0.6 & 50 & 0.10 & 10.2 & 13.0 & 9.1 & 11.6 \\ 
   &  & 0.25 & 33.0 & 39.8 & 31.6 & 39.1 \\ 
   &  & 0.50 & 53.0 & 56.8 & 47.0 & 51.1 \\ 
   & 100 & 0.10 & 12.1 & 16.6 & 8.6 & 12.5 \\ 
   &  & 0.25 & 62.6 & 70.9 & 51.9 & 60.3 \\ 
   &  & 0.50 & 83.1 & 84.9 & 75.0 & 78.6 \\ 
   & 200 & 0.10 & 30.4 & 37.8 & 21.0 & 28.8 \\ 
   &  & 0.25 & 95.2 & 97.0 & 87.8 & 91.0 \\ 
   &  & 0.50 & 99.4 & 99.4 & 97.0 & 97.2 \\ 
   \hline
\end{tabular}
\endgroup
\end{table}

As a second experiment, we estimated rejection percentages of the null hypothesis of stationarity for data generated under $H_1$ in~\eqref{eq:H1}. To do so, we considered a similar data generating model as in the first experiment except that the $\bm U_i$'s for $i \in \{ -100, \dots, k^\star \}$ are i.i.d.\ from a copula $C_1$ while the $\bm U_i$'s for $i \in \{k^\star + 1,\dots, n \}$ are i.i.d.\ from a copula $C_2 \neq C_1$. Following \cite{BucKojRohSeg14}, we set $k^\star = \ip{nt}$ with $t \in \{0.1,0.25,0.5\}$ and considered $n \in \{50, 100, 200\}$, $C_1$ the bivariate Frank copula with a Kendall's tau $0.2$ and $C_2$ the bivariate Frank copula with a Kendall's tau in $\{ 0.4, 0.6\}$. The results are reported in Table~\ref{tab:H1}. As one can see, the test based on $S_n^{\sss \Bin}$ appears overall to be more powerful than the one based on $S_n^{\sss \Dirac}$. The largest differences in power tend to occur for $\tau=0.6$ and $t \in \{0.1,0.25\}$ which corresponds to the situation when the test statistic should be the largest because of a difference between an empirical copula computed from a small number of observations (approximately $\ip{nt}$) and an empirical copula computed from the remaining observations. While one cannot conclude that smooth change-point detection tests such as the one based on $S_n^{\sss \Bin}$ will be more powerful than the non-smooth test based on $S_n^{\sss \Dirac}$ in all situations, the obtained results confirm in part the intuition that smooth tests might be more sensitive to changes at the beginning or at the end of the data sequence.

\section{Estimators of the first-order partial derivatives of the copula}
\label{sec:pd}

The multiplier bootstrap replicates defined in the previous section all depend on the choice of estimators of the first-order partial derivatives of $C$. For asymptotic reasons, the latter were required to satisfy Condition~\ref{cond:pd:est}. Obviously, the more accurate such estimators, the better we can expect the multiplier bootstraps to behave, whether they involve smoothing or not. After recalling existing definitions of such estimators based on finite differences of the classical empirical copula, we define two related classes of smooth estimators. Then, upon an appropriate choice of the underlying bandwidth parameters, we establish their weak consistency in a sequential setting which implies that many of the considered estimators satisfy Condition~\ref{cond:pd:est}. In the last subsection, we report the results of bivariate and trivariate Monte Carlo experiments comparing selected estimators in terms of integrated mean squared error.

Note that, as already mentioned in the introduction, the results of this section can be of independent interest since, as discussed for instance in \cite{JanSwaVer16}, estimators of the first-order partial derivatives of a copula have applications in mean and quantile regression as they lead to estimators of the conditional distribution function. In particular, as we shall see, several estimators considered in our Monte Carlo experiments display a better finite-sample performance than the Bernstein estimator studied in \cite{JanSwaVer16}.

\subsection{Estimators based on finite differences of the empirical copula}

As already mentioned in Section~\ref{sec:non-smooth:mult}, in their seminal work on the multiplier bootstrap for the classical empirical copula process, \cite{RemSca09} considered estimators of the first-order partial derivatives $\dot C_j$, $j \in \{1,\dots,d\}$, of $C$ based on finite-differences of the empirical copula. In a sequential context, given a stretch $\Xc_{k:l} = (\bm X_k,\dots,\bm X_l)$, $1 \leq k \leq l \leq n$, of observations and two bandwidth parameters $h$ and $h'$ in $[0,1/2]$ such that $h + h' > 0$, a slightly more general definition of the aforementioned estimators is
\begin{equation}
  \label{eq:pd:est:nabla}
\dot C_{j,k:l,h,h'}^{\sss \nabla}(\bm u) = \frac{C_{k:l}\{(\bm u + h\bm e_j) \wedge \bm 1 \} - C_{k:l} \{ (\bm u - h' \bm e_j) \vee \bm 0 \} }{h+h'}, \qquad \bm u \in [0,1]^d,
\end{equation}
where $\bm e_j$ is the $j$th vector of the canonical basis of $\R^d$, $\bm 0 = (0,\dots,0)$, $\bm 1 = (1, \dots, 1) \in \R^d$, $\wedge$ (resp.\ $\vee$) denotes the minimum (resp.\ maximum) componentwise operator and $C_{k:l}$ is the classical empirical copula of $\Xc_{k:l}$ defined in~\eqref{eq:C:kl}. The symbol $\nabla$ indicates that the estimators are based on finite-differences of $C_{k:l}$ with ``right'' (resp.\ ``left'') bandwidth $h$ (resp.\ $h'$).

In order to reduce the bias of the previous estimator for evaluation points $\bm u \in [0,1]$ with $u_j \in [0,h') \cup (1-h,1]$, \cite{KojSegYan11} considered the following minor variation of~\eqref{eq:pd:est:nabla}:
\begin{equation}
  \label{eq:pd:est:Delta}
\dot C_{j,k:l,h,h'}^{\sss \Delta}(\bm u) = \frac{C_{k:l}\{(\bm u + h\bm e_j) \wedge \bm 1 \} - C_{k:l} \{ (\bm u - h' \bm e_j) \vee \bm 0 \} }{(u_j+h) \wedge 1 - (u_j - h') \vee  0}, \qquad \bm u \in [0,1]^d.
\end{equation}
Note the use of the symbol $\Delta$ still referring to finite-differences but upside-down compared to $\nabla$ to distinguish~\eqref{eq:pd:est:Delta} from~\eqref{eq:pd:est:nabla}. 

As is well known, in general, $\dot C_j$ exists almost everywhere on $[0,1]^d$ and, for those $\bm u \in [0,1]^d$ for which it exists, $0 \leq \dot C_j(\bm u) \leq 1$ \citep[see e.g.,][Theorem 2.2.7]{Nel06}. A natural modification of the estimators $\dot C_{j,k:l,h,h'}^{\sss \nabla}$ in~\eqref{eq:pd:est:nabla} and $\dot C_{j,k:l,h,h'}^{\sss \Delta}$ in~\eqref{eq:pd:est:Delta} thus consists of ensuring that they take their values in $[0,1]$ by truncating them:
\begin{align}
  \label{eq:pd:est:nabla:trunc}
  \dotu{C}_{j,k:l,h,h'}^{\sss \nabla} = (\dot C_{j,k:l,h,h'}^{\sss \nabla} \vee 0) \wedge 1,\\
  \label{eq:pd:est:Delta:trunc}
  \dotu{C}_{j,k:l,h,h'}^{\sss \Delta} = (\dot C_{j,k:l,h,h'}^{\sss \Delta} \vee 0) \wedge 1.
\end{align}
Notice that taking the maximum with 0 in the previous expressions is actually not necessary as the estimators in~\eqref{eq:pd:est:nabla} and~\eqref{eq:pd:est:Delta} cannot be negative since the empirical copula $C_{k:l}$ is a multivariate d.f. We nonetheless keep~\eqref{eq:pd:est:nabla:trunc} and~\eqref{eq:pd:est:Delta:trunc}  as they are to be consistent with certain forthcoming definitions. More generally, in the rest of this section, underlining will be used to denote estimators constrained to take their values in $[0,1]$.

\subsection{Two classes of smooth estimators}

To obtain smooth estimators of the first-order partial derivatives of $C$, the proposals in~\eqref{eq:pd:est:nabla} and~\eqref{eq:pd:est:Delta} can be extended in two natural ways. The first approach consists of considering finite-differences of smooth estimators of $C$. Given a stretch $\Xc_{k:l}$, $1 \leq k \leq l \leq n$, of observations and two bandwidth parameters $h$ and $h'$ in $[0,1/2]$ such that $h + h' > 0$, this leads to the estimators
\begin{align}
  \label{eq:pd:est:nu:nabla}
  \dot C_{j,k:l,h,h'}^{\sss \nu,\nabla}(\bm u) = \frac{C_{k:l}^\nu\{(\bm u + h\bm e_j) \wedge \bm 1 \} - C_{k:l}^\nu\{ (\bm u - h' \bm e_j) \vee \bm 0 \} }{h+h'}, \qquad \bm u \in [0,1]^d, \\
  \label{eq:pd:est:nu:Delta}
  \dot C_{j,k:l,h,h'}^{\sss \nu,\Delta}(\bm u) = \frac{C_{k:l}^\nu\{(\bm u + h\bm e_j) \wedge \bm 1 \} - C_{k:l}^\nu\{ (\bm u - h' \bm e_j) \vee \bm 0 \} }{(u_j+h) \wedge 1 - (u_j - h') \vee  0}, \qquad \bm u \in [0,1]^d,
\end{align}
where $C_{k:l}^\nu$ is the smooth empirical copula of $\Xc_{k:l}$ defined in~\eqref{eq:C:kl:nu}. Notice the order of the symbols $\nu$ and $\nabla$ (resp.~$\Delta$) indicating that the empirical copula is first smoothed before finite-differencing is applied. Clearly, \eqref{eq:pd:est:nabla} (resp.~\eqref{eq:pd:est:Delta}) is a particular case of~\eqref{eq:pd:est:nu:nabla} (resp.~\eqref{eq:pd:est:nu:Delta}) when the smoothing distributions $\nu_{\bm u}^{\sss \Xc_{k:l}}$, $\bm u \in [0,1]^d$,  in~\eqref{eq:C:kl:nu} are chosen to be Dirac measures at $\bm u \in [0,1]^d$.

\begin{remark}
  \label{rem:constraint}
  Since $C$ is a multivariate d.f.\ with standard uniform margins, we have that, for any $j \in \{1,\dots,d\}$ and $\bm u \in V_j$ (where $V_j$ is defined in Condition~\ref{cond:pd}), $\dot C_j(\bm u^{(j)}) = \lim_{h \to 0} \{ C(\bm u^{(j)} + h \bm e_j) - C(\bm u^{(j)}) \} / h = 1$ (where the notation $\bm u^{(j)}$ is defined above Theorem~\ref{thm:wc:Cb:n}). Interestingly enough, the estimator $\dot C_{j,k:l,h,h'}^{\sss \nu,\Delta}$ in~\eqref{eq:pd:est:nu:Delta} can satisfy this boundary constraint, that is, we can have $\dot C_{j,k:l,h,h'}^{\sss \nu,\Delta}(\bm u^{(j)}) = 1$. This will indeed happen if $C_{k:l}^\nu$ is a genuine copula, which according to Proposition~\ref{prop:genuine:copula}, can occur under Condition~\ref{cond:no:ties} for specific choices of the smoothing distributions in~\eqref{eq:C:kl:nu} such as those leading to the empirical copulas $C_{k:l}^{\sss \Bin}$ or $C_{k:l}^{\sss \BetaB}$ defined in Section~\ref{sec:ec}.
\end{remark}

By analogy with~\eqref{eq:pd:est:nabla:trunc} and~\eqref{eq:pd:est:Delta:trunc}, it is straightforward to define truncated versions of the estimators in~\eqref{eq:pd:est:nu:nabla} and~\eqref{eq:pd:est:nu:Delta} by
\begin{align}
  \label{eq:pd:est:nu:nabla:trunc}
  \dotu{C}_{j,k:l,h,h'}^{\sss \nu,\nabla} = (\dot C_{j,k:l,h,h'}^{\sss \nu,\nabla} \vee 0) \wedge 1,\\
  \label{eq:pd:est:nu:Delta:trunc}
  \dotu{C}_{j,k:l,h,h'}^{\sss \nu,\Delta} = (\dot C_{j,k:l,h,h'}^{\sss \nu,\Delta} \vee 0) \wedge 1.
\end{align}

\begin{remark}
  As discussed in Remark~\ref{rem:constraint}, the smoothing distributions $\nu_{\bm u}^{\sss \Xc_{k:l}}$, $\bm u \in [0,1]^d$, can be chosen such that $C_{k:l}^\nu$ is a genuine copula under Condition~\ref{cond:no:ties}. In that case, using the fact that $C_{k:l}^\nu$ is a multivariate d.f.\ with standard uniform margins, we immediately obtain \cite[see, e.g.,][Lemma~1.2.14]{DurSem15} that, for any $\bm u \in [0,1]^d$ and $h,h'$ in $[0,1/2]$ such that $h+h'>0$,
  $$
  C_{k:l}^\nu\{(\bm u + h\bm e_j) \wedge \bm 1 \} - C_{k:l}^\nu\{ (\bm u - h' \bm e_j) \vee \bm 0 \}  \leq (u_j+h) \wedge 1 - (u_j - h') \vee  0 ,
  $$
  which implies that $0 \leq \dot C_{j,k:l,h,h'}^{\sss \nu,\Delta} \leq 1$ and thus that truncation of $\dot C_{j,k:l,h,h'}^{\sss \nu,\Delta}$ is not necessary in that case since $\dotu{C}_{j,k:l,h,h'}^{\sss \nu,\Delta}$ in~\eqref{eq:pd:est:nu:Delta:trunc} is equal to $\dot C_{j,k:l,h,h'}^{\sss \nu,\Delta}$. 
\end{remark}

By analogy with~\eqref{eq:C:kl:nu}, a second natural approach to obtain smooth partial derivative estimators consists of directly smoothing~\eqref{eq:pd:est:nabla} and~\eqref{eq:pd:est:Delta} and leads to the estimators
\begin{align}
  \label{eq:pd:est:nabla:nu}
\dot C_{j,k:l,h,h'}^{\sss \nabla,\nu}(\bm u) =  \int_{[0,1]^d} \dot C_{j,k:l,h,h'}^{\sss \nabla}(\bm w) \dd \nu_{\bm u}^{\sss \Xc_{k:l}} (\bm w), \qquad \bm u \in [0,1]^d, \\
  \label{eq:pd:est:Delta:nu}
\dot C_{j,k:l,h,h'}^{\sss \Delta,\nu}(\bm u) =  \int_{[0,1]^d} \dot C_{j,k:l,h,h'}^{\sss \Delta}(\bm w) \dd \nu_{\bm u}^{\sss \Xc_{k:l}} (\bm w),
\qquad \bm u \in [0,1]^d.
\end{align}
This time the order of the symbols $\nu$ and $\nabla$ (resp.~$\Delta$) is reversed indicating that it is the finite-differences-based estimator $\dot C_{j,k:l,h,h'}^{\sss \nabla}$ in~\eqref{eq:pd:est:nabla} (resp.~$\dot C_{j,k:l,h,h'}^{\sss \Delta}$ in~\eqref{eq:pd:est:Delta}) that is smoothed. Versions of these estimators that necessarily take their values in $[0,1]$ can be obtained by constructing them from the truncated estimators~\eqref{eq:pd:est:nabla:trunc} and \eqref{eq:pd:est:Delta:trunc} instead, leading respectively to
\begin{align}
  \label{eq:pd:est:nabla:nu:trunc}
\dotu{C}_{j,k:l,h,h'}^{\sss \nabla,\nu}(\bm u) =  \int_{[0,1]^d} \dotu{C}_{j,k:l,h,h'}^{\sss \nabla}(\bm w) \dd \nu_{\bm u}^{\sss \Xc_{k:l}} (\bm w), \qquad \bm u \in [0,1]^d, \\
  \label{eq:pd:est:Delta:nu:trunc}
\dotu{C}_{j,k:l,h,h'}^{\sss \Delta,\nu}(\bm u) =  \int_{[0,1]^d} \dotu{C}_{j,k:l,h,h'}^{\sss \Delta}(\bm w) \dd \nu_{\bm u}^{\sss \Xc_{k:l}} (\bm w),
\qquad \bm u \in [0,1]^d.
\end{align}

Note that a third approach to obtain a smooth estimator of the $j$th partial derivative $\dot C_j$ would consist of attempting to directly differentiate $C_{k:l}^\nu$ in~\eqref{eq:C:kl:nu} with respect to its $j$th argument (provided of course that $C_{k:l}^\nu$ is differentiable). The resulting estimator
\begin{equation*}
\dot C_{j,k:l}^\nu =  \frac{\partial C_{k:l}^\nu}{\partial u_j}
\end{equation*}
may exist only on the set $V_j$ defined in Condition~\ref{cond:pd}. This is the path followed by \cite{JanSwaVer16}, who, for some integer $m \geq 2$, started from the empirical Bernstein copula $C_{k:l,m}^{\sss \Bern}$ in~\eqref{eq:C:Bern:kl:m} (which, as discussed in Remark~\ref{rem:beta}, is a particular case of $C_{k:l}^\nu$ in~\eqref{eq:C:kl:nu}). Let $\dot C_{j,k:l,m}^{\sss \Bern} =  \partial C_{k:l,m}^{\sss \Bern} / \partial u_j$ be the resulting estimator. Interestingly enough, from Lemma~\ref{lem:Ber:pd} in Appendix~\ref{app:pd},  we have that
\begin{equation}
  \label{eq:Bern:pd}
  \dot C_{j,k:l,m}^{\sss \Bern}(\bm u) = \int_{[0,1]^d} \dot C_{j,k:l,\frac{1}{m},0}^{\sss \nabla}(\bm w) \dd \tilde \mu_{j,m,\bm u}(\bm w),  \qquad \bm u \in V_j,
\end{equation}
where $\dot C_{j,k:l,\frac{1}{m},0}^{\sss \nabla}$ is given by~\eqref{eq:pd:est:nabla} with $h = 1/m$ and $h'=0$ and, for any $\bm u \in [0,1]^d$, $\tilde \mu_{j,m,\bm u}$ is the law of the random vector $(\tilde S_{m,1,u_1}/m,\dots, \tilde S_{m,d,u_d}/m)$ whose components are independent such that, for $i \in \{1,\dots,d\} \setminus \{j\}$, $\tilde S_{m,i,u_i}$ is Binomial$(m,u_i)$ while $\tilde S_{m,j,u_j}$ is Binomial$(m-1,u_j)$. In other words, differentiating directly the empirical Bernstein copula $C_{k:l,m}^{\sss \Bern}$ in~\eqref{eq:C:Bern:kl:m} with respect to its $j$th argument leads to a special case of the estimator in~\eqref{eq:pd:est:nabla:nu}. Notice that, since the measures $\tilde \mu_{j,m,\bm u}$ are well-defined for any $\bm u \in [0,1]^d$, the integral in~\eqref{eq:Bern:pd} is actually well-defined for any $\bm u \in [0,1]^d$. Hence, as we continue, we take~\eqref{eq:Bern:pd} with $\bm u \in [0,1]^d$ as the definition of $\dot C_{j,k:l,m}^{\sss \Bern}$.

The following result, proven in Appendix~\ref{app:pd}, shows that $\dot C_{j,k:l,m}^{\sss \Bern}$ can be easily computed.

\begin{prop}
  \label{prop:Bern:pd:comp}
Given a stretch $\Xc_{k:l}$, $1 \leq k \leq l \leq n$, of observations, we have that, for any $j \in \{1,\dots,d\}$, $\bm u \in [0,1]^d$ and integer $m \geq 2$,
\begin{multline}
\label{eq:Bern:pd:comp}
\dot C_{j,k:l,m}^{\sss \Bern}(\bm u) = \frac{m}{l-k+1} \sum_{i=k}^l b_{m-1, u_j}\left\{\up{mR_{ij}^{k:l}/(l-k+1)} - 1\right\}  \\ \times \prod_{t=1 \atop t \neq j}^d \bar B_{m, u_t}\left\{\up{mR_{it}^{k:l}/(l-k+1)} - 1\right\},
\end{multline}
where $\up{\cdot}$ denotes the ceiling function and, for any $p \in \N$ and $u \in [0,1]$, $\bar B_{p, u}$ (resp.\ $b_{p, u}$) is the survival (resp.\ probability mass) function of the Binomial$(p,u)$.
\end{prop}

\subsection{Weak consistency}

In order to study the weak consistency of the estimators of the partial derivatives of $C$ defined in the previous subsection, it is necessary to link the bandwidth parameters in their expressions to the data (or, at least, to the amount of data) from which these estimators are computed. As we continue, for any $n \in \N$ and any potential $d$-dimensional data set $\bm x \in (\R^d)^n$, $h(\bm x)$ and $h'(\bm x)$ will denote the values of the left and right bandwidths for the data set $\bm x$. With this in mind, in the rest of this subsection, for the sake of a more compact notation, we shall write $\dot C_{j,k:l}^{\sss \nu,\nabla}$ (resp.\  $\dot C_{j,k:l}^{\sss \nu,\Delta}$, $\dot C_{j,k:l}^{\sss \nabla,\nu}$, \dots) for $\dot C_{j,k:l,h,h'}^{\sss \nu,\nabla}$ (resp.\ $\dot C_{j,k:l,h,h'}^{\sss \nu,\Delta}$, $\dot C_{j,k:l,h,h'}^{\sss \nabla,\nu}$, \dots) with the understanding that $h = h(\Xc_{k:l})$ and $h'= h'(\Xc_{k:l})$ are random variables. We impose in addition the following condition on the bandwidths.

\begin{cond}[Bandwidth condition]
 \label{cond:band}
There exists positive sequences $b_n \downarrow 0$ and $b_n' \downarrow 0$ and constants $L_2 \geq L_1 > 0$ such that, for all $n \in \N$, $b_n + b'_n  \geq n^{-1/2}$, and, for any $\bm x \in (\R^d)^n$, $L_1 b_n  \leq h(\bm x) \leq (L_2 b_n) \wedge 1/2$ and $L_1 b_n'  \leq h'(\bm x) \leq (L_2 b_n') \wedge 1/2$.
\end{cond}

As we shall see in Section~\ref{sec:MC:pd}, one meaningful possibility among many others is to consider that, for any $n \in \N$ and $\bm x \in \R^d$, the left and right bandwidths for the data set~$\bm x$ are defined by $h(\bm x) = h'(\bm x) = [M_2 \{1 - |\tau(\bm x)|\}^a + M_1] n^{-1/2} \wedge 1/2$, where $M_1,M_2 > 0$ are constants, $\tau(\bm x) \in [-1,1]$ is the value of the sample version of a suitable multivariate extension of Kendall's tau for the data set~$\bm x$ and $a \in (0,\infty)$ is a fixed power. Roughly speaking, the bandwidths will be larger (resp. smaller) in the case of weakly (resp.\ strongly) cross-sectionally dependent data. It is easy to verify that Condition~\ref{cond:band} holds for the previous definitions.

The following result, proven in Appendix~\ref{app:pd}, establishes the weak consistency of the smooth estimators of the first class in a sequential setting.

\begin{prop}[Weak consistency in a sequential setting for the first class of smooth estimators]
\label{prop:wc:nu:Delta}
Under Conditions~\ref{cond:pd},~\ref{cond:var:W},~\ref{cond:Bn} and~\ref{cond:band}, for any $j \in \{1,\dots,d\}$, $\delta \in (0,1)$ and $\eps \in (0,1/2)$,
\begin{equation}
  \label{eq:wc:nu:Delta}
  \sup_{\substack{(s,t) \in \Lambda \\ t-s \geq \delta}} \sup_{\substack{\bm u \in [0,1]^d\\ u_j \in [\eps, 1-\eps]}}  \left| \dot C_{j,\ip{ns}+1:\ip{nt}}^{\nu, \Delta}(\bm u) - \dot C_j(\bm u) \right| = o_\Pr(1),
\end{equation}
where $\dot C_{j,k:l}^{\nu, \Delta}$ is defined in~\eqref{eq:pd:est:nu:Delta}, and similarly for $\dot C_{j,k:l}^{\sss \nu,\nabla}$ in~\eqref{eq:pd:est:nu:nabla}, $\dotu{C}_{j,k:l}^{\sss \nu,\nabla}$ in~\eqref{eq:pd:est:nu:nabla:trunc} and $\dotu{C}_{j,k:l}^{\sss \nu,\Delta}$ in~\eqref{eq:pd:est:nu:Delta:trunc}.
\end{prop}

\begin{remark}
An inspection of the proof of the previous result reveals that the second supremum in~\eqref{eq:wc:nu:Delta} can be replaced by a supremum over $\bm u \in [0,1]^d$ if $\dot C_j$ happens to be continuous on $[0,1]^d$ instead of only satisfying Condition~\ref{cond:pd}; see also \cite{KojSegYan11}.
\end{remark}

As a consequence of the previous proposition, we have that, under the conditions of Proposition~\ref{prop:wc:nu:Delta}, the estimators $\dotu{C}_{j,k:l}^{\sss \nu,\nabla}$ and $\dotu{C}_{j,k:l}^{\sss \nu,\Delta}$ satisfy Condition~\ref{cond:pd:est} since they are bounded in absolute value (by one) by construction.

The next result is an immediate corollary of Proposition~\ref{prop:wc:nu:Delta} since the estimator in~\eqref{eq:pd:est:nabla} (resp.~\eqref{eq:pd:est:Delta}) is a particular case of the one in~\eqref{eq:pd:est:nu:nabla} (resp.~\eqref{eq:pd:est:nu:Delta}) when the smoothing distributions $\nu_{\bm u}^{\sss \Xc_{k:l}}$, $\bm u \in [0,1]^d$, in~\eqref{eq:C:kl:nu} are chosen to be Dirac measures at $\bm u \in [0,1]^d$ (the latter clearly satisfy Condition~\ref{cond:var:W}).

\begin{cor}[Weak consistency in a sequential setting for the non-smooth finite-differences-based estimators]
\label{cor:wc:Delta}
Under Conditions~\ref{cond:pd},~\ref{cond:Bn} and~\ref{cond:band}, for any $j \in \{1,\dots,d\}$, $\delta \in (0,1)$ and $\eps \in (0,1/2)$,
\begin{equation*}
  \sup_{\substack{(s,t) \in \Lambda \\ t-s \geq \delta}} \sup_{\substack{\bm u \in [0,1]^d\\ u_j \in [\eps, 1-\eps]}}  \left| \dot C_{j,\ip{ns}+1:\ip{nt}}^{\sss \Delta}(\bm u) - \dot C_j(\bm u) \right| = o_\Pr(1),
\end{equation*}
where $\dot C_{j,k:l}^{\sss \Delta}$ is defined in~\eqref{eq:pd:est:Delta}, and similarly for $\dot C_{j,k:l}^{\sss \nabla}$ in~\eqref{eq:pd:est:nabla}, $\dotu{C}_{j,k:l}^{\sss \nabla}$ in~\eqref{eq:pd:est:nabla:trunc} and $\dotu{C}_{j,k:l}^{\sss \Delta}$ in~\eqref{eq:pd:est:Delta:trunc}.
\end{cor}

We now move to the second class of smooth estimators of the partial derivatives. As we shall see below, to establish their weak consistency, it suffices, among others, that the underlying smoothing distributions satisfy the following weaker version of Condition~\ref{cond:var:W}.

\begin{cond}[Weak variance condition]
  \label{cond:var:W:weak}
  There exists a positive sequence $a_n \downarrow 0$ such that, for any $n \in \N$, $\bm x \in (\R^d)^n$, $\bm u \in [0,1]^d$ and $j \in \{1,\dots,d\}$, $\Var( W_{j,u_j}^{\bm x}) \leq a_n$.
\end{cond}

The following result is proven in Appendix~\ref{app:pd}.

\begin{prop}[Weak consistency in a sequential setting for the second class of smooth estimators]
\label{prop:wc:Delta:nu:trunc}
Under Conditions~\ref{cond:pd},~\ref{cond:Bn},~\ref{cond:band} and~\ref{cond:var:W:weak}, for any $j \in \{1,\dots,d\}$, $\delta \in (0,1)$ and $\eps \in (0,1/2)$,
\begin{align}
  \label{eq:wc:Delta:nu:trunc}
  \sup_{\substack{(s,t) \in \Lambda \\ t-s \geq \delta}} \sup_{\substack{\bm u \in [0,1]^d\\ u_j \in [\eps, 1-\eps]}}  \left| \dotu{C}_{j,\ip{ns}+1:\ip{nt}}^{\sss \Delta,\nu}(\bm u) - \dot C_j(\bm u) \right| = o_\Pr(1),
\end{align}
where $\dotu{C}_{j,k:l}^{\sss \Delta,\nu}$ is defined in~\eqref{eq:pd:est:Delta:nu:trunc}, and similarly for $\dotu{C}_{j,k:l}^{\sss \nabla,\nu}$ in~\eqref{eq:pd:est:nabla:nu:trunc}.
\end{prop}

One may wonder why the estimators $\dot C_{j,k:l}^{\sss \nabla,\nu}$ in~\eqref{eq:pd:est:nabla:nu} and $\dot C_{j,k:l}^{\sss \Delta,\nu}$ in~\eqref{eq:pd:est:Delta:nu} are not included in the previous proposition. Actually, upon additionally imposing that the left and right bandwidths of $\dot C_{j,k:l}^{\sss \nabla}$ in~\eqref{eq:pd:est:nabla} and $\dot C_{j,k:l}^{\sss \Delta}$ in~\eqref{eq:pd:est:Delta} (which are to be smoothed) are equal and in the absence of ties (see Condition~\ref{cond:no:ties}), weak consistency can also be proven for the estimators $\dot C_{j,k:l}^{\sss \nabla,\nu}$ and $\dot C_{j,k:l}^{\sss \Delta,\nu}$ using the same technique of proof. An inspection of the proof and some additional thinking reveals that this follows from the fact that these estimators are bounded on $[0,1]^d$ in this case. When one of the bandwidths is zero, this is not necessarily the case anymore. This is also why the previous proposition cannot be directly used to establish the weak consistency of the Bernstein estimator in~\eqref{eq:Bern:pd}. For this estimator, one additionally needs to rely on the fact that the finite difference-based estimator that is smoothed is bounded on the support of the smoothing distributions. This is used in the proof in Appendix~\ref{app:pd} of the next proposition.


\begin{prop}[Weak consistency of the Bernstein estimator in a sequential setting]
\label{prop:wc:Bern}
Assume that Conditions~\ref{cond:no:ties}, \ref{cond:pd} and~\ref{cond:Bn} hold and, for any $i \in \N$, let $m_i = \ip{L i^\theta} \vee 2$ for some constants $L > 0$ and $\theta \in (0,1/2]$. Then, for any $j \in \{1,\dots,d\}$, $\delta \in (0,1)$ and $\eps \in (0,1/2)$,
\begin{align}
  \label{eq:wc:Bern}
  \sup_{\substack{(s,t) \in \Lambda \\ t-s \geq \delta}} \sup_{\substack{\bm u \in [0,1]^d\\ u_j \in [\eps, 1-\eps]}}  \left| \dot C_{j,\ip{ns}+1:\ip{nt},m_{\ip{nt}-\ip{ns}}}^{\sss \Bern}  - \dot C_j(\bm u) \right| = o_\Pr(1),
\end{align}
where $\dot C_{j,k:l,m}^{\sss \Bern}$ is defined in~\eqref{eq:Bern:pd}. In addition, for any $n \in \N$,
\begin{equation}
  \label{eq:bounded:Bern}
  \sup_{(s,t,\bm u) \in \Lambda \times [0,1]^d}  \left| \dot C_{j,\ip{ns}+1:\ip{nt},m_{\ip{nt}-\ip{ns}}}^{\sss \Bern}(\bm u) \right| \leq 1 + L \vee 2.
\end{equation}
\end{prop}

Note that our technique of proof allows us to establish uniform weak consistency of the estimator even for $\theta = 1/2$, whereas the approach used in the proof of Proposition~3.1 of \cite{BouGhoTaa13} (for Bernstein copula density estimators and which could be adapted to first-order partial derivative estimators) leads to the result only for $\theta < 1/2$. Another consequence of the previous result is that the Bernstein partial derivative estimator as parametrized in Proposition~\ref{prop:wc:Bern} can satisfy Condition~\ref{cond:pd:est}.

\subsection{Finite-sample performance of selected estimators}
\label{sec:MC:pd}

The aim of this subsection is to compare the finite-sample performance of some of the estimators introduced previously. Specifically, for each $n \in \{10,20,\dots,100\}$, each data generating copula $C$ and each partial derivative estimator $\dot C_{j,1:n}$ under investigation, we estimated its integrated mean squared error 
\begin{equation*}
  \mathrm{IMSE}(\dot C_{j,1:n}) = \int_{[0,1]^d} \Ex \left[ \left\{ \dot C_{j,1:n}(\bm u) - \dot C_j(\bm u) \right\}^2 \right] \dd \bm u.
\end{equation*}
To do so, we applied the trick described in detail in Appendix~B of \cite{SegSibTsu17} allowing to compute $\mathrm{IMSE}(\dot C_{j,1:n})$ as a single expectation and proceeded by Monte Carlo simulation using 20 000 independent random samples of size $n$ from $C$. 

In a first experiment, we compared estimators of the form $\dotu{C}_{j,1:n,h,h'}^{\sss \nu,\nabla}$ in \eqref{eq:pd:est:nu:nabla:trunc} to estimators of the form $\dotu{C}_{j,1:n,h,h'}^{\sss \nu,\Delta}$ in \eqref{eq:pd:est:nu:Delta:trunc} for deterministic bandwidths $h = h' = n^{-1/2} \wedge 1/2$. Specifically, we considered estimators based, respectively, on finite-differences of the classical empirical copula $C_{1:n}$ in~\eqref{eq:C:kl}, on finite-differences of the empirical beta copula $C_{1:n}^{\sss \Bin}$ defined in~\eqref{eq:C:1n:beta}, and on finite-differences of its data-adaptive extension $C_{1:n}^{\sss \BetaB}$ defined at the end of Section~\ref{sec:ec}. As data-generating copula $C$, we considered the bivariate or trivariate Clayton or Gumbel--Hougaard copula with bivariate margins with a Kendall's tau of $\tau \in \{0,0.25,0.5,0.75\}$ as well as the bivariate Frank copula with a Kendall's tau of $\tau \in \{0,-0.25,-0.5,-0.75\}$. Note that, since all data-generating copulas are exchangeable, it suffices to focus on only one partial derivative estimator, say the first one. As expected, the integrated mean squared error of estimators of the form $\dotu{C}_{j,1:n,h,h'}^{\sss \nu,\Delta}$ was always found to be (substantially) below that of the corresponding estimator $\dotu{C}_{j,1:n,h,h'}^{\sss \nu,\nabla}$, confirming that the adjusted numerator in~\eqref{eq:pd:est:nu:Delta} compared to the one in~\eqref{eq:pd:est:nu:nabla} helps indeed to improve the finite-sample performance of finite-difference-based estimators.

In a second experiment, we compared the aforementioned three estimators of the form $\dotu{C}_{j,1:n,h,h'}^{\sss \nu,\Delta}$ in \eqref{eq:pd:est:nu:Delta:trunc}. They will be denoted as $\dotu{C}_{j,1:n,h,h'}^{\sss \Delta} = \dotu{C}_{j,1:n,h,h'}^{\sss \Dirac,\Delta}$,  $\dotu{C}_{j,1:n,h,h'}^{\sss \Bin,\Delta}$ and $\dotu{C}_{j,1:n,h,h'}^{\sss \BetaB,\Delta}$ as we continue. As could have been expected from the experiments of \cite{KojYi22} comparing the underlying copula estimators, it is the estimator $\dotu{C}_{j,1:n,h,h'}^{\sss \BetaB,\Delta}$ that always displayed the lowest integrated mean squared error, followed by $\dotu{C}_{j,1:n,h,h'}^{\sss \Bin,\Delta}$ and $\dotu{C}_{j,1:n,h,h'}^{\sss \Delta}$.

We next investigated the influence of the bandwidths on the integrated mean squared error of $\dotu{C}_{j,1:n,h,h'}^{\sss \BetaB,\Delta}$. Deterministic bandwidths of the form $h=h'= (L n^{-1/2}) \wedge 1/2$ were considered with $L \in \{0.5,1,2,4\}$. The corresponding integrated mean squared errors are represented in the first column of graphs of Figure~\ref{fig:pd:Frank} (resp.~Figure~\ref{fig:pd:Gumbel}) when the data-generating copula is the bivariate Frank copula with negative dependence (resp.\ the trivariate Gumbel--Hougaard copula). The legend ``BetaB 0.5'' refers to the estimator with $L=0.5$ and so on. As one can see, the weaker the cross-sectional dependence, the larger the (constant $L$ in the expression of the) bandwidths should be.

\begin{figure}[t!]
\begin{center}
  \includegraphics*[width=1\linewidth]{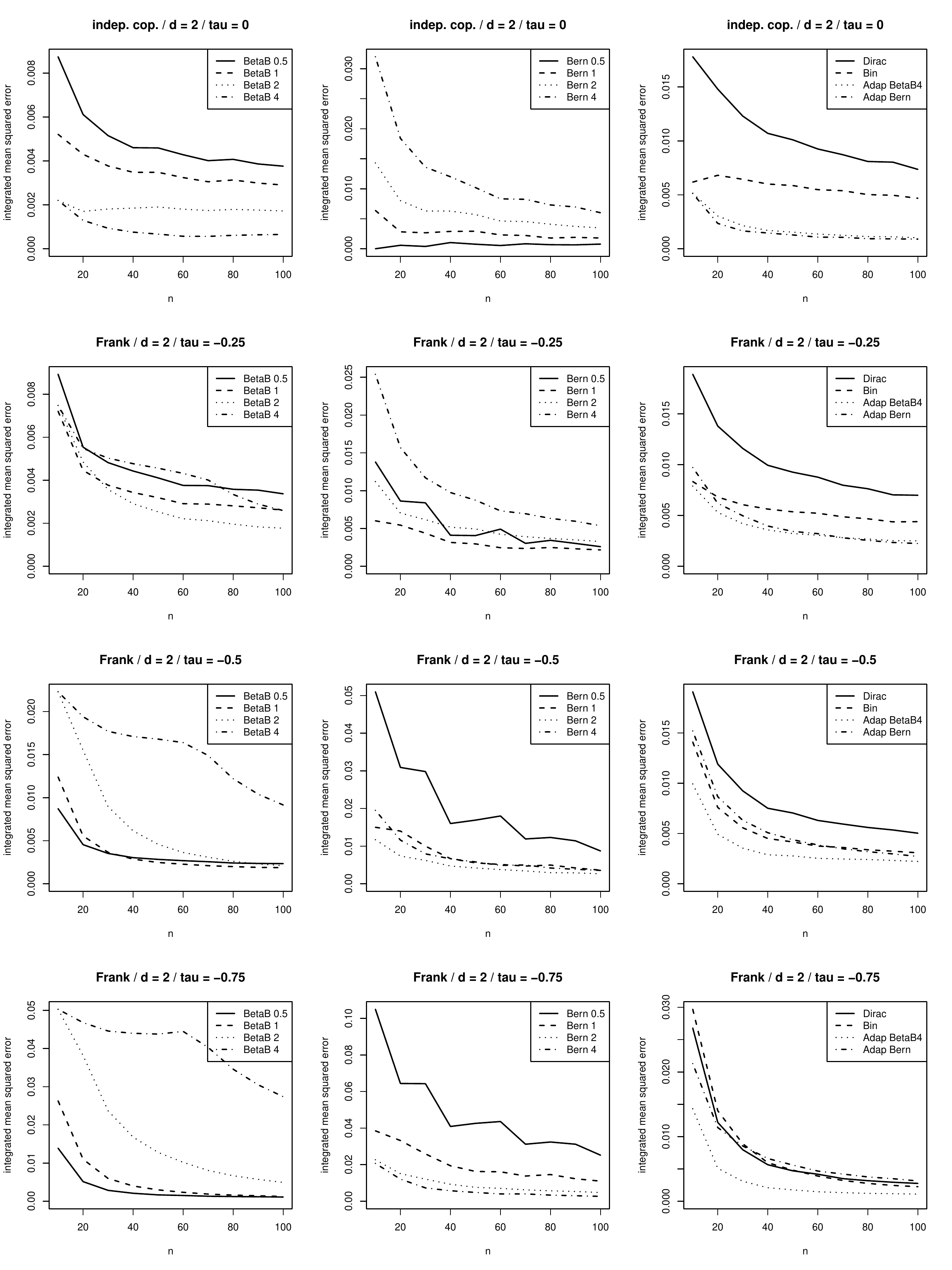}
  \caption{\label{fig:pd:Frank} Estimated integrated mean squared errors against $n \in \{10,20, \dots, 100\}$ of four estimators of $\dot C_1$ when $C$ is the bivariate Frank copula with a Kendall's tau in $\{0,-0.25,-0.5,-0.75\}$.}
\end{center}
\end{figure}

\begin{figure}[t!]
\begin{center}
  \includegraphics*[width=1\linewidth]{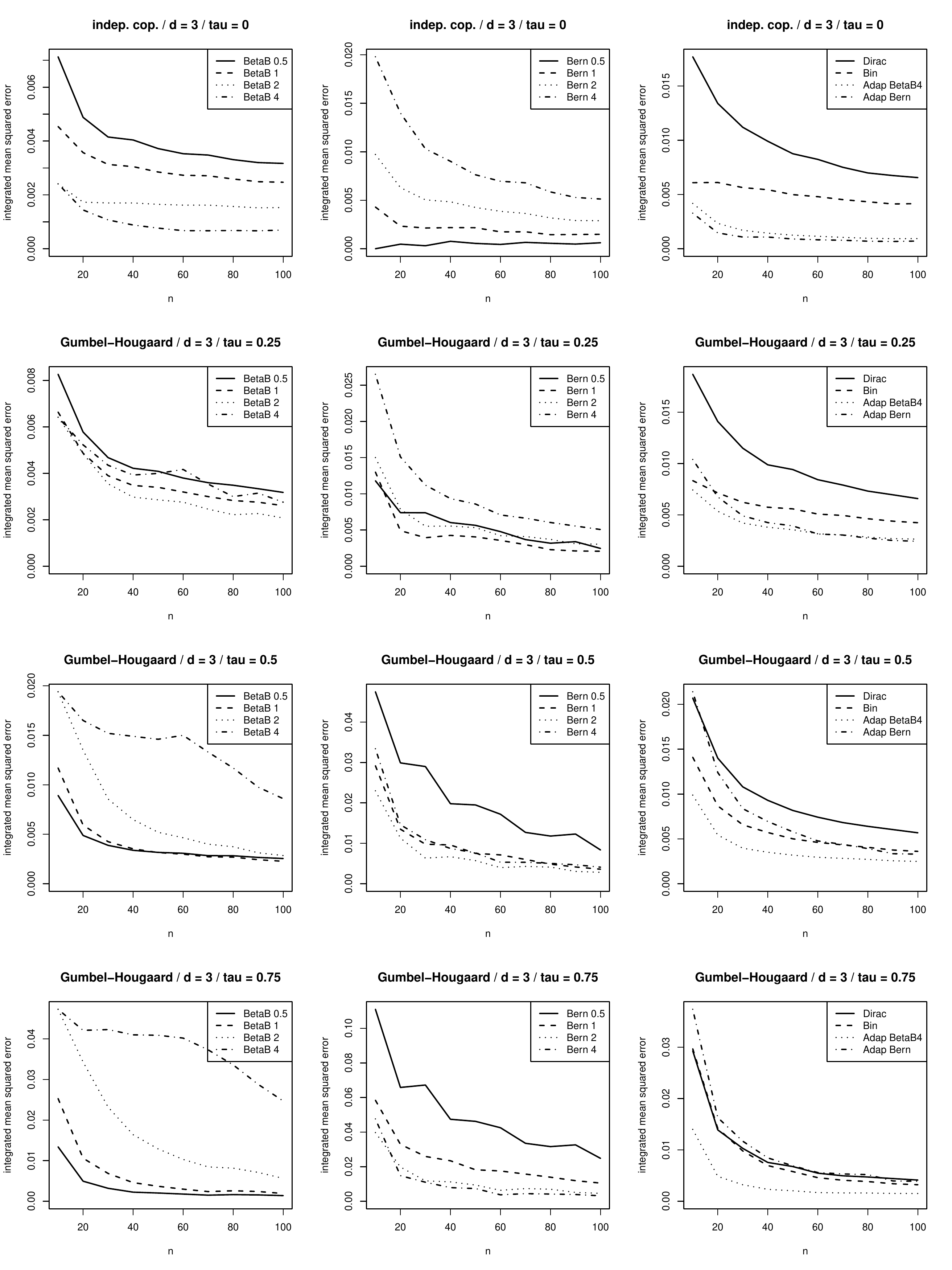}
  \caption{\label{fig:pd:Gumbel} Estimated integrated mean squared errors against $n \in \{10,20, \dots, 100\}$ of four estimators of $\dot C_1$ when $C$ is the trivariate Gumbel--Hougaard copula whose bivariate margins have a Kendall's tau in $\{0,0.25,0.5,0.75\}$.}
\end{center}
\end{figure}

An inspection of~\eqref{eq:pd:est:nabla:nu} and~\eqref{eq:pd:est:Delta:nu} and some thinking reveals that estimators from the second class can be difficult to compute in practice. For that reason, in our experiments, we solely focused on the Bernstein estimator $\dot C_{j,1:n,m}^{\sss \Bern}$ in~\eqref{eq:Bern:pd} which can be computed using~\eqref{eq:Bern:pd:comp}. Mimicking the previous experiment, we considered a deterministic choice for the parameter $m$ of the form $m = \ip{L n^{1/2}} \vee 2$ with $L \in \{0.5,1,2,4\}$. The corresponding integrated mean squared errors are represented in the second column of graphs of Figure~\ref{fig:pd:Frank} (resp.~Figure~\ref{fig:pd:Gumbel}) when the data-generating copula is the bivariate Frank copula with negative dependence (resp.\ the trivariate Gumbel--Hougaard copula). The legend ``Bern 0.5'' refers to the estimator with $L=0.5$ and so on. As one can see, this time, the stronger the cross-sectional dependence, the larger the (constant $L$ in the expression of the) parameter $m$ should be. This is of course not surprising given the previous experiment and since $1/m$ plays the role of a bandwidth.

The two previous experiments suggested to focus on data-adaptive bandwidths for the estimators $\dotu{C}_{j,1:n,h,h'}^{\sss \BetaB,\Delta}$ and $\dot C_{j,1:n,m}^{\sss \Bern}$. Specifically, for $d \in \{2,3\}$ and a data set $\bm x \in (\R^d)^n$, we considered the settings $h(\bm x) = h'(\bm x) = ([4 \{1 - |\tau(\bm x)|\}^6 + 1/2] n^{-1/2} ) \wedge 1/2$ for $\dotu{C}_{j,1:n,h,h'}^{\sss \BetaB,\Delta}$ and $m(\bm x) = \ip{\{ 4 |\tau(\bm x)|^{3/2} + 1/2 \} n^{1/2}} \vee 2$ for $\dot C_{j,1:n,m}^{\sss \Bern}$, where $\tau(\bm x) \in [-1,1]$ is the average of the values of the sample version of Kendall's tau computed from the bivariate margins of the data set~$\bm x$. The integrated mean squared errors of the resulting data-adaptive estimators are represented in the third column of graphs of Figure~\ref{fig:pd:Frank} (resp.~Figure~\ref{fig:pd:Gumbel}) when the data-generating copula is the bivariate Frank copula with negative dependence (resp.\ the trivariate Gumbel--Hougaard copula). The legends ``Adap BetaB4'' and ``Adap Bern'' refer to the above-mentioned versions of the estimators $\dotu{C}_{j,1:n,h,h'}^{\sss \BetaB,\Delta}$ and $\dot C_{j,1:n,m}^{\sss \Bern}$, respectively, while ``Dirac'' and ``Bin'' refer to the benchmark estimators $\dotu{C}_{j,1:n,h,h'}^{\sss \Delta} = \dotu{C}_{j,1:n,h,h'}^{\sss \Dirac,\Delta}$ and  $\dotu{C}_{j,1:n,h,h'}^{\sss \Bin,\Delta}$ with deterministic bandwidths $h = h' = n^{-1/2} \wedge 1/2$ based on the empirical copula and the empirical beta copula, respectively. Overall, it is the data-adaptive estimator $\dotu{C}_{j,1:n,h,h'}^{\sss \BetaB,\Delta}$ that displays the best finite-sample behavior. The data-adaptive Bernstein estimator $\dot C_{j,1:n,m}^{\sss \Bern}$ appears to be competitive only when the data-generating copula is close to the independence copula.

\section{Conclusion}

Smooth nonparametric copula estimators, such as the empirical beta copula proposed by \cite{SegSibTsu17} or its data-adaptive extension studied in \cite{KojYi22}, can be substantially better estimators than the classical empirical copula in small samples. To use such estimators in inference procedures, one typically needs to rely on resampling techniques.

As investigated in Section~\ref{sec:boot}, in the case of i.i.d.\ observations, a smooth bootstrap \emph{à la} \cite{KirSegTsu21} can be asymptotically valid for a large class of smooth estimators that can be expressed as mixtures of d.f.s. When based on the empirical beta copula, \cite{KirSegTsu21} found such a smooth bootstrap to be a competitive alternative to the multiplier bootstrap while being substantially simpler to implement. An empirical finding of this work is that the smooth bootstrap based on the data-adaptive extension of the empirical beta copula proposed in \cite{KojYi22} seems to lead to even better-behaved inference procedures than the former as it copes better with stronger dependence.

Unfortunately, such smooth bootstraps cannot be used anymore in the time series setting. A second contribution of this work was to study both theoretically and empirically smooth extensions of the sequential dependent multiplier bootstrap of \cite{BucKoj16}. As illustrated at the end of the fourth section, the latter can for instance be used to derive smooth change-point detection tests which are likely to be more sensitive to early or late changes than their non-smooth counterparts since, as already mentioned, smooth estimators are likely to be more accurate than the empirical copula when computed from small subsets of observations.

In connection with the multiplier bootstrap, a third contribution of this work was the study of the weak consistency and finite-sample performance of two classes of smooth estimators of the first-order partial derivatives of the copula. The obtained results may be of independent interest since such estimators have applications in mean and quantile regression as they lead to estimators of the conditional distribution function. From an empirical perspective, our investigations led to the proposal of a smooth data-adaptive estimator of the first-order partial derivatives of the copula that substantially outperforms, among others, the Bernstein estimator studied in \cite{JanSwaVer16}.


\begin{appendix}

\section{Proof of Corollary~\ref{cor:wc:Cb:n:nu}}
\label{proof:cor:wc:Cb:n:nu}

The proof of Corollary~\ref{cor:wc:Cb:n:nu} is based on the following two lemmas.

\begin{lem}
  \label{lem:stochastic}
  Let $\Xb_n$ be a process in $\ell^\infty(\Lambda \times [0,1]^d)$ such that for all $\bm u \in [0,1]^d$ and $s \in [0,1]$, $\Xb_n(s,s,\bm u) = 0$. Furthermore, assume that $\Xb_n \leadsto \Xb$ in $ \ell^\infty(\Lambda \times [0,1]^d)$ where $\Xb$ has continuous trajectories almost surely. Then, under Condition~\ref{cond:var:W:weak} (which is implied by Condition~\ref{cond:var:W}),
  \begin{align}
     \label{eq:seq}
    \sup_{(s,t,\bm u) \in \Lambda \times[0, 1]^d} \left| \int_{[0,1]^d}\Xb_n(s,t,\bm w)\dd \nu_{\bm u}^{\sss \Xc_{\ip{ns}+1:\ip{nt}}}(\bm w)-\Xb_n(s,t,\bm u) \right| &= o_\Pr(1),\\
     \label{eq:nonseq}
    \sup_{(s,t,\bm u) \in \Lambda \times[0, 1]^d} \left| \int_{[0,1]^d}\Xb_n(s,t,\bm w)\dd \nu_{\bm u}^{\sss \Xc_{1:n}}(\bm w)-\Xb_n(s,t,\bm u) \right| &= o_\Pr(1).
  \end{align}
\end{lem}

\begin{proof}
The first claim was proven in the proof of Lemma~32 of \cite{KojYi22}. The proof of~\eqref{eq:nonseq} is very similar. 
\end{proof}

\begin{lem}
\label{lem:bias}
Assume that Conditions~\ref{cond:pd} and~\ref{cond:var:W} hold. Then, almost surely,
\begin{align}
  \label{eq:bias:claim1}
  \sup_{(s,t,\bm u) \in \Lambda \times[0, 1]^d} &\sqrt{n}\lambda_n(s,t) \left| \int_{[0,1]^d} C(\bm w) \dd \nu_{\bm u}^{\sss \Xc_{\ip{ns}+1:\ip{nt}}}(\bm w)-C(\bm u) \right| = o(1),\\
  \label{eq:bias:claim2}
  \sup_{(s,t,\bm u) \in \Lambda \times[0, 1]^d} &\sqrt{n}\lambda_n(s,t) \left| \int_{[0,1]^d} C(\bm w) \dd\nu_{\bm u}^{\sss \Xc_{1:n}}(\bm w)-C(\bm u) \right| = o(1).
\end{align}
\end{lem}
\begin{proof}
The first claim was proven in the proof of Lemma~33 of \cite{KojYi22}. The proof of \eqref{eq:bias:claim2} is an immediate consequence of the fact that the left-hand side of \eqref{eq:bias:claim2} is almost surely smaller than
$$
\sup_{\bm u \in [0, 1]^d} \sqrt{n} \left| \int_{[0,1]^d} C(\bm w) \dd\nu_{\bm u}^{\sss \Xc_{1:n}}(\bm w)-C(\bm u) \right|,
$$
which is smaller than the left-hand side of \eqref{eq:bias:claim1} with probability~1.
\end{proof}

\begin{proof}[\bf Proof of Corollary~\ref{cor:wc:Cb:n:nu}]
Combining Theorem~\ref{thm:Cb:n:nu} and Theorem~\ref{thm:wc:Cb:n}, by the triangle inequality, we immediately obtain that
$$
\sup_{(s, t,\bm u)\in \Lambda \times [0, 1]^d} \left| \Cb_n^\nu(s, t, \bm u) - \tilde \Cb_n(s, t, \bm u) \right| = o_\Pr(1),
$$
where $\tilde \Cb_n$ is defined in~\eqref{eq:Cb:n:tilde}. It thus remains to show that
\begin{align*}
  \sup_{(s, t,\bm u)\in \Lambda \times [0, 1]^d} \left| \tilde \Cb_n(s, t, \bm u) - \tilde \Cb_n^\nu(s, t, \bm u) \right| &= o_\Pr(1), \\
  \sup_{(s, t,\bm u)\in \Lambda \times [0, 1]^d} \left| \tilde \Cb_n(s, t, \bm u) - \bar \Cb_n^\nu(s, t, \bm u) \right| &= o_\Pr(1).
\end{align*}
We only prove the first claim, the proof of the second one being similar. For any $(s,t,\bm u) \in \Lambda \times [0,1]^d$, let
$$
\breve{\Bb}_n^\nu(s,t,\bm u) = \int_{[0,1]^d}\Bb_n(s,t,\bm w) \dd \nu_{\bm u}^{\sss \Xc_{\ip{ns}+1:\ip{nt}}}(\bm w).
$$
Under Condition~\ref{cond:Bn}, $\Bb_n \leadsto \Bb_C$ in $ \ell^\infty(\Lambda \times [0,1]^d)$, where $\Bb_C$ has continuous trajectories almost surely. We then obtain from \eqref{eq:seq} in Lemma~\ref{lem:stochastic} that
$$
\sup_{(s,t,\bm u) \in \Lambda \times[0, 1]^d} \left| \breve{\Bb}_n^\nu(s,t,\bm u) - \Bb_n(s,t,\bm u) \right| = o_\Pr(1),
$$
and, furthermore, since Conditions~\ref{cond:pd} and~\ref{cond:var:W} hold, from \eqref{eq:bias:claim1} in Lemma~\ref{lem:bias}, that
$$
\sup_{(s,t,\bm u) \in \Lambda \times[0, 1]^d} \left| \breve{\Bb}_n^\nu(s,t,\bm u)- \tilde \Bb_n^\nu(s,t,\bm u) \right| = o(1),
$$
with probability one, where $\tilde \Bb_n^\nu$ is defined in \eqref{eq:tilde:Bb:n:nu}, which implies that
\begin{equation}
  \label{eq:ae:Bb:n:tilde}
  \sup_{(s,t,\bm u) \in \Lambda \times[0, 1]^d} \left| \tilde{\Bb}_n^\nu(s,t,\bm u)- \Bb_n(s,t,\bm u) \right| = o_\Pr(1).
\end{equation}
Moreover, from the triangle inequality, we have
\begin{align*}
\sup_{\substack{(s,t) \in \Lambda \\ \bm u \in [0,1]^d}} & \left| \tilde \Cb_n^\nu(s, t, \bm u) - \tilde \Cb_n(s, t, \bm u)\right| \leq  \sup_{\substack{(s,t) \in \Lambda \\ \bm u \in [0,1]^d}} \left| \tilde \Bb_n^\nu(s, t, \bm u) - \Bb_n(s, t, \bm u)  \right|\\
& +\sum_{j=1}^d \sup_{\bm u \in [0,1]^d}\left|\dot C_{j}(\bm u)\right|\ \sup_{\substack{(s,t) \in \Lambda \\ \bm u \in [0,1]^d}} \left| \tilde \Bb_n^\nu(s,t,\bm u^{(j)})-\Bb_n(s,t,\bm u^{(j)}) \right|.
\end{align*}
The terms on the right-hand side of the previous display converge to zero in probability as a consequence of~\eqref{eq:ae:Bb:n:tilde} and the fact that $0 \leq \dot C_j \leq 1$.
\end{proof}


\section{Proofs of Proposition~\ref{prop:mixture} and Lemma~\ref{lem:barBc:strict}}
\label{proofs:boot}

\begin{proof}[\bf Proof of Proposition~\ref{prop:mixture}]
  Fix $n \in \N$, $\bm x \in (\R^d)^n$ and $\bm r \in \{1,\dots,n\}^d$ and let us check that $\Kc_{\bm r}^{\bm x}$, which can be expressed as in~\eqref{eq:K:dis} under Condition~\ref{cond:unif:marg}, is a multivariate d.f.
  
  By Condition~\ref{cond:smooth:surv:marg}, for any $j \in \{1,\dots,d\}$, the function $\Kc_{r_j,j}^{\bm x}$ defined by $\Kc_{r_j,j}^{\bm x}(u) = \bar \Fc_{j,u}^{\bm x}\{ (r_j - 1) / n\}$, $u \in [0,1]$, is a univariate d.f.\ on $[0,1]$. Indeed, $\Kc_{r_j,j}^{\bm x}$ is right-continuous and increasing on $[0,1]$ and, by properties of the smoothing distributions,
\begin{align*}
  \Kc_{r_j,j}^{\bm x}(0) &= \bar \Fc_{j,0}^{\bm x}\{ (r_j - 1) / n\} = \Pr\{W_{j,0}^{\bm x} > (r_j - 1) / n\} =  \Pr\{0 > (r_j - 1) / n\} = 0, \\
  \Kc_{r_j,j}^{\bm x}(1) &= \bar \Fc_{j,1}^{\bm x}\{ (r_j - 1) / n\} = \Pr\{W_{j,1}^{\bm x} > (r_j - 1) / n\} =  \Pr\{1 > (r_j - 1) / n\} = 1.
\end{align*}
Using additionally Condition~\ref{cond:smooth:cop}, the expression of $\Kc_{\bm r}^{\bm x}$ in~\eqref{eq:K:dis} can then by further simplified to
$$
\Kc_{\bm r}^{\bm x}(\bm u) = \bar \Cc^{\bm x} \{ \Kc_{r_1,1}^{\bm x}(u_1),\dots,\Kc_{r_d,d}^{\bm x}(u_d) \}, \qquad \bm u \in [0,1]^d.
$$
From Sklar's Theorem \citep{Skl59}, $\Kc_{\bm r}^{\bm x}$ is thus a d.f.\ on $[0,1]^d$ with univariate margins $\Kc_{r_1,1}^{\bm x},\dots,\Kc_{r_d,d}^{\bm x}$ and copula $\bar \Cc^{\bm x}$.
\end{proof}

The proof of Lemma~\ref{lem:barBc:strict} below is based on the following lemma.

\begin{lem}
 \label{lem:barB:strict}
For any $n \in \N$ and $t \in [0,1]$, let $\bar B_{n,t}$ be the survival function of a Binomial$(n,t)$. Then, for any $n \in \N$ and $w \in [0,n)$, the function $t \mapsto \bar B_{n,t}(w)$ is strictly increasing on $[0,1]$.
\end{lem}

\begin{proof}
Fix $n \in \N$ and $w \in [0,n)$. Since $t \mapsto \bar B_{n,t}(w)$ is continuous on $[0,1]$, it suffices to prove that, for any $t \in (0,1)$, $\frac{\partial}{\partial t}\left\{\bar B_{n,t}(w)\right\}>0$. We have
{\small \begin{align*}
  \frac{\partial}{\partial t}&\left\{\bar B_{n,t}(w)\right\}  = \frac{\partial}{\partial t}\left\{\sum_{k=\ip{w}+1}^n{n \choose k} t^k (1-t)^{n-k}\right\} = \frac{\partial}{\partial t}\left\{\sum_{k=\ip{w}+1}^n{n \choose k} t^k (1-t)^{n-k} \right\}\\
  =& \sum_{k=\ip{w}+1}^n \frac{n!}{k!(n-k)!} \{k t^{k-1} (1-t)^{n-k} - (n-k) t^k (1-t)^{n-k-1} \}\\
  =& \sum_{k=\ip{w}+1}^n\frac{n!}{k!(n-k)!}k t^{k-1} (1-t)^{n-k}  - \sum_{k=\ip{w}+1}^{n-1}\frac{n!}{k!(n-k)!}(n-k) t^k (1-t)^{n-k-1}\\
  =& \sum_{k=\ip{w}+1}^n\frac{n!}{(k-1)!(n-k)!} t^{k-1} (1-t)^{n-k} - \sum_{k=\ip{w}+1}^{n-1}\frac{n!}{k!(n-k-1)!}t^k (1-t)^{n-k-1}\\
  =& \sum_{k=\ip{w}+1}^n\frac{n!}{(k-1)!(n-k)!} t^{k-1} (1-t)^{n-k} - \sum_{k=\ip{w}+2}^n\frac{n!}{(k-1)!(n-k)!}t^{k-1} (1-t)^{n-k}\\
  =& \frac{n!}{\ip{w}!(n-\ip{w}-1)!}t^{\ip{w}} (1-t)^{n-\ip{w}-1} > 0.
\end{align*}}
\end{proof}

\begin{lem}
 \label{lem:barBc:strict}
For any $n \in \N$, $t \in [0,1]$ and $\rho \in (1,n)$, let $\bar \Bc_{n,t,\rho}$ be the survival function of a Beta-Binomial$(n,\alpha, \beta)$, where  $\alpha = t(n - \rho)/(\rho - 1)$ and $\beta = (1-t)(n - \rho)/(\rho - 1)$. Then, for any $n \in \N$, $\rho \in (1,n)$ and $w \in [0,n)$, the function $t \mapsto \bar \Bc_{n,t,\rho}(w)$ is strictly increasing on $[0,1]$.
\end{lem}

\begin{proof}
First, notice that, for any $t \in [0,1]$, $\rho \in (1,n)$ and $w \in [0,n)$, by definition of the beta-binomial distribution,
\begin{equation}
 \label{eq:barBc:mixture}
 \bar \Bc_{n,t,\rho}(w)= \Ex_{\Theta}\{\bar B_{n,\Theta}(w)\},
\end{equation}
where $\bar B_{n,t}$ is the survival function of a Binomial$(n,t)$ and $\Theta$ is Beta$(\alpha,\beta)$ with $\alpha = t(n-\rho)/(\rho-1)$ and $\beta = (1-t)(n-\rho)/(\rho-1)$. From Lemma~30 in \cite{KojYi22}, we have that, for any $n \in \N$, $\rho \in (1,n)$ and $w \in [0,n)$, the function $t \mapsto \bar \Bc_{n,t,\rho}(w)$ is increasing on $[0,1]$. It thus suffices to show strict increasingness. Let us prove this by contradiction. Suppose that there exists $0\leq t_1 < t_2 \leq 1$ such that $\bar \Bc_{n,t_1,\rho}(w) = \bar \Bc_{n,t_2,\rho}(w)$ for some $n \in \N$, $\rho \in (1,n)$ and $w \in [0,n)$. Then, from \eqref{eq:barBc:mixture}, we have that
\begin{equation}
 \label{eq:exp:Theta}
\Ex_{\Theta_1}\{\bar B_{n,\Theta_1}(w)\} = \Ex_{\Theta_2}\{\bar B_{n,\Theta_2}(w)\}, 
\end{equation}
where $\Theta_1$ (resp.\ $\Theta_2$) is Beta$(\alpha_1,\beta_1)$ (resp.\ Beta$(\alpha_2,\beta_2)$) with $\alpha_1 = t_1(n-\rho)/(\rho-1)$ and $\beta_1 = (1-t_1)(n-\rho)/(\rho-1)$ (resp.\ $\alpha_2 = t_2(n-\rho)/(\rho-1)$ and $\beta_2 = (1-t_2)(n-\rho)/(\rho-1)$). From the proof of Lemma~29 in \cite{KojYi22}, we have that $\Theta_1 \leq_{st} \Theta_2$, where $\leq_{st}$ denotes the usual stochastic order. Using additionally~\eqref{eq:exp:Theta} and the fact that the function $t \mapsto \bar B_{n,t}(w)$ is strictly increasing on $[0,1]$ from Lemma~\ref{lem:barB:strict}, we have, according to Theorem~1.A.8 of \cite{ShaSha07}, that $\Theta_1$ and $\Theta_2$ have the same distribution. This contradicts the fact that $t_1 < t_2$.
\end{proof}


\section{Proof of Theorem~\ref{thm:asym:val:Cb:n:hash}}
\label{proof:thm:asym:val:Cb:n:hash}

The proof of Theorem~\ref{thm:asym:val:Cb:n:hash} is based on two lemmas which we show first.

Let $\Phi$ be the map from $\ell^\infty([0,1]^d)$ to $\ell^\infty([0,1]^d)$ defined for any d.f $H$ on $[0,1]^d$ whose univariate margins $H_1,\dots,H_d$ do not assign mass at zero by
\begin{equation}
 \label{eq:Phi}
 \Phi(H)(\bm u) = H\{H_1^{-1}(u_1), \dots, H_d^{-1}(u_d)\}, \qquad \bm u \in [0,1]^d.
\end{equation}

\begin{lem}
\label{lem:unif:as:C:1n:nu}
Assume that the random vectors in $\Xc_{1:n}$ are i.i.d.\ and that Condition~\ref{cond:var:W:weak} holds. Then, almost surely,
\begin{equation}
  \label{eq:unif:as:C:1n:nu}
  \sup_{\bm u \in [0,1]^d} |C_{1:n}^\nu(\bm u) - C(\bm u) | = o(1),
\end{equation}
where $C_{1:n}^\nu$ is defined in~\eqref{eq:C:kl:nu}. 
\end{lem}

\begin{proof}
  The supremum on the left-hand side of~\eqref{eq:unif:as:C:1n:nu} is smaller than $I_n + J_n$, where
\begin{align*}
  I_n &= \sup_{\bm u \in [0,1]^d} \left|\int_{[0,1]^d} C_{1:n}(\bm w) \dd \nu_{\bm u}^{\sss \Xc_{1:n}}(\bm w) - \int_{[0,1]^d} C(\bm w) \dd \nu_{\bm u}^{\sss \Xc_{1:n}}(\bm w) \right|, \\
  J_n &= \sup_{\bm u \in [0,1]^d} \left|\int_{[0,1]^d} C(\bm w) \dd \nu_{\bm u}^{\sss \Xc_{1:n}}(\bm w) - C(\bm u) \right|.
\end{align*}
\emph{Term $I_n$:} From the triangle inequality, $I_n$ is smaller than
\begin{align*}
\sup_{\bm u \in [0,1]^d}& |C_{1:n}(\bm u) - C(\bm u) | \leq I_n' + I_n'' + I_n''',
\end{align*} where 
\begin{align*}
I_n' &= \sup_{\bm u \in [0,1]^d} |C_{1:n}(\bm u) - \Phi(G_{1:n})(\bm u) |, \qquad I_n'' = \sup_{\bm u \in [0,1]^d} |\Phi(G_{1:n})(\bm u) -  G_{1:n}(\bm u) |,\\
I_n''' &= \sup_{\bm u \in [0,1]^d} |G_{1:n}(\bm u) -  C(\bm u) |,
\end{align*}
where the map $\Phi$ is defined in~\eqref{eq:Phi} and $G_{1:n}$ is empirical d.f.\ of the unobservable random sample $\bm U_1,\dots,\bm U_n$ obtained from $\Xc_{1:n}$ by the probability integral transformations $U_{ij} = F_j (X_{ij})$, $i \in \{1,\dots,n\}$, $j \in \{1, \dots, d\}$. Using the well-known facts \citep[see, e.g.,][]{Seg12} that $\Phi(G_{1:n}) = \Phi(F_{1:n})$, where $F_{1:n}$ is the empirical d.f. of $\Xc_{1:n}$, and 
$$
\sup_{\bm u \in [0,1]^d} |C_{1:n}(\bm u) - \Phi(F_{1:n})(\bm u) | \leq \frac{d}{n},
$$
we obtain that $I_n' = o(1)$. Furthermore, from the Glivenko-Cantelli lemma \citep[see, e.g.,][Theorem~19.1]{Van98}, $I_n''' = o(1)$ with probability one. Finally, using a well-known property of multivariate d.f.s \citep[see, e.g.,][Lemma~1.2.14]{DurSem15}, the well-known fact, for any $j \in \{1,\dots,d\}$, $\sup_{u \in [0,1]} |G_{1:n,j}^{-1}(u) - u| = \sup_{u \in [0,1]} |G_{1:n,j}(u) - u|$ \citep[see, e.g.,][Chapter~3]{ShoWel86} and, again, the Glivenko-Cantelli lemma, we obtain that, almost surely,
$$
I_n'' \leq \sum_{j=1}^d \sup_{u \in [0,1]} |G_{1:n,j}^{-1}(u) - u| = \sum_{j=1}^d \sup_{u \in [0,1]} |G_{1:n,j}(u) - u| = o(1).
$$

\emph{Term $J_n$:} We proceed as in the proof of Lemma~3.2 of \cite{SegSibTsu17}. Fix $\eta > 0$ and let us show that, with probability one, $J_n$ can be made smaller than $\eta$ provided $n$ is large enough.  Let $|\cdot|_\infty$ denote the maximum norm on $\R^d$. For any $\eps > 0$, we have that 
\begin{align*}
J_n =& \sup_{\bm u \in [0,1]^d} \left|\int_{[0,1]^d} \{C(\bm w) - C(\bm u)\} \dd \nu_{\bm u}^{\sss \Xc_{1:n}}(\bm w) \right|\\
                   \leq & \sup_{\bm u \in [0, 1]^d} \left| \int_{\{\bm w \in [0,1]^d: |\bm u - \bm w|_\infty  \leq \eps \}} \{ C(\bm w) - C(\bm u) \} \dd \nu_{\bm u}^{\sss \Xc_{1:n}}(\bm w) \right| \\
                                             &+ \sup_{\bm u \in [0, 1]^d} \left| \int_{\{\bm w \in [0,1]^d: |\bm u - \bm w|_\infty > \eps \}} \{ C(\bm w) - C(\bm u) \} \dd \nu_{\bm u}^{\sss \Xc_{1:n}}(\bm w) \right| \leq J_n' + J_n'',
\end{align*}
where
\begin{align*}
  J_n' &= \sup_{\substack{(\bm u, \bm w) \in [0,1]^{2d}\\|\bm u - \bm w|_{\infty} \leq \eps}} \left|C(\bm w) - C(\bm u) \right|, \\
  J_n'' &= \sup_{\bm u \in [0, 1]^d} \nu_{\bm u}^{\sss \Xc_{1:n}}(\{\bm w \in [0,1]^d:|\bm u - \bm w|_{\infty} > \eps \}).
\end{align*}
Let $\eps = \eta/(2d)$. Then, from the Lipschitz continuity of $C$, $J_n' \leq \eta/2$. As far as $J_n''$ is concerned, conditionally on $\bm X_1,\bm X_2,\dots$, for almost any sequence $\bm X_1,\bm X_2,\dots$, using Chebyshev's inequality and Condition~\ref{cond:var:W:weak}, we have that
\begin{align*}
  J_n'' &= \sup_{\bm u \in [0, 1]^d} \Pr \left\{ \left| \bm W_{\bm u}^{\sss \Xc_{1:n}}  - \bm u) \right|_{\infty} > \eps \mid \Xc_{1:n} \right\} \\
        &= \sup_{\bm u \in [0, 1]^d} \Pr \left[ \bigcup_{j=1}^d \left\{ \left| W_{j,u_j}^{\sss \Xc_{1:n}} - u_j \right| >\eps \right\} \mid \Xc_{1:n} \right]\\
        & \leq \sup_{\bm u \in [0, 1]^d} \sum_{j=1}^{d}\Pr \left\{ \left| W_{j,u_j}^{\sss \Xc_{1:n}}- u_j \right| >\eps \mid \Xc_{1:n} \right\} \\
        &\leq \sup_{\bm u \in [0, 1]^d} \sum_{j=1}^{d}\frac{\Var(W_{j,u_j}^{\sss \Xc_{1:n}} \mid \Xc_{1:n} )}{\eps^2} \leq \frac{d a_n}{\eps^2}.
\end{align*}
which implies that, for $n$ sufficiently large, $J_n'' \leq \eta/2$. The latter holds conditionally on $\bm X_1,\bm X_2,\dots$ for almost any sequence $\bm X_1,\bm X_2,\dots$, which completes the proof.
\end{proof}

Next, we recall the mode of convergence classically used to state asymptotic validity results of resampling techniques when dealing with empirical processes; see, e.g., \citet[Chapter~2.9]{vanWel96} or \citet[Section~2.2.3]{Kos08}. Let
\begin{multline*}
BL_1 = \{h:\ell^\infty([0,1]^d) \to [-1,1] \text{ such that,} \\ \text{for all } x, y \in \ell^\infty([0,1]^d), \, |h(x) - h(y)|  \leq \sup_{\bm u \in [0,1]^d} |x(\bm u) - y(\bm u)|\}.
\end{multline*}
Let $\Xb_n = \Xb_n(\Xc_{1:n},\Wc_n)$ be a sequence of bootstrapped empirical processes in $\ell^\infty([0,1]^d)$ depending on an additional source of randomness $\Wc_n$ (often called the ``bootstrap weights''). For the smooth bootstraps under consideration, $\Wc_n$ is independent of the data $\Xc_{1:n}$ and corresponds to $n$ independent copies of the independent random variables $I$ and $\bm U^\#$ necessary to carry out Algorithm~\ref{algo:sampling} (see also~\eqref{eq:V:hash}) $n$ times independently. The notation $\Xb_n \cwc \Xb$ then means that
\begin{itemize}
\item $\sup_{h \in BL_1} |\Ex_\Wc \{h(\Xb_n)\} - \Ex \{ h(\Xb) \} \to 0$ in outer probability,
\item $\Ex_\Wc \{h(\Xb_n)^*\} - \Ex_\Wc \{h(\Xb_n)_*\} \p 0$ for all $h \in BL_1$,
\end{itemize}
where $\Ex_\Wc$ denotes an expectation with respect to the bootstrap weights $\Wc_n$ only and $h(\Xb_n)^*$ and $h(\Xb_n)_*$ denote the minimal measurable majorant and maximal measurable minorant with respect to $(\Xc_{1:n},\Wc_n)$.

The next lemma is very closely related to Proposition 3.3 of \cite{KirSegTsu21}.

\begin{lem}
  \label{lem:cwc:Cb:n:hash}
  Assume that the random vectors in $\Xc_{1:n}$ are i.i.d., and that Conditions~\ref{cond:unif:marg}, \ref{cond:smooth:surv:marg}, \ref{cond:smooth:cop}, \ref{cond:pd} and~\ref{cond:var:W} hold. Then,
  \begin{equation}
    \label{eq:cwc:Cb:n:hash}
  \sqrt{n}(C_{1:n}^{\sss \#}-C_{1:n}^\nu)  \cwc \Cb_C(0,1,\cdot),
  \end{equation}
  where $\Cb_C$ is defined in~\eqref{eq:Cb:C}.
\end{lem}

\begin{proof}
Let $G_{1:n}^{\sss \#}$ be the empirical d.f.\ of $\Vc_{1:n}^{\sss \#}$. Using Lemma~\ref{lem:unif:as:C:1n:nu} and proceeding as in Step~1 of the proof of Proposition 3.3 of \cite{KirSegTsu21}, one obtains that
\begin{equation}
  \label{eq:wc:Gb:n:hash}
  \sqrt{n}(G_{1:n}^{\sss \#}-C_{1:n}^\nu) \cwc \Bb_C(0,1,\cdot),
\end{equation}
where $\Bb_C$ is defined in~\eqref{eq:Bb:C}. Then, proceeding as in Step~2 of the proof of Proposition 3.3 of \cite{KirSegTsu21}, that is, combining~\eqref{eq:wc:Gb:n:hash} with the Hadamard differentiability of the map $\Phi$ in~\eqref{eq:Phi} established in Theorem~2.4 of \cite{BucVol13}, the functional delta method for the bootstrap ``in probability'' \citep[Theorem~3.9.11]{vanWel96} and the fact $\Phi(C_{1:n}^\nu) = C_{1:n}^\nu$ (since $C_{1:n}^\nu$ has standard uniform margins under Condition~\ref{cond:unif:marg} in the considered i.i.d.\ setting), one obtains
\begin{equation}
  \label{eq:wc:tilde:Cb:n:hash}
  \sqrt{n}(\Phi(G_{1:n}^{\sss \#}) - C_{1:n}^\nu)  \cwc \Cb_C(0,1,\cdot).
\end{equation}
The desired result finally follows from~\eqref{eq:wc:tilde:Cb:n:hash} and the well-known fact that
$$
\sup_{\bm u \in [0,1]^d} |C_{1:n}^{\sss \#}(\bm u) - \Phi(G_{1:n}^{\sss \#})(\bm u) | \leq \frac{d}{n}
$$
since the components samples of $\Vc_{1:n}^{\sss \#}$ contain no ties with probability one.
\end{proof}

\begin{proof}[\bf Proof of Theorem~\ref{thm:asym:val:Cb:n:hash}]
Combining Lemma~\ref{lem:cwc:Cb:n:hash} with Lemma~3.1 of \cite{BucKoj19}, we obtain that~\eqref{eq:cwc:Cb:n:hash} is equivalent to 
\begin{multline}
  \label{eq:joint}
\big(\Cb_n(0,1,\cdot), \sqrt{n}(C_{1:n}^{\sss \#,[1]} - C_{1:n}^\nu), \sqrt{n}(C_{1:n}^{\sss \#,[2]} - C_{1:n}^\nu) \big) \\ \leadsto \big(\Cb_C(0,1,\cdot),\Cb_C^{\sss [1]}(0,1,\cdot),\Cb_C^{\sss [2]}(0,1,\cdot)\big) 
\end{multline}
in $\{\ell^\infty([0,1]^d)\}^3$, where $\Cb_n$ is defined in~\eqref{eq:Cb:n}. From Theorem~\ref{thm:Cb:n:nu}, we have that
\begin{equation}
  \label{eq:intermediate}
 \sup_{\bm u \in [0,1]^d} \sqrt{n} |C_{1:n}^\nu(\bm u) - C_{1:n}(\bm u)| = \sup_{\bm u \in [0,1]^d} |\Cb_n^\nu(0,1,\bm u) - \Cb_n(0,1,\bm u)| = o_\Pr(1),
\end{equation}
where $\Cb_n^\nu$ is defined in~\eqref{eq:Cb:n:nu}. The first joint weak convergence in Theorem~\ref{thm:asym:val:Cb:n:hash} then follows from~\eqref{eq:joint} and~\eqref{eq:intermediate}.

Fix $j \in \{1,2\}$. Since~\eqref{eq:intermediate} holds, to establish the second joint weak convergence from the first, it suffices to show that
\begin{equation}
  \label{eq:aim}
  \sup_{\bm u \in [0,1]^d} \left| \sqrt{n}\{C_{1:n}^{\sss \#,\nu,[j]}(\bm u) - C_{1:n}^\nu(\bm u) \}  - \sqrt{n} \{ C_{1:n}^{\sss \#,[j]}(\bm u) - C_{1:n}(\bm u) \} \right| = o_\Pr(1).
\end{equation}
The supremum on the left hand-side of~\eqref{eq:aim} is smaller than $I_n + J_n$, where
{\small \begin{align*}
I_n =& \sup_{\bm u \in [0,1]^d} \left| \int_{[0,1]^d} \sqrt{n} \{ C_{1:n}^{\sss \#,[j]}(\bm w) - C_{1:n}(\bm w) \} \dd \nu_{\bm u}^{\sss \Vc_{1:n}^{\sss \#, [j]}}(\bm w) - \sqrt{n} \{ C_{1:n}^{\sss \#,[j]}(\bm u) - C_{1:n}(\bm u) \} \right|, \\
J_n =& \sqrt{n} \sup_{\bm u \in [0,1]^d} \left|\int_{[0,1]^d} C_{1:n}(\bm w)\dd \nu_{\bm u}^{\sss \Vc_{1:n}^{\sss \#, [j]}}(\bm w) - \int_{[0,1]^d} C_{1:n}(\bm w)\dd \nu_{\bm u}^{\sss \Xc_{1:n}}(\bm w) \right|.
\end{align*}}
Since, according to the first claim of Theorem~\ref{thm:asym:val:Cb:n:hash}, $\sqrt{n}(C_{1:n}^{\sss \#,[j]} - C_{1:n}) \leadsto \Cb_C(0,1,\cdot)$ in $\ell^\infty([0, 1]^d)$ and $\Cb_C(0,1,\cdot)$ has continuous trajectories almost surely, it can be verified by proceeding as in the proof of Lemma~\ref{lem:stochastic} that $I_n = o_\Pr(1)$. For the term $J_n$, we have that $J_n \leq K_n + L_n$, where
\begin{align*}
  K_n =& \sqrt{n}\sup_{\bm u \in [0,1]^d}\left|\int_{[0,1]^d}\{C_{1:n}(\bm w)-C(\bm w)\}\dd \nu_{\bm u}^{\sss \Vc_{1:n}^{\sss \#, [j]}}(\bm w) \right. \\ &- \left. \int_{[0,1]^d}\{C_{1:n}(\bm w)-C(\bm w)\}\dd \nu_{\bm u}^{\sss \Xc_{1:n}}(\bm w) \right|,\\
  L_n =& \sqrt{n}\sup_{\bm u \in [0,1]^d}\left|\int_{[0,1]^d}C(\bm w)\dd \nu_{\bm u}^{\sss \Vc_{1:n}^{\sss \#, [j]}}(\bm w)-\int_{[0,1]^d}C(\bm w)\dd \nu_{\bm u}^{\sss \Xc_{1:n}}(\bm w) \right|.
\end{align*}
The term $K_n$ is smaller than $K_n' + K_n''$, where
\begin{align*}
  K_n' &= \sup_{\bm u \in [0,1]^d} \left| \int_{[0,1]^d} \Cb_n(0,1,\bm w) \dd \nu_{\bm u}^{\sss \Vc_{1:n}^{\sss \#,[j]}}(\bm w)- \Cb_n(0,1,\bm u)\right|, \\
  K_n'' &= \sup_{\bm u \in [0,1]^d} \left| \Cb_n(0,1,\bm u)-\int_{[0,1]^d}\Cb_n(0,1,\bm w) \dd \nu_{\bm u}^{\sss \Xc_{1:n}}(\bm w)\right|.
\end{align*}
From~\eqref{eq:nonseq} in Lemma~\ref{lem:stochastic}, $K_n'' = o_\Pr(1)$ and, proceeding again as in the proof of the latter lemma, it can be verified that $K_n' = o_\Pr(1)$. The term $L_n$ is smaller than $L_n' + L_n''$, where
\begin{align*}           
L_n' &= \sqrt{n}\sup_{\bm u \in [0,1]^d}\left|\int_{[0,1]^d}C(\bm w)\dd \nu_{\bm u}^{\sss \Vc_{1:n}^{\sss \#,[j]}}(\bm w)-C(\bm u)\right|,\\
L_n'' &= \sqrt{n}\sup_{\bm u \in [0,1]^d}\left|C(\bm u)-\int_{[0,1]^d}C(\bm w)\dd \nu_{\bm u}^{\sss \Xc_{1:n}}(\bm w) \right|.
\end{align*}
The term $L_n''$ converges almost surely to zero as a consequence of~\eqref{eq:bias:claim2} in Lemma~\ref{lem:bias}. The proof of the latter result can be adapted to verify that the term $L_n'$ also converges almost surely to zero. Hence, \eqref{eq:aim} holds, which completes the proof. 
\end{proof}


\section{Proof of Theorem~\ref{thm:ae:Cb:n:hat:check}}
\label{proof:thm:ae:Cb:n:hat:check}

The proof of Theorem~\ref{thm:ae:Cb:n:hat:check} is based on the following lemma which we prove first.

\begin{lem}
\label{lem:Bb:n:nu}
Under Conditions~\ref{cond:DGP} and~\ref{cond:var:W:weak} (the latter is implied by Condition~\ref{cond:var:W}), for any $b \in \N$,
\begin{align}
  \label{eq:ae:Bb:n:hat}
  \sup_{(s,t,\bm u) \in \Lambda \times [0,1]^d} \left| \hat{\Bb}_n^{\sss [b],\nu} (s,t,\bm u)-\hat{\Bb}_n^{\sss [b]} (s,t,\bm u)\right| = o_\Pr(1), \\
  \label{eq:ae:Bb:n:check}
  \sup_{(s,t,\bm u) \in \Lambda \times [0,1]^d} \left| \check{\Bb}_n^{\sss [b],\nu} (s,t,\bm u)-\check{\Bb}_n^{\sss [b]} (s,t,\bm u)\right| = o_\Pr(1).
\end{align}
\end{lem}

\begin{proof}
Fix $b \in \N$. We first prove \eqref{eq:ae:Bb:n:hat}. Starting from~\eqref{eq:Bb:n:nu:hat}, we have that 
\begin{multline*}
\sup_{(s,t,\bm u) \in \Lambda \times [0,1]^d} \left| \hat{\Bb}_n^{\sss [b],\nu} (s,t,\bm u)-\hat{\Bb}_n^{\sss [b]} (s,t,\bm u)\right| \\ = \sup_{(s,t,\bm u) \in \Lambda \times [0,1]^d} \left| \int_{[0,1]^d} \hat{\Bb}_n^{\sss [b]} (s,t,\bm w) \dd\nu_{\bm u}^{\sss \Xc_{1:n}}(\bm w)-\hat{\Bb}_n^{\sss [b]} (s,t,\bm u)\right|= o_\Pr(1),
\end{multline*}
where the last equality follows from~\eqref{eq:nonseq} in Lemma~\ref{lem:stochastic} since $\hat{\Bb}_n^{\sss [b]} \leadsto \Bb_C$ in $ \ell^\infty(\Lambda \times [0,1]^d)$, where $\Bb_C$ is defined in~\eqref{eq:Bb:C} and has continuous trajectories almost surely under Condition~\ref{cond:DGP}. Under Condition~\ref{cond:DGP}~(ii), the latter is a consequence of Lemmas~D.1 and~D.2 in \cite{BucKoj16} as well as Theorem~2.1 in \cite{BucKoj16}. Under Condition~\ref{cond:DGP}~(i), one can rely on Theorem~1 of \cite{HolKojQue13} instead.
  
The proof of~\eqref{eq:ae:Bb:n:check} is similar. Starting from~\eqref{eq:Bb:n:nu:check}, we have that
\begin{multline*}
\sup_{(s,t,\bm u) \in \Lambda \times [0,1]^d} \left| \check{\Bb}_n^{\sss [b],\nu} (s,t,\bm u)-\check{\Bb}_n^{\sss [b]} (s,t,\bm u)\right| \\ = \sup_{(s,t,\bm u) \in \Lambda \times [0,1]^d} \left| \int_{[0,1]^d} \check{\Bb}_n^{\sss [b]} (s,t,\bm w) \dd\nu_{\bm u}^{\sss \Xc_{\ip{ns}+1:\ip{nt}}}(\bm w) - \check{\Bb}_n^{\sss [b]} (s,t,\bm u)  \right| = o_\Pr(1),
\end{multline*}
where the last equality follows from~\eqref{eq:seq} in Lemma~\ref{lem:stochastic} since $\check{\Bb}_n^{\sss [b]} \leadsto \Bb_C$ in $ \ell^\infty(\Lambda \times [0,1]^d)$. Under Condition~\ref{cond:DGP}~(ii), the latter is a consequence of (B.3) in the proof of Proposition~4.3 in \citet{BucKojRohSeg14} and Theorem~2.1 in \citet{BucKoj16}. Under Condition~\ref{cond:DGP}~(i), one can rely again on Theorem~1 of \citet{HolKojQue13} instead.
\end{proof}

\begin{proof}[\bf Proof of Theorem~\ref{thm:ae:Cb:n:hat:check}]
Fix $b \in \N$. We only prove~\eqref{eq:ae:Cb:n:check}, the proof of~\eqref{eq:ae:Cb:n:hat} being simpler. Starting from~\eqref{eq:Cb:n:check} and~\eqref{eq:Cb:n:check:nu}, we have that
\begin{multline*}
\sup_{\substack{(s,t) \in \Lambda \\ \bm u \in [0,1]^d}} \left| \check{\Cb}_n^{\sss [b]} (s,t,\bm u) - \check{\Cb}_n^{\sss [b],\nu} (s,t,\bm u)\right| \leq  \sup_{\substack{(s,t) \in \Lambda \\ \bm u \in [0,1]^d}} \left| \check{\Bb}_n^{\sss [b]} (s,t,\bm u) - \check{\Bb}_n^{\sss [b],\nu} (s,t,\bm u) \right|\\
 +\sum_{j=1}^d \sup_{\substack{(s,t) \in \Lambda \\ \bm u \in [0,1]^d}} \left|\dot C_{j,\ip{ns}+1:\ip{nt}}(\bm u)\right| \sup_{\substack{(s,t) \in \Lambda \\ \bm u \in [0,1]^d}} \left|\check{\Bb}_n^{\sss [b]}(s,t,\bm u^{(j)})-\check{\Bb}_n^{\sss [b],\nu}(s,t,\bm u^{(j)}) \right|.
\end{multline*}
The terms on the right-hand side of the previous display converge to zero in probability as a consequence of~\eqref{eq:ae:Bb:n:check} in Lemma~\ref{lem:Bb:n:nu} and the fact that $\sup_{(s,t,\bm u) \in \Lambda \times [0,1]^d}  \left| \dot C_{j,\ip{ns}+1:\ip{nt}}(\bm u) \right| \leq \zeta$ from Condition~\ref{cond:pd:est}.

The last two claims of the theorem are an immediate consequence of~\eqref{eq:ae:Cb:n:hat}, \eqref{eq:ae:Cb:n:check} and straightforward extensions of Propositions~4.2 and~4.3 in \cite{BucKojRohSeg14} for non-smooth multiplier replicates based on arbitrary partial derivative estimators satisfying Condition~\ref{cond:pd:est}. 
\end{proof}


\section{Proofs of Propositions~\ref{prop:Bern:pd:comp}, \ref{prop:wc:nu:Delta}, \ref{prop:wc:Delta:nu:trunc} and~\ref{prop:wc:Bern}}
\label{app:pd}

\begin{lem}
 \label{lem:Ber:pd}
 Let $f$ be any function from $[0,1]^d$ to $[0,1]$, let $m \in \N$, $m \geq 2$ and recall that, for any $\bm u \in [0,1]^d$, $\mu_{m,\bm u}$ is the law of the random vector $(S_{m,1,u_1}/m,\dots, S_{m,d,u_d}/m)$, where $S_{m,1,u_1},\dots S_{m,d,u_d}$ are independent, and for each $k \in \{1,\dots,d\}$, $S_{m,k,u_k}$ is Binomial$(m,u_k)$. Moreover, recall that, for any $j \in \{1,\dots,d\}$, $\tilde \mu_{j,m,\bm u}$ is the law of the random vector $(\tilde S_{m,1,u_1}/m,\dots, \tilde S_{m,d,u_d}/m)$ whose components are independent and, for $i \in \{1,\dots,d\} \setminus \{j\}$, $\tilde S_{m,i,u_i}$ is Binomial$(m,u_i)$, whereas $\tilde S_{m,j,u_j}$ is Binomial$(m-1,u_j)$. Then, for any $j \in \{1,\dots,d\}$ and $\bm u \in [0,1]^d$ such that $u_j \in (0,1)$, 
\begin{align*}
\partial_{u_j} \left\{\int_{[0,1]^d} f(\bm w) \dd \mu_{m,\bm u}(\bm w) \right\} = m \int_{[0,1]^d} \left\{ f(\bm w + \bm e_j/m) - f(\bm w)\right\} \dd \tilde \mu_{j,m,\bm u}(\bm w).
\end{align*}
\end{lem}

\begin{proof}
Fix $m \geq 2$ and, without loss of generality, fix $j=1$. Also, for any $u \in [0,1]$, let $b_{m, u}(s)={m \choose s} u^{s}(1-u)^{m-s}$, $s \in \{0,\dots,m\}$. Then, for all $\bm u \in [0,1]^d$ such that $u_1 \in (0,1)$, 
{\small \begin{align*}
  &\partial_{u_1} \left\{\int_{[0,1]^d} f(\bm w) \dd \mu_{m,\bm u}(\bm w) \right\} = \partial_{u_1} \left\{ \sum_{s_1=0}^m \cdots \sum_{s_d=0}^m f{\left(\frac{s_1}{m}, \dots, \frac{s_d}{m} \right)} \prod_{j=1}^d b_{m, u_j}(s_j) \right\}\\
          =& \sum_{s_1=0}^m \cdots \sum_{s_d=0}^m f{\left(\frac{s_1}{m}, \dots, \frac{s_d}{m} \right)} \partial_{u_1} b_{m, u_1}(s_1) \prod_{j=2}^d b_{m, u_j}(s_j)\\
  =& \sum_{s_1=0}^m \cdots \sum_{s_d=0}^m f{\left(\frac{s_1}{m}, \dots, \frac{s_d}{m} \right)} \\
                &\times  {m \choose s_1} \left\{ s_1 u_1^{s_1-1}(1-u_1)^{m-s_1} -  (m-s_1) u_1^{s_1}(1-u_1)^{m-s_1-1}\right\}  \prod_{j=2}^d b_{m, u_j}(s_j) \\
  =&\; m \sum_{s_1=1}^m \cdots \sum_{s_d=0}^m f{\left(\frac{s_1}{m}, \dots, \frac{s_d}{m} \right)} \frac{(m-1)!}{(s_1-1)! (m-s_1)!}u_1^{s_1-1}(1-u_1)^{m-s_1}  \prod_{j=2}^d b_{m, u_j}(s_j) \\
                &- m \sum_{s_1=0}^{m-1} \cdots \sum_{s_d=0}^m f{\left(\frac{s_1}{m}, \dots, \frac{s_d}{m} \right)} \frac{(m-1)!}{s_1 ! (m-s_1-1)!} u_1^{s_1}(1-u_1)^{m-s_1-1}\prod_{j=2}^d b_{m, u_j}(s_j)\\
  =&\; m \sum_{s_1=0}^{m-1} \cdots \sum_{s_d=0}^m f{\left(\frac{s_1+1}{m}, \dots, \frac{s_d}{m} \right)} \frac{(m-1)!}{s_1! (m-s_1-1)!} u_1^{s_1}(1-u_1)^{m-s_1-1} \prod_{j=2}^d b_{m, u_j}(s_j) \\
                &- m \sum_{s_1=0}^{m-1} \cdots \sum_{s_d=0}^m f{\left(\frac{s_1}{m}, \dots, \frac{s_d}{m} \right)} \frac{(m-1)!}{s_1 ! (m-s_1-1)!} u_1^{s_1}(1-u_1)^{m-s_1-1} \prod_{j=2}^d b_{m, u_j}(s_j) \\
  =& \;m\sum_{s_1=0}^{m-1} \cdots \sum_{s_d=0}^m \left\{ f \left(\frac{s_1+1}{m}, \dots, \frac{s_d}{m} \right) - f\left(\frac{s_1}{m}, \dots, \frac{s_d}{m} \right)\right\} b_{m-1, u_1}(s_1)\prod_{j=2}^d b_{m, u_j}(s_j)\\
  =& \;m \int_{[0,1]^d} \left\{ f(\bm w + \bm e_1/m) - f(\bm w)\right\} \dd \tilde \mu_{1,m,\bm u}(\bm w).
\end{align*}}
\end{proof}

\begin{proof}[\bf Proof of Proposition~\ref{prop:Bern:pd:comp}]
  Fix $j \in \{1,\dots,d\}$, $\bm u \in [0,1]^d$ and $m \geq 2$, and recall the definition of the measure $\tilde \mu_{j,m,\bm u}$ given in Lemma~\ref{lem:Ber:pd}. From \eqref{eq:Bern:pd}, we have that
\begin{equation}
  \label{eq:Bern:pd:split}
  \dot C_{j,k:l,m}^{\sss \Bern}(\bm u) = m \int_{[0,1]^d} C_{k:l}(\bm w + \bm e_j/m) \dd \tilde \mu_{j,m,\bm u}(\bm w) - m \int_{[0,1]^d} C_{k:l}(\bm w)\dd \tilde \mu_{j,m,\bm u}(\bm w).
\end{equation}
Let $\bm{\tilde S} = (\tilde S_{m,1,u_1},\dots, \tilde S_{m,d,u_d})$ so that $\bm{\tilde S}/m$ is a random vector with law $\tilde \mu_{j,m,\bm u}$. Then, combined with the definition of $C_{k:l}$ in \eqref{eq:C:kl}, the first integral on the right-hand side of~\eqref{eq:Bern:pd:split} can be rewritten as
\begin{align*}
  \int_{[0,1]^d} &\frac{1}{l-k+1} \sum_{i=k}^l \1 \left\{\bm R_i^{k:l} / (l-k+1) \leq \bm w + \bm e_j/m \right\}\dd \tilde \mu_{j,m,\bm u}(\bm w)\\
                 =&  \frac{1}{l-k+1} \sum_{i=k}^l \int_{[0,1]^d} \1 \left\{\bm R_i^{k:l} / (l-k+1) - \bm e_j/m \leq \bm w  \right\}\dd \tilde \mu_{j,m,\bm u}(\bm w)\\
                 =& \frac{1}{l-k+1} \sum_{i=k}^l \Pr \left\{\bm{\tilde S} \geq m\bm R_i^{k:l} / (l-k+1) - \bm e_j \mid \Xc_{k:l} \right\}\\
  =& \frac{1}{l-k+1} \sum_{i=k}^l \Pr\left\{\tilde S_{m, j, u_j} \geq m R_{ij}^{k:l} / (l-k+1) - 1 \mid \Xc_{k:l} \right\} \\
                 &\times \prod_{t=1 \atop t \neq j}^d \Pr\left\{\tilde S_{m, t, u_t} \geq m R_{it}^{k:l} / (l-k+1) \mid \Xc_{k:l} \right\}  \\
                 =& \frac{1}{l-k+1} \sum_{i=k}^l \Pr\left\{\tilde S_{m, j, u_j} > \up{m R_{ij}^{k:l} / (l-k+1) - 1} - 1 \mid \Xc_{k:l} \right\} \\ &\times \prod_{t=1 \atop t \neq j}^d \Pr\left\{\tilde S_{m, t, u_t} > \up{m R_{it}^{k:l} / (l-k+1)} - 1 \mid \Xc_{k:l}\right\}  \\
                 =& \frac{1}{l-k+1} \sum_{i=k}^l \bar B_{m-1, u_j}\left\{\up{mR_{ij}^{k:l}/(l-k+1)} - 2\right\} \\ &\times \prod_{t=1 \atop t \neq j}^d \bar B_{m, u_t}\left\{\up{mR_{it}^{k:l}/(l-k+1)} - 1\right\},
\end{align*}
where we have used the fact that, for any $t \in \{1,\dots,d\}$ and $x \in \R$, $\Pr(\tilde S_{m, t, u_t} \geq x) = \Pr(\tilde S_{m, t, u_t} > \up{x}-1)$ and $\up{x-1} = \up{x} - 1$. Similarly, for the second integral on the right-hand side of~\eqref{eq:Bern:pd:split}, we have
\begin{align*}
  \int_{[0,1]^d} &C_{k:l}(\bm w)\dd \tilde \mu_{j,m,\bm u}(\bm w) \\
  =& \frac{1}{l-k+1} \int_{[0,1]^d} \sum_{i=k}^l \1 \left\{\bm R_i^{k:l} / (l-k+1) \leq \bm w \right\} \dd \tilde \mu_{j,m,\bm u}(\bm w) \\
  =& \frac{1}{l-k+1} \sum_{i=k}^l \Pr \left\{ \bm{\tilde S} \geq m \bm R_i^{k:l} / (l-k+1) \mid \Xc_{k:l} \right\} \\
  =& \frac{1}{l-k+1} \sum_{i=k}^l \prod_{t=1}^d \Pr\left\{ \tilde S_{m, t, u_t} \geq m R_{it}^{k:l} / (l-k+1) \mid \Xc_{k:l} \right\} \\
  =& \frac{1}{l-k+1} \sum_{i=k}^l \bar B_{m-1, u_j}\left\{\up{mR_{ij}^{k:l}/(l-k+1)} - 1\right\} \\ &\times \prod_{t=1 \atop t \neq j}^d \bar B_{m, u_t}\left\{\up{mR_{it}^{k:l}/(l-k+1)} - 1\right\}.
\end{align*}
The desired result finally follows by noticing that
\begin{align*}
  &\bar B_{m-1, u_j} \left\{\up{mR_{ij}^{k:l}/(l-k+1)} - 2\right\} - \bar B_{m-1, u_j}\left\{\up{mR_{ij}^{k:l}/(l-k+1)} - 1\right\}  \\
                    &= B_{m-1, u_j}\left\{\up{mR_{ij}^{k:l}/(l-k+1)} - 1\right\}  - B_{m-1, u_j} \left\{\up{mR_{ij}^{k:l}/(l-k+1)} - 2\right\} \\
                    &= b_{m-1, u_j}\left\{\up{mR_{ij}^{k:l}/(l-k+1)} - 1\right\}.
\end{align*}
\end{proof}

\begin{lem}
  \label{lem:ae:nabla:Delta}
  Under Condition~\ref{cond:band}, for any $j \in \{1,\dots,d\}$, $\delta \in (0,1)$ and $\eps \in (0,1/2)$, with probability~1,
\begin{align*}
 \sup_{\substack{(s,t) \in \Lambda \\ t-s \geq \delta}} \sup_{\substack{\bm u \in [0,1]^d\\ u_j \in [\eps, 1-\eps]}}  \left| \dot C_{j,\ip{ns}+1:\ip{nt}}^{\sss \nu,\nabla}(\bm u) - \dot C_{j,\ip{ns}+1:\ip{nt}}^{\sss \nu,\Delta}(\bm u) \right| =o(1).
\end{align*}
\end{lem}

\begin{proof}
Fix $j \in \{1,\dots,d\}$, $\delta \in (0,1)$ as well as $\eps \in (0,1/2)$ and assume that $n$ is large enough so that, for any $(s,t) \in \Lambda$ such that $t-s > \delta$, $L_2 b_{\ip{nt}-\ip{ns}}$ and $L_2 b'_{\ip{nt}-\ip{ns}}$ are smaller than~$\eps$. Then, using the fact that $C_{k:l}^\nu$ in~\eqref{eq:C:kl:nu} is between 0 and 1, we obtain that, with probability~1,
  \begin{multline*}
    \sup_{\substack{(s,t) \in \Lambda \\ t-s \geq \delta}}\sup_{\substack{\bm u \in [0,1]^d\\ u_j \in [\eps, 1-\eps]}}  \left| \dot C_{j,\ip{ns}+1:\ip{nt}}^{\sss \nu,\nabla}(\bm u) - \dot C_{j,\ip{ns}+1:\ip{nt}}^{\sss \nu,\Delta}(\bm u) \right|  
     \\ \leq \sup_{\substack{(s,t) \in \Lambda \\ t-s \geq \delta}} \sup_{\substack{\bm u \in [0,1]^d\\ u_j \in [\eps, 1-\eps]}} \left| \frac{1}{h+h'} - \frac{1}{(u_j+h) \wedge 1 - (u_j - h') \vee  0} \right| = 0.   \\
  \end{multline*}
\end{proof}

\begin{proof}[\bf Proof of Proposition~\ref{prop:wc:nu:Delta}]
Fix $j \in \{1,\dots,d\}$ and let us first prove~\eqref{eq:wc:nu:Delta} by proceeding along the lines of the proof of (B.4) in \cite{BucKojRohSeg14}. From~\eqref{eq:Cb:n:nu}, notice that, for any $(s,t,\bm u) \in \Lambda \times [0,1]^d$ such that $\ip{ns} < \ip{nt}$,
$$
C_{\ip{ns}+1:\ip{nt}}^\nu(\bm u) = C(\bm u) + \frac{1}{\sqrt{n}\lambda_n(s,t)}\Cb_n^\nu(s,t,\bm u).
$$
Fix $\delta \in (0,1)$ and notice that, by Condition~\ref{cond:band},
\begin{equation}
  \label{eq:dn}
  d_n = \sup_{\substack{(s,t) \in \Lambda \\ t-s \geq \delta}} (b_{\ip{nt}-\ip{ns}} + b'_{\ip{nt}-\ip{ns}}) \leq \sup_{k \geq \ip{n\delta}-1} (b_k + b'_k) \to 0.
\end{equation}
Next, fix $\eps \in (0,1/2)$ and assume that $n$ is large enough so that, for any $t-s > \delta$, $L_2 b_{\ip{nt}-\ip{ns}}$ and $L_2 b'_{\ip{nt}-\ip{ns}}$ are smaller than $\eps/2$. Then, for any $t-s > \delta$ and $\bm u \in [0,1]^d$ such that $u_j \in [\eps, 1-\eps]$,
\begin{multline}
  \label{eq:pd:1}
  \dot C_{j,\ip{ns}+1:\ip{nt}}^{\sss \nu,\Delta}(\bm u) = \frac{1}{h+h'}\left\{C(\bm u + h \bm e_j) - C(\bm u - h' \bm e_j) \right\}
  \\ + \frac{1}{(h+h')\sqrt{n}\lambda_n(s,t)} \left\{\Cb_n^\nu(s,t,\bm u + h \bm e_j) - \Cb_n^\nu(s,t,\bm u - h' \bm e_j) \right\}.
\end{multline}
Since, by Condition~\ref{cond:pd}, $\dot C_j$ exists (and is continuous) on the set $\{\bm u \in [0,1]^d : u_j \in [\eps/2, 1-\eps/2]\}$, from the mean value theorem, for any $t-s > \delta$ and $\bm u \in [0,1]^d$ such that $u_j \in [\eps, 1-\eps]$,
$$
\frac{1}{h+h'}\left\{C(\bm u + h \bm e_j) - C(\bm u - h' \bm e_j) \right\} = \dot C_j(\bm u^*_{n,s,t}),
$$
where $\bm u^*_{n,s,t}$ is between $\bm u - h' \bm e_j$ and $\bm u + h \bm e_j$ almost surely. Hence, with probability~1,
\begin{multline}
  \label{eq:pd:2}
  \sup_{\substack{(s,t) \in \Lambda \\ t-s \geq \delta}} \sup_{\substack{\bm u \in [0,1]^d\\ u_j \in [\eps, 1-\eps]}} \left|\frac{1}{h+h'}\left\{C(\bm u + h \bm e_j) - C(\bm u - h' \bm e_j) \right\} - \dot C_j(\bm u) \right|\\
  = \sup_{\substack{(s,t) \in \Lambda \\ t-s \geq \delta}} \sup_{\substack{\bm u \in [0,1]^d\\ u_j \in [\eps, 1-\eps]}} \left| \dot C_j(\bm u^*_{n,s,t})- \dot C_j(\bm u) \right| \leq  \sup_{\substack{(\bm u,\bm v) \in [0,1]^{2d} \\ u_j, v_j \in [\eps/2,1-\eps/2] \\ |\bm u - \bm v|_\infty \leq L_2 d_n }}  \left|\dot C_j(\bm u) - \dot C_j(\bm v) \right| \to 0,
\end{multline}
where $d_n$ is defined in~\eqref{eq:dn}. Furthermore, since, by Condition~\ref{cond:var:W} and as a result of Theorem~\ref{thm:Cb:n:nu}, $\Cb_n^\nu$ is asymptotically uniformly equicontinuous in probability, we have that
\begin{align}
\nonumber
\sup_{\substack{(s,t) \in \Lambda \\ t-s \geq \delta}} &\sup_{\substack{\bm u \in [0,1]^d\\ u_j \in [\eps, 1-\eps]}}  \left|\Cb_n^\nu(s,t,\bm u + h \bm e_j) - \Cb_n^\nu(s,t,\bm u - h' \bm e_j) \right|\\
\label{eq:pd:3}
&\leq \sup_{\substack{(s,t) \in \Lambda \\ t-s \geq \delta}} \sup_{\substack{(\bm u,\bm v) \in [0,1]^{2d} \\ u_j, v_j \in [\eps/2,1-\eps/2] \\ |\bm u - \bm v|_\infty \leq L_2 d_n }}  \left|\Cb_n^\nu(s,t,\bm u) - \Cb_n^\nu(s,t,\bm v) \right| = o_\Pr(1).
\end{align}
The fact that~\eqref{eq:wc:nu:Delta} holds is then an immediate consequence of \eqref{eq:pd:1}, \eqref{eq:pd:2}, \eqref{eq:pd:3} and the fact that, from Condition~\ref{cond:band},
\begin{align*}
  \sup_{\substack{(s,t) \in \Lambda \\ t-s \geq \delta}} \frac{1}{\left(h+h'\right)\sqrt{n}\lambda_n(s,t)} &\leq \sup_{\substack{(s,t) \in \Lambda \\ t-s \geq \delta}} \frac{1}{L_1 (b_{\ip{nt}- \ip{ns}} + b_{\ip{nt}- \ip{ns}}') \sqrt{n}\lambda_n(s,t)} \\ 
                                                                                                         &\leq \sup_{\substack{(s,t) \in \Lambda \\ t-s \geq \delta}} \frac{1}{L_1 (\ip{nt}-\ip{ns})^{-1/2}\sqrt{n}\lambda_n(s,t)} \\
                                                                                                         &= \sup_{\substack{(s,t) \in \Lambda \\ t-s \geq \delta}}\frac{1}{L_1 \sqrt{\lambda_n(s,t)}} \leq \frac{1}{L_1 \sqrt{\delta-1/n}}.
\end{align*}
The claim for $\dot C_{j,k:l}^{\sss \nu,\nabla}$ (resp.\ for $\dotu{C}_{j,k:l}^{\sss \nu,\Delta}$ and $\dotu{C}_{j,k:l}^{\sss \nu,\nabla}$) follows from Lemma~\ref{lem:ae:nabla:Delta} (resp.\ the continuous mapping theorem).
\end{proof}

\begin{proof}[\bf Proof of Proposition~\ref{prop:wc:Delta:nu:trunc}]
Fix $j \in \{1,\dots,d\}$, $\delta \in (0,1)$ and $\eps \in (0,1/2)$. We first prove~\eqref{eq:wc:Delta:nu:trunc}. From~\eqref{eq:pd:est:Delta:nu:trunc} and the triangle inequality, we have that
\begin{align*}
\sup_{\substack{(s,t) \in \Lambda \\ t-s \geq \delta}}& \sup_{\substack{\bm u \in [0,1]^d\\ u_j \in [\eps, 1-\eps]}}  \left| \dotu{C}_{j,\ip{ns}+1:\ip{nt}}^{\sss \Delta,\nu}(\bm u) - \dot C_j(\bm u) \right| \leq \; I_{j,n,\delta,\eps} +J_{j,n,\delta,\eps},
\end{align*}
where
{\small \begin{align*}
I_{j,n,\delta,\eps} &= \sup_{\substack{(s,t) \in \Lambda \\ t-s \geq \delta}} \sup_{\substack{\bm u \in [0,1]^d\\ u_j \in [\eps, 1-\eps]}}  \left| \int_{[0,1]^d}\left\{ \dotu{C}_{j,\ip{ns}+1:\ip{nt}}^{\sss \Delta}(\bm{w}) - \dot C_j(\bm w)\right\} \dd \nu_{\bm u}^{\sss \Xc_{\ip{ns}+1:\ip{nt}}}(\bm w) \right|,\\
  J_{j,n,\delta,\eps} &= \sup_{\substack{(s,t) \in \Lambda \\ t-s \geq \delta}} \sup_{\substack{\bm u \in [0,1]^d\\ u_j \in [\eps, 1-\eps]}}  \left|  \int_{[0,1]^d} \dot C_j(\bm w) \dd \nu_{\bm u}^{\sss \Xc_{\ip{ns}+1:\ip{nt}}}(\bm w) -  \dot C_j(\bm u)\right|,
\end{align*}}
where $\dotu{C}_{j,k:l}^{\sss \Delta}$ is defined in~\eqref{eq:pd:est:Delta:trunc}. We shall now show that both $I_{j,n,\delta,\eps} = o_\Pr(1)$  and $J_{j,n,\delta,\eps} = o_\Pr(1)$.

\emph{Term $I_{j,n,\delta,\eps}$:} From the triangle inequality and the fact that $0 \leq \dotu{C}_{j,k:l}^{\sss \Delta} \leq 1$ and $0 \leq \dot{C}_j \leq 1$, we have that $I_{j,n,\delta,\eps}$ is smaller than
{\small \begin{align*}
  & \sup_{\substack{(s,t) \in \Lambda \\ t-s \geq \delta}} \sup_{\substack{\bm u \in [0,1]^d\\ u_j \in [\eps, 1-\eps]}}  \left| \int_{\substack{\{\bm w \in [0,1]^d : \\ w_j \in [\eps/2, 1-\eps/2]\}}}\left\{ \dotu{C}_{j,\ip{ns}+1:\ip{nt}}^{\sss \Delta}(\bm{w}) - \dot C_j(\bm w)\right\} \dd \nu_{\bm u}^{\sss \Xc_{\ip{ns}+1:\ip{nt}}}(\bm w) \right|\\
                      & + \sup_{\substack{(s,t) \in \Lambda \\ t-s \geq \delta}} \sup_{\substack{\bm u \in [0,1]^d\\ u_j \in [\eps, 1-\eps]}}  \left| \int_{\substack{\{\bm w \in [0,1]^d : \\ w_j < \eps/2\}}} \left\{ \dotu{C}_{j,\ip{ns}+1:\ip{nt}}^{\sss \Delta}(\bm{w}) - \dot C_j(\bm w)\right\} \dd \nu_{\bm u}^{\sss \Xc_{\ip{ns}+1:\ip{nt}}}(\bm w) \right|\\
                      & + \sup_{\substack{(s,t) \in \Lambda \\ t-s \geq \delta}} \sup_{\substack{\bm u \in [0,1]^d\\ u_j \in [\eps, 1-\eps]}}  \left| \int_{\substack{\{\bm w \in [0,1]^d :\\ w_j > 1-\eps/2\}}}\left\{ \dotu{C}_{j,\ip{ns}+1:\ip{nt}}^{\sss \Delta}(\bm{w}) - \dot C_j(\bm w)\right\} \dd \nu_{\bm u}^{\sss \Xc_{\ip{ns}+1:\ip{nt}}}(\bm w) \right|\\
                   &\leq \; I_{j,n,\delta,\eps}' + I_{j,n,\delta,\eps}'' + I_{j,n,\delta,\eps}''',
\end{align*}}
where
\begin{align*}
I_{j,n,\delta,\eps}' =&  \sup_{\substack{(s,t) \in \Lambda \\ t-s \geq \delta}} \sup_{\substack{\bm w \in [0,1]^d\\ w_j \in [\eps/2, 1-\eps/2]}}  \left| \dotu{C}_{j,\ip{ns}+1:\ip{nt}}^{\sss \Delta}(\bm w) - \dot C_j(\bm w) \right|,\\
I_{j,n,\delta,\eps}'' =& \sup_{\substack{(s,t) \in \Lambda \\ t-s \geq \delta}} \sup_{\substack{\bm u \in [0,1]^d\\ u_j \in [\eps, 1-\eps]}} \nu_{\bm u}^{\sss \Xc_{\ip{ns}+1:\ip{nt}}}\left\{\bm w \in [0,1]^d:w_j < \eps/2  \right\},\\
I_{j,n,\delta,\eps}''' =& \sup_{\substack{(s,t) \in \Lambda \\ t-s \geq \delta}} \sup_{\substack{\bm u \in [0,1]^d\\ u_j \in [\eps, 1-\eps]}} \nu_{\bm u}^{\sss \Xc_{\ip{ns}+1:\ip{nt}}}\left\{\bm w \in [0,1]^d:w_j > 1 - \eps/2  \right\}.
\end{align*}
We have that $I_{j,n,\delta,\eps}' = o_\Pr(1)$ as a consequence of Corollary~\ref{cor:wc:Delta}. We shall now show that both $I_{j,n,\delta,\eps}''$ and $I_{j,n,\delta,\eps}'''$ converge almost surely to zero. To do so, it suffices to show that $I_{j,n,\delta,\eps}''$ and $I_{j,n,\delta,\eps}'''$ converge to zero conditionally on $\bm X_1,\bm X_2,\dots$ for almost any sequence $\bm X_1,\bm X_2,\dots$. Concerning $I_{j,n,\delta,\eps}'' $, using Chebyshev's inequality and Condition~\ref{cond:var:W:weak}, for almost any sequence $\bm X_1, \bm X_2, \dots$, conditionally on $\bm X_1, \bm X_2, \dots$, we obtain that
{\small \begin{align*}
  I_{j,n,\delta,\eps}'' &= \sup_{\substack{(s,t) \in \Lambda \\ t-s \geq \delta}} \sup_{u_j \in [\eps, 1-\eps]} \Pr \left\{W_{j,u_j}^{\sss \Xc_{\ip{ns}+1:\ip{nt}}} < \eps/2 \mid \Xc_{\ip{ns}+1:\ip{nt}} \right\}  \\
                                                         &= \sup_{\substack{(s,t) \in \Lambda \\ t-s \geq \delta}} \sup_{u_j \in [\eps, 1-\eps]} \Pr \left\{W_{j,u_j}^{\sss \Xc_{\ip{ns}+1:\ip{nt}}} - u_j < \eps/2 - u_j \mid \Xc_{\ip{ns}+1:\ip{nt}} \right\} \\
                                                         &\leq \sup_{\substack{(s,t) \in \Lambda \\ t-s \geq \delta}} \sup_{u_j \in [\eps, 1-\eps]} \Pr \left\{ - \left| W_{j,u_j}^{\sss \Xc_{\ip{ns}+1:\ip{nt}}} - u_j \right| \leq - u_j + \eps/2 \mid \Xc_{\ip{ns}+1:\ip{nt}} \right\} \\
                        &\leq \sup_{\substack{(s,t) \in \Lambda \\ t-s \geq \delta}} \sup_{u_j \in [\eps, 1-\eps]} \frac{\Var\left(W_{j,u_j}^{\sss \Xc_{\ip{ns}+1:\ip{nt}}}\mid \Xc_{\ip{ns}+1:\ip{nt}} \right)}{(u_j-\eps/2)^2} \\
                        &\leq \sup_{\substack{(s,t) \in \Lambda \\ t-s \geq \delta}} \sup_{u_j \in [\eps, 1-\eps]} \frac{a_{\ip{nt}-\ip{ns}}}{(u_j-\eps/2)^2} \leq \sup_{\substack{(s,t) \in \Lambda \\ t-s \geq \delta}} a_{\ip{nt}-\ip{ns}} \sup_{u_j \in [\eps, 1-\eps]} \frac{1}{(u_j - \eps/2)^2} \\
                        &\leq \frac{4}{\eps^2} \sup_{k \geq \ip{n\delta}-1} a_k \to 0.
\end{align*}}
Similarly, concerning $I_{j,n,\delta,\eps}''' $, for almost any sequence $\bm X_1, \bm X_2, \dots$, conditionally on $\bm X_1, \bm X_2, \dots$, we obtain that
{\small \begin{align*}
  I_{j,n,\delta,\eps}''' &= \sup_{\substack{(s,t) \in \Lambda \\ t-s \geq \delta}} \sup_{u_j \in [\eps, 1-\eps]} \Pr \left\{W_{j,u_j}^{\sss \Xc_{\ip{ns}+1:\ip{nt}}} > 1- \eps/2 \mid \Xc_{\ip{ns}+1:\ip{nt}} \right\}  \\
&= \sup_{\substack{(s,t) \in \Lambda \\ t-s \geq \delta}} \sup_{u_j \in [\eps, 1-\eps]} \Pr \left\{W_{j,u_j}^{\sss \Xc_{\ip{ns}+1:\ip{nt}}} - u_j > 1- \eps/2 - u_j \mid \Xc_{\ip{ns}+1:\ip{nt}} \right\} \\
&\leq \sup_{\substack{(s,t) \in \Lambda \\ t-s \geq \delta}} \sup_{u_j \in [\eps, 1-\eps]} \Pr \left\{\left| W_{j,u_j}^{\sss \Xc_{\ip{ns}+1:\ip{nt}}} - u_j \right| \geq 1- \eps/2 - u_j \mid \Xc_{\ip{ns}+1:\ip{nt}} \right\} \\
                         &\leq \sup_{\substack{(s,t) \in \Lambda \\ t-s \geq \delta}} \sup_{u_j \in [\eps, 1-\eps]} \frac{\Var\left(W_{j,u_j}^{\sss \Xc_{\ip{ns}+1:\ip{nt}}} \mid \Xc_{\ip{ns}+1:\ip{nt}} \right)}{(1- \eps/2 - u_j)^2} \\
                         &\leq \sup_{\substack{(s,t) \in \Lambda \\ t-s \geq \delta}} \sup_{u_j \in [\eps, 1-\eps]} \frac{a_{\ip{nt}-\ip{ns}}}{(1- \eps/2 - u_j)^2}\\
&\leq \sup_{\substack{(s,t) \in \Lambda \\ t-s \geq \delta}} a_{\ip{nt}-\ip{ns}} \sup_{u_j \in [\eps, 1-\eps]} \frac{1}{(1- \eps/2 - u_j)^2} \leq \frac{4}{\eps^2} \sup_{k \geq \ip{n\delta}-1} a_k \to 0.
\end{align*}}

\emph{Term $J_{j,n,\delta,\eps}$:} Let $\eta  > 0$ and let us show that $J_{j,n,\delta,\eps} \leq \eta$ for $n$ sufficiently large. For any $\rho \in (0,1)$, from the triangle inequality and the fact that $0 \leq \dot{C}_j \leq 1$, we have that $J_{j,n,\delta,\eps}$ is smaller than
\begin{align*}
  \nonumber
  & \sup_{\substack{(s,t) \in \Lambda \\ t-s \geq \delta}} \sup_{\substack{\bm u \in [0,1]^d \\ u_j \in [\eps,1-\eps]}} \left|\int_{\{\bm w \in [0,1]^d: |\bm w - \bm u|_\infty \leq \rho \}} \{ \dot C_j(\bm w) -  \dot C_j(\bm u) \} \dd \nu_{\bm u}^{\sss \Xc_{\ip{ns}+1:\ip{nt}}}(\bm w)\right|\\
  \nonumber
                           &+ \sup_{\substack{(s,t) \in \Lambda \\ t-s \geq \delta}} \sup_{\substack{\bm u \in [0,1]^d \\ u_j \in [\eps,1-\eps]}} \left|\int_{\{\bm w \in [0,1]^d: |\bm w - \bm u|_\infty > \rho \}} \{ \dot C_j(\bm w) -  \dot C_j(\bm u) \} \dd \nu_{\bm u}^{\sss \Xc_{\ip{ns}+1:\ip{nt}}}(\bm w)\right|\\
  &\leq \; J_{j,\eps,\rho} ' + J_{j,n,\delta,\rho} '',
\end{align*}
where
\begin{align*}
J_{j,\eps,\rho}' &=  \sup_{\substack{\bm u \in [0,1]^d \\ u_j \in [\eps,1-\eps]}} \sup_{\substack{\bm w \in [0,1]^d \\ |\bm w - \bm u|_\infty \leq \rho}} \left| \dot C_j(\bm w) -  \dot C_j(\bm u) \right|, \\
J_{j,n,\delta,\rho} '' &= \sup_{\substack{(s,t) \in \Lambda \\ t-s \geq \delta}} \sup_{\bm u \in [0,1]^d} \int_{[0,1]^d} \1\{|\bm w-\bm u|_\infty>\rho\} \dd \nu_{\bm u}^{\sss \Xc_{\ip{ns}+1:\ip{nt}}}(\bm w).
\end{align*}
From Condition~\ref{cond:pd}, $\dot C_j$ is uniformly continuous on the set $\{\bm u \in [0, 1]^d : u_j \in [\eps/2,1-\eps/2] \}$. We then choose $\rho = \rho(\eps,\eta)>0$ sufficiently small such that
\begin{align}
 \label{ineq:J:1}
J_{j,\eps,\rho}' = \sup_{\substack{\bm u \in [0, 1]^d \\ u_j \in [\eps,1-\eps]}} \sup_{\substack{\bm w \in [0,1]^d \\ |\bm w-\bm u|_{\infty}\leq \rho}} \left|\dot{C}_j(\bm w) - \dot{C}_j(\bm u)\right| \leq \frac{\eta}{2}.
\end{align}
As far as $J_{j,n,\delta,\rho}''$ is concerned, using Chebyshev's inequality and Condition~\ref{cond:var:W:weak}, for almost any sequence $\bm X_1, \bm X_2, \dots$, conditionally on $\bm X_1, \bm X_2, \dots$, we obtain that
\begin{align*}
J_{j,n,\delta,\rho} '' &= \sup_{\substack{(s,t) \in \Lambda \\ t-s \geq \delta}}\sup_{\bm u \in [0,1]^d} \nu_{\bm u}^{\sss \Xc_{\ip{ns}+1:\ip{nt}}}(\{\bm w \in [0,1]^d:|\bm u - \bm w|_{\infty} > \rho \})\\
                                 & = \sup_{\substack{(s,t) \in \Lambda \\ t-s \geq \delta}} \sup_{\bm u \in [0,1]^d} \Pr \left[\bigcup_{j=1}^d \left\{\left|W_{j,u_j}^{\sss \Xc_{\ip{ns}+1:\ip{nt}}} - u_j\right| > \rho\right\} \mid \Xc_{\ip{ns}+1:\ip{nt}} \right] \\
                                 & \leq \sum_{j=1}^d \sup_{\substack{(s,t) \in \Lambda \\ t-s \geq \delta}} \sup_{\bm u \in [0,1]^d} \Pr \left\{ \left|W_{j,u_j}^{\sss \Xc_{\ip{ns}+1:\ip{nt}}} - u_j\right| > \rho \mid \Xc_{\ip{ns}+1:\ip{nt}} \right \}\\
                                 & \leq \sum_{j=1}^d \sup_{\substack{(s,t) \in \Lambda \\ t-s \geq \delta}} \sup_{\bm u \in [0,1]^d} \frac{\Var\left(W_{j,u_j}^{\sss \Xc_{\ip{ns}+1:\ip{nt}}} \mid \Xc_{\ip{ns}+1:\ip{nt}} \right)}{\rho^2}\\
                                 & \leq \frac{d}{\rho^2}\sup_{\substack{(s,t) \in \Lambda \\ t-s \geq \delta}} a_{\ip{nt}-\ip{ns}} \leq \frac{d}{\rho^2} \sup_{k \geq \ip{n\delta}-1} a_k \to 0,
\end{align*}
which implies that, for $n$ sufficiently large, with probability~1, $J_{j,n,\delta,\rho}'' \leq \eta/2$. Using additionally~\eqref{ineq:J:1}, we obtain that $J_{j,n,\delta,\eps}$ converges almost surely to zero, which concludes the proof of~\eqref{eq:wc:Delta:nu:trunc}. The proof of the analogous result for $\dotu{C}_{j,k:l}^{\sss \nabla,\nu}$ in~\eqref{eq:pd:est:nabla:nu:trunc} is almost identical.
\end{proof}

\begin{proof}[\bf Proof of Proposition~\ref{prop:wc:Bern}]
Fix $j \in \{1,\dots,d\}$, $\delta \in (0,1)$ and $\eps \in (0,1/2)$. From~\eqref{eq:Bern:pd}, we have that, for any $(s,t) \in \Lambda$ and $\bm u \in [0,1]^d$, 
\begin{multline*}
\dot C_{j,\ip{ns}+1:\ip{nt},m_{\ip{nt}-\ip{ns}}}^{\sss \Bern}(\bm u) \\= \int_{[0,1]^d} \dot C_{j,\ip{ns}+1:\ip{nt},1/m_{\ip{nt}-\ip{ns}},0}^{\sss \nabla}(\bm w) \dd \tilde \mu_{j,m_{\ip{nt}-\ip{ns}},\bm u}(\bm w),
\end{multline*}
where $\dot{C}_{j,k:l,1/m,0}^{\sss \nabla}$ is defined in~\eqref{eq:pd:est:nabla} and, for any $m \geq 2$, $\tilde \mu_{j,m,\bm u}$ is the law of the random vector $(\tilde S_{m,1,u_1}/m,\dots, \tilde S_{m,d,u_d}/m)$ whose components are independent such that, for $i \in \{1,\dots,d\} \setminus \{j\}$, $\tilde S_{m,i,u_i}$ is Binomial$(m,u_i)$ while $\tilde S_{m,j,u_j}$ is Binomial$(m-1,u_j)$. It follows that, for any $(s,t) \in \Lambda$ and $\bm u \in [0,1]^d$, 
\begin{multline}
\label{eq:Bern:alt}
\dot C_{j,\ip{ns}+1:\ip{nt},m_{\ip{nt}-\ip{ns}}}^{\sss \Bern}(\bm u) \\= \int_{\Wc_{j,n,s,t}} \dot C_{j,\ip{ns}+1:\ip{nt},1/m_{\ip{nt}-\ip{ns}},0}^{\sss \nabla}(\bm w) \dd \tilde \mu_{j,m_{\ip{nt}-\ip{ns}},\bm u}(\bm w),
\end{multline}
where $\Wc_{j,n,s,t} =\{ \bm w \in [0,1]^d : w_j \leq 1 - 1/m_{\ip{nt}-\ip{ns}} \}$. For the sake of a more compact notation, from now on, we shall write $m_{s,t}$ for $m_{\ip{nt}-\ip{ns}}$, $(s,t) \in \Lambda$. From the triangle inequality, the left-hand side of~\eqref{eq:wc:Bern} is smaller than $I_{j,n,\delta,\eps} +J_{j,n,\delta,\eps}$, where
\begin{align*}
I_{j,n,\delta,\eps} &= \sup_{\substack{(s,t) \in \Lambda \\ t-s \geq \delta}} \sup_{\substack{\bm u \in [0,1]^d\\ u_j \in [\eps, 1-\eps]}}  \left| \int_{\Wc_{j,n,s,t}}\left\{ \dot{C}_{j,\ip{ns}+1:\ip{nt},1/m_{s,t},0}^{\sss \nabla}(\bm{w}) \right. \right. \\ & \qquad \qquad - \left. \left. \dot C_j(\bm w)\right\} \dd \tilde \mu_{j,m_{s,t},\bm u}(\bm w) \right|,\\
J_{j,n,\delta,\eps} &= \sup_{\substack{(s,t) \in \Lambda \\ t-s \geq \delta}} \sup_{\substack{\bm u \in [0,1]^d\\ u_j \in [\eps, 1-\eps]}}  \left|  \int_{[0,1]^d} \dot C_j(\bm w) \dd \tilde \mu_{j,m_{s,t},\bm u}(\bm w) -  \dot C_j(\bm u)\right|.
\end{align*}
For any $n \in \N$, $\bm x \in (\R^d)^n$ and $\bm u \in [0,1]^d$, let $\nu_{\bm u}^{\bm x} = \tilde \mu_{j, \ip{Ln^\theta}  \vee 2,\bm u}$. With this notation, Condition~\ref{cond:var:W:weak} holds for the considered smoothing distributions and it can be verified that $J_{j,n,\delta,\eps} = o_\Pr(1)$ by proceeding exactly as in the proof of Proposition~\ref{prop:wc:Delta:nu:trunc} for the analogous term. It thus remain to show that $I_{j,n,\delta,\eps} = o_\Pr(1)$.

From the triangle inequality, we have that $I_{j,n,\delta,\eps}$ is smaller than
{\small \begin{align*}
  &\sup_{\substack{(s,t) \in \Lambda \\ t-s \geq \delta}} \sup_{\substack{\bm u \in [0,1]^d\\ u_j \in [\eps, 1-\eps]}}  \left| \int_{\substack{\{\bm w \in \Wc_{j,n,s,t} : \\ w_j \in [\eps/2, 1-\eps/2]\}}} \left\{ \dot{C}_{j,\ip{ns}+1:\ip{nt},1/m_{s,t},0}^{\sss \nabla}(\bm{w}) -  \dot C_j(\bm w)\right\} \dd \tilde \mu_{j,m_{s,t},\bm u}(\bm w) \right|\\
  &+ \sup_{\substack{(s,t) \in \Lambda \\ t-s \geq \delta}} \sup_{\substack{\bm u \in [0,1]^d\\ u_j \in [\eps, 1-\eps]}}  \left| \int_{\substack{\{\bm w \in \Wc_{j,n,s,t} :\\ w_j < \eps/2\}}} \left\{ \dot{C}_{j,\ip{ns}+1:\ip{nt},1/m_{s,t},0}^{\sss \nabla}(\bm{w})  -  \dot C_j(\bm w)\right\} \dd \tilde \mu_{j,m_{s,t},\bm u}(\bm w) \right|\\
  &+ \sup_{\substack{(s,t) \in \Lambda \\ t-s \geq \delta}} \sup_{\substack{\bm u \in [0,1]^d\\ u_j \in [\eps, 1-\eps]}}  \left| \int_{\substack{\{\bm w \in \Wc_{j,n,s,t} :\\ w_j > 1-\eps/2\}}} \left\{ \dot{C}_{j,\ip{ns}+1:\ip{nt},1/m_{s,t},0}^{\sss \nabla}(\bm{w}) - \dot C_j(\bm w)\right\} \dd \tilde \mu_{j,m_{s,t},\bm u}(\bm w) \right| \\
  \leq & \; I_{j,n,\delta,\eps}' + M_n I_{j,n,\delta,\eps}'' + M_n I_{j,n,\delta,\eps}''',
\end{align*}}
where
\begin{align*}
I_{j,n,\delta,\eps}' =&  \sup_{\substack{(s,t) \in \Lambda \\ t-s \geq \delta}} \sup_{\substack{\bm w \in [0,1]^d\\ w_j \in [\eps/2, 1-\eps/2]}}  \left| \dot{C}_{j,\ip{ns}+1:\ip{nt},1/m_{s,t},0}^{\sss \nabla}(\bm w) - \dot C_j(\bm w) \right|,\\
I_{j,n,\delta,\eps}'' =& \sup_{\substack{(s,t) \in \Lambda \\ t-s \geq \delta}} \sup_{\substack{\bm u \in [0,1]^d\\ u_j \in [\eps, 1-\eps]}} \tilde \mu_{j,m_{s,t},\bm u}\left\{\bm w \in [0,1]^d:w_j < \eps/2  \right\},\\
I_{j,n,\delta,\eps}''' =& \sup_{\substack{(s,t) \in \Lambda \\ t-s \geq \delta}} \sup_{\substack{\bm u \in [0,1]^d\\ u_j \in [\eps, 1-\eps]}} \tilde \mu_{j,m_{s,t},\bm u}\left\{\bm w \in [0,1]^d:w_j > 1 - \eps/2  \right\}, \\
M_n =& \sup_{(s,t) \in \Lambda} \sup_{\bm w \in \Wc_{j,n,s,t}}  \left| \dot C_{j,\ip{ns}+1:\ip{nt},1/m_{s,t},0}^{\sss \nabla}(\bm w) \right|.
\end{align*}
Since the conditions of the proposition imply that Condition~\ref{cond:band} holds with $h(\bm x) = 1/(\ip{Ln^\theta} \vee 2)$ and $h'(\bm x) = 0$ for all $n \in \N$ and $\bm x \in (\R^d)^n$, we have that $I_{j,n,\delta,\eps}' = o_\Pr(1)$ as a consequence of Corollary~\ref{cor:wc:Delta}. Also, given that Condition~\ref{cond:var:W:weak} holds for the considered smoothing distributions, it can be verified that $I_{j,n,\delta,\eps}''$ and $I_{j,n,\delta,\eps}'''$ converge almost surely to zero by proceeding exactly as in the proof of Proposition~\ref{prop:wc:Delta:nu:trunc} for the analogous terms. To complete the proof of~\eqref{eq:wc:Bern}, it suffices to show that, there exists a constant $\zeta > 0$ such that, for any $n \in \N$, $M_n < \zeta$ almost surely.

Fix $n \in \N$. From the adopted conventions, we have that $\dot C_{j,\ip{ns}+1:\ip{nt},1/m_{s,t},0}^{\sss \nabla} = 0$ for all $(s,t) \in \Lambda$ such that $\ip{ns} = \ip{nt}$. Fix $(s,t) \in \Lambda$ such that $\ip{ns} < \ip{nt}$ and let $p = \ip{nt} - \ip{ns}$. The empirical copula $C_{\ip{ns}+1:\ip{nt}}$, generically defined in~\eqref{eq:C:kl}, is a multivariate d.f.\ whose $d$ univariate margins, under Condition~\ref{cond:no:ties}, are all equal to $G_{\ip{ns}+1:\ip{nt}}$, where $G_{\ip{ns}+1:\ip{nt}}(u) = \ip{p u} / p$, $u \in [0,1]$. As a consequence of a well-known property of multivariate d.f.s \citep[see, e.g.,][Lemma~1.2.14]{DurSem15}, we have that
$$
\left| C_{\ip{ns}+1:\ip{nt}}(\bm u) - C_{\ip{ns}+1:\ip{nt}}(\bm v) \right| \leq \sum_{j=1}^d \left| G_{{\ip{ns}+1:\ip{nt}}}(u_j) - G_{{\ip{ns}+1:\ip{nt}}}(v_j) \right|
$$
for all $\bm u, \bm v \in [0,1]^d$. We then obtain that, for any $\bm u \in \Wc_{j,n,s,t}$,
\begin{multline*}
\left| C_{\ip{ns}+1:\ip{nt}}(\bm u + \bm e_j / m_{s,t}) - C_{\ip{ns}+1:\ip{nt}}(\bm u) \right| \\ \leq \left| G_{{\ip{ns}+1:\ip{nt}}}(u_j+1/m_{s,t}) - G_{{\ip{ns}+1:\ip{nt}}}(u_j)  \right|,
\end{multline*}
which implies that
\begin{align*}
  \left| \dot C_{j,\ip{ns}+1:\ip{nt}}(\bm u) \right| &\leq \frac{\left| G_{{\ip{ns}+1:\ip{nt}}}(u_j+1/m_{s,t}) - G_{{\ip{ns}+1:\ip{nt}}}(u_j) \right|}{1/m_{s,t}}  \\
                                                     &=  m_{s,t} \left\{ \frac{\ip{p(u_j + 1/m_{s,t})}}{p} - \frac{\ip{pu_j}}{p} \right\} \\
                                                     &\leq m_{s,t} \left\{ \frac{p(u_j+1/m_{s,t})}{p} - \frac{pu_j - 1}{p} \right\} \\
                                                     &\leq  m_{s,t} \left( \frac{1}{m_{s,t}} + \frac{1}{p} \right) \leq 1 + \frac{m_{s,t}}{p} = 1 + \frac{\ip{Lp^\theta} \vee 2}{p} \\
                                                     &\leq 1 + Lp^{\theta-1} \vee (2/p) \leq 1 + L \vee 2,
\end{align*}
which completes the proof of~\eqref{eq:wc:Bern}. The fact that~\eqref{eq:bounded:Bern} holds is finally an immediate consequence of the previous centered display and~\eqref{eq:Bern:alt}.
\end{proof}

\end{appendix}

\bibliographystyle{imsart-nameyear}
\bibliography{biblio}

\end{document}